\documentclass[pdflatex,sn-mathphys-num]{sn-jnl}

\usepackage{graphicx}%
\usepackage{multirow}%
\usepackage{amsmath,amssymb,amsfonts}%
\usepackage{amsthm}%
\usepackage{mathrsfs}%
\usepackage[title]{appendix}%
\usepackage{xcolor}%
\usepackage{textcomp}%
\usepackage{manyfoot}%
\usepackage{booktabs}%
\usepackage{algorithm}%
\usepackage{algorithmicx}%
\usepackage{algpseudocode}%
\usepackage{listings}%
\usepackage{mathtools}
\usepackage{epsfig,color,scalerel,subfigure}
\usepackage[super]{nth}
\usepackage{xparse}
\usepackage{siunitx}
\usepackage{afterpage}
\usepackage{enumitem}
\usepackage{comment}
\usepackage{tikz}
\newcommand{\bn}{\boldsymbol{n}}

\newcommand{\bu}{\boldsymbol{u}}
\newcommand{\bv}{\boldsymbol{v}}
\newcommand{\bw}{\boldsymbol{w}}

\newcommand{\fb}{\boldsymbol{f}}
\newcommand{\bg}{\boldsymbol{g}}
\newcommand{\bt}{\boldsymbol{t}}
\newcommand{\bx}{\boldsymbol{x}}
\newcommand{\by}{\boldsymbol{y}}
\newcommand{\bz}{\boldsymbol{z}}
\newcommand\beps{\boldsymbol{\varepsilon}}
\newcommand\bsigma{\boldsymbol{\sigma}}
\newcommand\bzeta{\boldsymbol{\zeta}}
\newcommand\bzero{\boldsymbol{0}}
\newcommand\bxi{\boldsymbol{\xi}}
\newcommand\btau{\boldsymbol{\tau}}
\newcommand\bnabla{\boldsymbol{\nabla}}
\newcommand\bPi{\boldsymbol{\Pi}}
\newcommand\bphi{\boldsymbol{\phi}}
\newcommand\bchi{\boldsymbol{\chi}}
\newcommand\bpsi{\boldsymbol{\psi}}
\newcommand\bh{\boldsymbol{h}}

\newcommand{\bH}{\mathbf{H}}

\newcommand{\bL}{\mathbf{L}}

\newcommand{\bV}{\mathbf{V}}
\newcommand{\bQ}{\mathbf{Q}}

\newcommand{\bbH}{\mathbb{H}}
\newcommand{\bbM}{\mathbb{M}}
\newcommand{\bbI}{\mathbb{I}}
\newcommand{\bbR}{\mathbb{R}}

\newcommand\bdiv{\mathop{\mathbf{div}}\nolimits}
\newcommand\vdiv{\mathop{\mathrm{div}}\nolimits}
\newcommand\brot{\mathop{\mathbf{rot}}\nolimits}
\newcommand\vrot{\mathop{\mathrm{rot}}\nolimits}
\newcommand\bcurl{\mathop{\mathbf{curl}}\nolimits}

\newcommand\tr{\mathop{\mathrm{tr}}\nolimits}

\newcommand{\lJump}{[\![}
\newcommand{\rJump}{]\!]}

\DeclarePairedDelimiter\jump{\lJump}{\rJump}

\DeclarePairedDelimiter\norm{\lVert}{\rVert}

\newtagform{scroman}[\renewcommand{\theequation}{\scshape\roman{equation}}][]

\def\Xint#1{\mathchoice
	{\XXint\displaystyle\textstyle{#1}}%
	{\XXint\textstyle\scriptstyle{#1}}%
	{\XXint\scriptstyle\scriptscriptstyle{#1}}%
	{\XXint\scriptscriptstyle\scriptscriptstyle{#1}}%
	\!\int}
\def\XXint#1#2#3{{\setbox0=\hbox{$#1{#2#3}{\int}$ }
		\vcenter{\hbox{$#2#3$ }}\kern-.6\wd0}}

\def\dashint{\Xint-}

\numberwithin{figure}{section}
\numberwithin{equation}{section}

\allowdisplaybreaks

\theoremstyle{thmstyleone}%
\newtheorem{theorem}{Theorem}[section]
\newtheorem{proposition}[theorem]{Proposition}%
\newtheorem{lemma}[theorem]{Lemma}

\theoremstyle{thmstyletwo}%
\newtheorem{remark}{Remark}%

\theoremstyle{thmstylethree}%

\begin{document}

\title[Article Title]{A posteriori error analysis of a 
virtual element method for stress-assisted diffusion problems}


\author[1]{\fnm{Franco} \sur{Dassi}}\email{Franco.Dassi@unimib.it}

\author[2]{\fnm{Rekha} \sur{Khot}}\email{rekhakhot@iitpkd.ac.in}

\author*[3]{\fnm{Andr\'es E.} \sur{Rubiano}}\email{Andres.RubianoMartinez@monash.edu}

\author[3,4]{\fnm{Ricardo} \sur{Ruiz-Baier}}\email{Ricardo.RuizBaier@monash.edu}

\affil[1]{\orgdiv{Dipartimento di Matematica e Applicazioni}, \orgname{Universita` degli Studi di Milano Bicocca}, \orgaddress{\street{Via Roberto Cozzi 55 - 20125}, \city{Milano}, \postcode{20125}, \country{Italy}}}

\affil[2]{\orgdiv{Department of Mathematics}, \orgname{Indian Institute of Technology Palakkad}, \orgaddress{\street{Kanjikode}, \postcode{678623}, \city{Kerala}, \country{India}}}

\affil[3]{\orgdiv{School of Mathematics}, \orgname{Monash University}, \orgaddress{\street{9 Rainforest Walk}, \city{Melbourne}, \postcode{3800}, \state{Victoria}, \country{Australia}}}

\affil[4]{\orgname{Universidad Adventista de Chile}, \orgaddress{\street{Casilla 7-D}, \city{Chill\'an}, \country{Chile}}}


\abstract{We develop and analyse residual-based a posteriori error estimates for the virtual element discretisation of a nonlinear stress-assisted diffusion problem in two and three dimensions. The model problem involves a two-way coupling between elasticity and diffusion equations in perturbed saddle-point form. A robust global inf-sup condition and Helmholtz decomposition for $\mathbf{H}(\mathrm{div},\Omega)$ lead to a reliable and efficient error estimator based on appropriately weighted norms that ensure parameter robustness. The a posteriori error analysis uses novel quasi-interpolation operators for Stokes and edge virtual element spaces, and we include the proofs of such operators with estimates in 3D for completeness. Finally, we present numerical experiments in both 2D and 3D to demonstrate the optimal performance of the proposed error estimator.}




\maketitle

\section{Introduction}\label{sec:introduction}
The physical interaction between solute diffusion, driven by chemical reactions, in a deformable medium was first formally addressed in \cite{unger1983theory, wilson1982theory} known as the stress-assisted diffusion model. This phenomenon appears in a variety of scientific, medical, and engineering applications such as the lithiation of ion batteries \cite{Taralov2015}, cardiac tissue electromechanics \cite{Loppini2018-vk}, semiconductor fabrication \cite{Shaw2007}, biomechanics of brain tissue \cite{PREVOST201183}, among others. These applications operate across multiple scales, 
and motivating the need of robust models. 
The mathematical analysis and discretisation of this model have been conducted using mixed-primal and mixed-mixed  finite element methods (FEM)  in Hilbert spaces \cite{gatica18}, pseudo-stress-based formulations using Banach spaces \cite{gatica22}, and well-posedness results for a primal formulation  \cite{lewicka2016local, malaeke2023mathematical}. In the virtual element method (VEM) framework, this model was first studied in \cite{khot2024}, which introduced a robust mixed-mixed formulation with parameter-weighted norms ensuring the unique solvability of two uncoupled systems across different parameter scales (i.e. robustness is achieved), along with a fixed-point argument establishing well-posedness for the fully coupled equations. 
Our goal here is to derive an adaptive mesh refinement strategy driven by a posteriori error estimators for this type of formulations, and extend reliability and efficiency results to the parameter-robust setting.

Adaptive schemes driven by a posteriori error estimators allow optimal convergence recovery in cases such as singular solutions, rough data, and complex geometries (e.g., non-convex domains and sharp corners). The key advantage of the VEM framework in adaptive algorithms is its natural handling of hanging nodes during mesh refinement. However, several persistent challenges remain, 
such as developing open-source implementations for polytopal conforming mesh refinement \cite{antonietti2022}.
The FEM literature on a posteriori error analysis of this problem can be found in \cite{gatica2022posteriori}. On the other hand, a posteriori error estimates for VEM were first introduced in \cite{cangiani2017posteriori}, and the extension to displacement-based deformation models and reaction-diffusion equations using VEM can be found in \cite{wang20,munar2024}. 

To the best of the authors' knowledge, the present paper is the first one addressing robust a posteriori error estimators for VEM applied to stress-assisted diffusion problems both in 2D and 3D. We first establish residual-based error estimators in 2D and prove the reliability and efficiency under a small data assumption due to the Banach fixed-point argument. Here, we use classical tools such as quasi-interpolation operators, polynomial projections, a Helmholtz decomposition, and the definition of bubble functions to achieve our purpose. Further, we provide ``building blocks" for the 3D case. Specifically, we prove novel quasi-interpolation operator for both Stokes- and edge-like virtual element spaces in 3D, together with a stable Helmholtz decomposition for reaction-diffusion equations (in mixed VEM). The associated estimators are implemented in the \texttt{VEM++} library \cite{dassi2023vem++} (available upon request) with an extensible design, enabling the incorporation of various a posteriori error estimators for the virtual element (VE) spaces available in the library.

\medskip
\noindent\textbf{Plan of the paper.} The contents of this paper have been organised as follows. The remainder of this section contains preliminary notational conventions and  useful functional spaces. Section~\ref{sec:problem} presents the coupled stress-assisted diffusion model, weak formulation consisting of two coupled perturbed saddle-point problems, unique solvability result, and robust global inf-sup result. 
The mesh assumptions and VE spaces for the coupled problem along with their properties are provided in Section~\ref{sec:vem} for the 2D case. Section~\ref{sec:aposteriori} is devoted to deriving a reliable and efficient error estimator. Next, we extend the VE spaces to the 3D case in Section~\ref{sec:3d}, and construct two novel quasi-interpolation operators for Stokes and edge spaces. 
Finally, our theoretical results are illustrated via numerical examples in Section~\ref{sec:numerical-examples}.

\medskip 
\noindent\textbf{Recurrent notation.}
Let $D$ be a domain of $\bbR^d$ ($d=2,3$), with boundary $\partial D$. 
Given a tensor function $\bsigma:D\to \bbR^{d\times d}$, a vector field $\bu:D\to \bbR^d$ and a scalar field ${p}:D\to \bbR$ we set the tensor divergence $\bdiv \bsigma:D \to \bbR^d$, the  vector gradient $\bnabla \bu:D \to \bbR^{d\times d}$, the symmetric gradient $\beps(\bu) : D \to \bbR^{d\times d}$, the vector divergence $\vdiv \bu:D \to \bbR$, the 3D vector rotational $\bcurl \bu: D\to \bbR^3$, the 2D scalar rotational $\vrot \bu:D \to \bbR$, and the 2D vector rotational $\brot {p}:D \to \bbR^2$, as $(\bdiv \bsigma)_i := \sum_j   \partial_j \bsigma_{ij}$, $(\bnabla \bu)_{ij} := \partial_j \bu_i$, $\beps(\bu) := \frac{1}{2}\left[\bnabla\bu+(\bnabla\bu)^{\tt t}\right]$, $\vdiv \bu := \sum_i \partial_i \bu_i$, $\bcurl \bu := (\partial_2 \bu_3- \partial_3 \bu_2, \partial_3 \bu_1 - \partial_1 \bu_3, \partial_1 \bu_2 - \partial_2 \bu_1)$, $\vrot \bu := \partial_1 \bu_2 -\partial_2 \bu_1$, and $\brot {p} := (\partial_2 {p}, -\partial_1 {p})^{\tt t}$, respectively.

The component-wise inner product for vectors $\bu,\bv \in \bbR^d$ and matrices $\bsigma, \,\btau \in\bbR^{d\times d}$ are defined by $\bu:\bv:= \sum_{i}\bu_{i}\bv_{i}$, and $\bsigma:\btau:= \sum_{i,j}\bsigma_{ij}\btau_{ij}$. 
For $s \geq 0$, we denote the  Sobolev space of scalar functions with domain $D$  as $\mathrm{H}^s(D)$, and their vector and tensor counterparts as $\bH^s(D)$ and $\bbH^s(D)$, respectively. 
The norm in $\mathrm{H}^s(D)$ is denoted $\norm{\cdot}_{s,D}$ and the corresponding semi-norm $|\cdot|_{s,D}$. We also use the convention  $\mathrm{H}^0(D):=\mathrm{L}^2(D)$ and let $(\cdot, \cdot)_{\mathrm{D}}$ be the inner product for in $\mathrm{L}^2(D)$ (similarly for the vector and tensor counterparts). The space $\text{H}^{\frac{1}{2}}(\partial D)$ contains traces of functions of $\text{H}^1(D)$, $\text{H}^{-\frac{1}{2}}(\partial D)$ denotes its dual, and $\langle \cdot, \cdot \rangle_{\partial D}$ stands for the duality
pairing between them. The Hilbert space $\bH(\vdiv, D)$ of vectors in $\bL^2(D)$ with divergence in $\mathrm{L}^2(D)$ equipped with the norm $\norm{\cdot}^2_{\vdiv, D}:=\norm{\cdot}_{0,D}^2+\norm{\vdiv \cdot}^2_{0, D}$. Similarly, the Hilbert spaces $\bH(\vrot, D)$ and $\bH(\bcurl, D)$ are equipped with the norms $\norm{\cdot}^2_{\vrot,D}:=\norm{\cdot}_{0,D}^2+\norm{\vrot\cdot}^2_{0,D}$ and $\norm{\cdot}^2_{\bcurl,D}:=\norm{\cdot}_{0,D}^2+\norm{\bcurl \cdot}^2_{0,D}$. The outward unit normal vector and unit tangential vector to $\partial D$ are denoted respectively by $\bn$ and $\bt$.
 
Throughout this paper, we shall use the letter $C$ to denote a generic positive constant independent of the mesh size $h$ and physical constants, which might stand for different values at its different occurrences. Moreover, given any positive expressions $X$ and $Y$, the notation $X \,\lesssim\, Y$  means that $X \,\le\, C\, Y$ (and similarly for $X\gtrsim Y$).

\section{The stress-assisted diffusion problem}\label{sec:problem}
This section recalls from \cite{khot2024} 
the weak formulation and its well-posedness analysis based on the Babuška--Brezzi--Braess theory and a fixed-point argument. 

\medskip 
\noindent\textbf{Model problem.}
Let $\Omega$ be a polytopal (polygonal in 2D and polyhedral in 3D) bounded domain with boundary $\Gamma = \Gamma_{\mathrm{D}} \cup \Gamma_{\mathrm{N}}$, such that $\Gamma_{\mathrm{D}}\cap \Gamma_{\mathrm{N}} = \emptyset$, and consider the following coupled 
PDE with mixed boundary conditions 
\begin{subequations}\label{eq:mixed-formulation}
\begin{alignat}{2}
-\bdiv(2\mu \beps(\bu) -{p}\bbI) = \fb, &\quad \text{in $\Omega$},\quad \bu=\mathbf{0}, &\quad \text{on $\Gamma_{\mathrm{D}}$},\label{eq:linear-momentum}\\
{p} = -\lambda \vdiv\bu\, +\ell(\varphi), &\quad \text{in $\Omega$},\quad (2\mu \beps(\bu) -{p}\bbI)\bn = \bzero, &\quad \text{on $\Gamma_{\mathrm{N}}$},\label{eq:hooke-law}\\
\bzeta = \bbM(\beps(\bu),{p}) \nabla \varphi, &\quad \text{in $\Omega$},\quad \varphi=\varphi_{\mathrm{D}}, &\quad \text{on $\Gamma_{\mathrm{D}}$},\label{eq:diffusive-flux}\\
\theta \varphi 
- \vdiv(\bzeta) = g, &\quad \text{in $\Omega$},\quad \bzeta \cdot \bn = 0, &\quad \text{on $\Gamma_{\mathrm{N}}$}\label{eq:reaction-diffusion}.
\end{alignat}
\end{subequations}
The coupling 
uses an active stress approach defining the total Cauchy stress as $\bsigma 
= 2\mu \beps(\bu) -{p}\bbI$, where 
$\bu$ is the displacement vector, $\beps(\bu)$ is the tensor of infinitesimal strains, ${p}$ denotes a Herrmann-type pressure, $\bbI$ denotes the identity tensor in $\bbR^{d\times d}$, $\varphi$ is the solute's concentration, $\mu$ and $\lambda$ are the Lam\'e parameters  of the solid, $\fb$ is a vector of external body loads, $\ell$  modulates  the (isotropic) active stress. 
On the other hand, 
$\bzeta$ is the diffusive flux $\bbM(\beps(\bu),{p})$ is the stress-assisted diffusion coefficient (assumed uniformly bounded away from zero), $\theta$ is a model parameter that lies in $[0,1]$, and $g$ is a given net volumetric source of solute.

\medskip 
\noindent\textbf{Assumptions on the nonlinear terms.}\label{nonlinear-terms} 
We suppose that $\ell: \mathrm{L}^2(\Omega)\to \mathrm{L}^2(\Omega)$   satisfies 
$
\norm{\ell(\vartheta)}_{0,\Omega} \lesssim \norm{\vartheta}_{0,\Omega}$ for all $\vartheta \in \mathrm{L}^2(\Omega)$. 
We also assume that $\bbM(\cdot,\cdot)$ is   invertible, symmetric, positive semi-definite and uniformly bounded in $\mathbb{L}^\infty(\Omega)$ (likewise for $\bbM^{-1}(\cdot,\cdot)$). In addition, for all $\bw\in \bH^1(\Omega), r\in \mathrm{L}^2(\Omega) \text{ and }  \bx,\by \in \mathbb{R}^d$,  there exists $M\geq 1$ such that 
${M}^{-1} \bx\cdot \by \leq \bx\cdot[\bbM^{-1}(\beps(\bw),r)\by]$ and  
$\by\cdot[\bbM^{-1}(\beps(\bw),r)\bx]  \leq M \bx\cdot\by$.
Finally, we assume that $\bbM^{-1}(\cdot,\cdot)$ and $\ell(\cdot)$ are Lipschitz continuous with Lipschitz constants $L_{\bbM}$ and $L_{\ell}$. Examples of these terms can be found in   \cite{cherubini17} for the stress-assisted diffusion  and \cite{murray2003mathematical,Taralov2015} for the active stress.

\subsection{Weak formulation}\label{sec:weak-formulation}
In view of the boundary conditions, we define the  Hilbert spaces 
\begin{gather*}
\bH^1_{\mathrm{D}}(\Omega):=\{\bv \in \bH^1(\Omega): \bv = \bzero\quad \text{on }\Gamma_{\mathrm{D}}\}, \quad \bH_{\mathrm{N}}(\vdiv,\Omega):=\{\bxi \in \bH(\vdiv,\Omega): \bxi\cdot\bn = 0 \quad \text{on }\Gamma_{\mathrm{N}}\}.  
\end{gather*}
With them, the weak form consists in, given $\fb\in \bL^2(\Omega)$, $g\in \mathrm{L}^2(\Omega)$, 
and $\varphi_{\mathrm{D}}\in \text{H}^{\frac{1}{2}}(\Gamma_{\mathrm{D}})$, finding 
$(\bu,{p},\bzeta,\varphi) \in \bH_{\mathrm{D}}^1(\Omega)\times \mathrm{L}^2(\Omega) \times \bH_{\mathrm{N}}(\vdiv,\Omega) \times \mathrm{L}^2(\Omega)$ such that 
\begin{subequations}\label{eq:weak}
\begin{align}
 a_1(\bu,\bv) + b_1(\bv,{p})  = F_1(\bv), \qquad &\forall \bv \in \bH_{\mathrm{D}}^1(\Omega),\label{weak-1}\\
  b_1(\bu,{q}) - \frac{1}{\lambda}c_1({p},{q})  = G_1^\varphi({q}), \qquad &\forall {q} \in \mathrm{L}^2(\Omega),\\
a_2^{\bu,{p}}(\bzeta,\bxi) + b_2(\bxi,\varphi)  = F_2(\bxi), \qquad &\forall \bxi \in \bH_{\mathrm{N}}(\vdiv,\Omega),\label{weak-3}\\
  b_2(\bzeta,\psi) - \theta c_2(\varphi,\psi) = G_2(\psi), \qquad &\forall \psi \in \mathrm{L}^2(\Omega).
\end{align}
\end{subequations} 
The bilinear forms $a_1: \bH_{\mathrm{D}}^1(\Omega)\times \bH_{\mathrm{D}}^1(\Omega)\to \bbR$, $b_1:\bH_{\mathrm{D}}^1(\Omega)\times \mathrm{L}^2(\Omega)\to \bbR$, 
$c_1:\mathrm{L}^2(\Omega)\times \mathrm{L}^2(\Omega)\to \bbR$, $a_2^{\bu,p}:\bH_{\mathrm{N}}(\vdiv,\Omega)\times \bH_{\mathrm{N}}(\vdiv,\Omega)\to \bbR$, $b_2:\bH_{\mathrm{N}}(\vdiv,\Omega)\times \mathrm{L}^2(\Omega)\to \bbR$, $c_2:\mathrm{L}^2(\Omega)\times \mathrm{L}^2(\Omega)\to \bbR$, and linear functionals $F_1 : \bH_{\mathrm{D}}^1(\Omega)\to \bbR$, $G_1^\varphi: \mathrm{L}^2(\Omega)\to \bbR$,  $F_2:\bH_{\mathrm{N}}(\vdiv,\Omega)\to \bbR$, $G_2:\mathrm{L}^2(\Omega)\to\bbR$, are  
\begin{gather*}
    a_1(\bu,\bv) : = 2\mu \int_\Omega \beps(\bu):\beps(\bv), \quad 
    b_1(\bv,{q}) :=  - \int_\Omega {q}\vdiv\bv,\quad
    c_1({p},{q}) : = \int_\Omega {p}{q}, \\ 
    F_1(\bv):= \int_\Omega \fb\cdot\bv, \quad
    G_1^\varphi({q}): = -\frac{1}{\lambda}\int_\Omega  \ell(\varphi){q},
    \\ 
    a_2^{\bu,p}(\bzeta,\bxi): =  \int_\Omega \bbM^{-1}(\beps(\bu),p) \bzeta \cdot \bxi, \quad
    b_2(\bxi,\psi):= \int_\Omega \psi \vdiv \bxi,\quad 
    c_2(\varphi,\psi):= \int_\Omega \varphi\psi,\\
    F_2(\bxi): = \langle \varphi_{\mathrm{D}}, \bxi\cdot\bn\rangle_{\Gamma_{\mathrm{D}}}, \quad G_2(\psi):= - \int_\Omega g\psi.
\end{gather*}

\subsection{Parameter-dependent norms and robust solvability}\label{sec:robust-setting}
Given $K\subseteq \Omega$, let $\bV_1(K) := \bH^1(K)$, $\mathrm{Q}_1(K) := \mathrm{L}^2(K)$, $\bV_2(K) := \bH(\vdiv,K)$ and $\mathrm{Q}_2 := \mathrm{L}^2(K)$  be the displacement, Herrmann  pressure, flux, and concentration spaces on $K$, respectively, equipped with the following weighted norms and semi-norms
\begin{gather*}
    \norm{(\bu,{p})}_{\bV_1(K)\times \mathrm{Q}_1(K)}^2 := \norm{\bu}_{\bV_1(K)}^2 + \norm{{p}}_{\mathrm{Q}_1(K)}^2, \; \norm{\bu}_{\bV_1(K)}^2 := 2\mu\norm{\beps(\bu)}_{0,K}^2, \\ 
    \norm{{p}}_{\mathrm{Q}_1(K)}^2 := \left(\frac{1}{2\mu} +  \frac{1}{\lambda}\right)\norm{{p}}_{0,K}^2,\\
    |\bu|_{\bV_1(K)}^2 := 2\mu|\bu|_{1,K}^2, \; |{p}|_{\mathrm{Q}_1(K)}^2 := \frac{1}{2\mu}\norm{p}_{0,K}^2,\\
    \norm{(\bzeta,\varphi)}_{\bV_2(K)\times \mathrm{Q}_2(K)}^2 := \norm{\bzeta}_{\bV_2(K)}^2 + \norm{\varphi}_{\mathrm{Q}_2(K)}^2, \; \norm{\bzeta}_{\bV_2(K)}^2 := \norm{\bzeta}_{\bbM,K}^2 + M\norm{\vdiv \bzeta}_{0,K}^2, \\ 
    \norm{\bzeta}_{\bbM,K}^2 := 
    \int_K \bbM^{-1}(\beps(\bu),p) \bzeta \cdot \bzeta, \;
    \norm{\varphi}_{\mathrm{Q}_2(K)}^2 :=  \left(\frac{1}{M} + \theta \right)\norm{\varphi}_{0,K}^2,\\
    |\bzeta|_{\bV_2(K)}^2 := M|\bzeta|_{1,K}^2, \; |\varphi|_{\mathrm{Q}_2(K)}^2 := \frac{1}{M}\norm{\varphi}_{0,K}^2.
\end{gather*}
Note that the definition of $|\bzeta|_{\bV_2(K)}^2$ requires $\bzeta \in \bV_2(K)\cap \bH^1(K)$. For simplicity, we denote the global spaces (i.e., $K=\Omega$)  as $\bV_1:= \bH_{\mathrm{D}}^{1}(\Omega), \mathrm{Q}_1 = \mathrm{Q}_2  := \mathrm{L}^2(\Omega), \bV_2 := \bH_{\mathrm{N}}(\vdiv,\Omega)$ while imposing the boundary conditions. 

Theorem~\ref{well-posedness} below establishes the well-posedness of the fully coupled system \eqref{eq:weak} under the small data assumption 
\begin{equation}\label{small_data_cont}
    C_1 L_\ell \max\left\{\frac{1}{\sqrt{2\mu}},\sqrt{2\mu} \right\}M^{2}C_2^2 L_{\bbM}\left(\norm{\varphi_{\mathrm{D}}}_{\frac{1}{2},\Gamma_{\mathrm{D}}} + \norm{g}_{0,\Omega}\right) < 1.
\end{equation}
We recall that this condition arises only from the fixed-point treatment of the nonlinearity; the method’s robustness is guaranteed by the weighted norm spaces (see \cite{braess96penalty}) and the mixed formulation introduced before, for a detailed proof we refer to \cite{khot2024}. In addition, Theorem~\ref{th:global-inf-sup} provides a robust global inf-sup condition. The proof follows from the Brezzi--Braess conditions in \cite[Theorem 2.1]{boon21}.
\begin{theorem}\label{well-posedness}
Suppose that $1\leq \lambda$, $0<\mu$, $\theta \leq \frac{1}{M}$, and that the small data assumption \eqref{small_data_cont} holds.   
Then, there exists an unique solution $(\bu,{p},\bzeta,\varphi)\in \bV_1\times \mathrm{Q}_1 \times \bV_2 \times \mathrm{Q}_2$ to \eqref{eq:weak} such that 
    \begin{align*}
        \norm{(\bu,{p})}_{\bV_1\times \mathrm{Q}_1} &\leq C_1 \left( \norm{F_1}_{\bV'_1} + \norm{G^\varphi_1}_{Q'_1} \right),
        \\
        \norm{(\bzeta,\varphi)}_{\bV_2\times \mathrm{Q}_2} &\leq C_2 \left( \norm{F_2}_{\bV'_2} +\norm{G_2}_{Q'_2} \right),
    \end{align*}
where the corresponding constants $C_1$ and $C_2$ do not depend on the physical parameters. 
\end{theorem}
\begin{theorem} \label{th:global-inf-sup}
Let $(V,\norm{\cdot}_{V})$ and $(Q_b,\norm{\cdot}_{Q_b})$ be Hilbert spaces, 
let $Q$ be a dense (with respect to 
$\norm{\cdot}_{Q_b}$) linear subspace of $Q_b$ and three bilinear forms $a(\cdot,\cdot)$ on $V\times V$ (continuous, symmetric and positive semi-definite), $b(\cdot,\cdot)$ on $V\times Q_b$ (continuous), and $c(\cdot,\cdot)$ on $Q\times Q$ (symmetric and positive semi-definite); defining the linear operators $A: V \rightarrow V'$, $B : V \rightarrow Q'_b$ and $C:Q \rightarrow Q'$, respectively. For $t\in[0,1]$, consider the $t$-dependent energy norm  
$$\norm{(v,q)}_{V\times Q}^2:=\norm{v}_{V}^2+\norm{q}_{Q}^2=\norm{v}_{V}^2+\norm{q}_{Q_b}^2+t^2|q|_c^2.$$ 
Assume that $Q$ is complete with respect to the norm $\norm{\cdot}_{Q}^2:=\norm{\cdot}_{Q_b}^2+t^2|\cdot|_c^2$, where $|\cdot|_c^2 := c(\cdot,\cdot)$ is a semi-norm in $Q$. Suppose further that there exist positive constants $\alpha,\beta,\gamma$ (independent of the model parameters) such that 
\begin{subequations}\label{brezzi-braess-conditions}
\begin{align}\label{brezzi-condition-1}
\alpha \norm{\hat{v}}_V^2 \leq a(\hat{v},\hat{v}), \qquad &\forall \hat{v} \in \mathrm{Ker}(B),\\
\label{brezzi-condition-2}
\beta \norm{q}_{Q_b} \leq \sup_{v\in V} \frac{b(v,q)}{{\norm{v}}_V}, \qquad &\forall q \in Q_b,\\
\label{braess-condition}
\gamma \norm{u}_{V} \leq \sup_{(v,q)\in V\times Q} \frac{a(u,v)+b(u,q)}{\norm{(v,q)}_{V\times Q}}, \qquad &\forall u \in V.
\end{align}
\end{subequations}
Then, the multilinear form $\mathcal{A}(u,p;v,q):=a(u,v)+b(v,p)+b(u,q)-t^2c(p,q)$ satisfies the following condition 
\begin{align*}
    \norm{(u,p)}_{V\times Q} &\leq \left(1 + \frac{4+4\norm{a}}{\gamma} + \frac{2}{\beta} + \left(\frac{1+\norm{a}}{\gamma}\right)^2\right)\textbf{}  \sup_{(v,q)\in V\times Q} \frac{\mathcal{A}(u,p;v,q)}{\norm{(v,q)}_{V\times Q}}, \quad \forall (u,p)\in V\times Q.
\end{align*}
\end{theorem}
\begin{proof}
    The triangle inequality and the second Brezzi condition \eqref{brezzi-condition-2} lead to 
    \begin{align*}
        \norm{(u,p)}_{V\times Q} \leq \norm{u}_{V} + \frac{1}{\beta} \sup_{v\in V} \frac{\mathcal{A}(u,p;v,0)-a(u,v)}{{\norm{v}}_V} + t|p|_c.
    \end{align*}
    Note that $t|p|_c \leq \left(a(u,u)+t^2c(p,p)\right)^{\frac{1}{2}} = \left(\mathcal{A}(u,p;u,-p)\right)^{\frac{1}{2}}$, which together with the continuity 
    of $a(\cdot,\cdot)$ yields 
    \begin{align*}
        \norm{(u,p)}_{V\times Q} \leq \left(1+\norm{a}\right) \norm{u}_{V} + \frac{1}{\beta} \sup_{(v,q)\in V\times Q} \frac{\mathcal{A}(u,p;v,q)}{{\norm{(v,q)}_{V\times Q}}} + \left(A(u,p;u,-p)\right)^{\frac{1}{2}}.
    \end{align*}
    The Braess condition \eqref{braess-condition} implies that
    \begin{align*}
        &\norm{(u,p)}_{V\times Q} \leq \left(\frac{1+\norm{a}}{\gamma}\right) \sup_{(v,q)\in V\times Q} \frac{a(u,v)+b(u,q)}{\norm{(v,q)}_{V\times Q}} + \frac{1}{\beta} \sup_{(v,q)\in V\times Q} \frac{\mathcal{A}(u,p;v,q)}{{\norm{(v,q)}_{V\times Q}}} + \left(\mathcal{A}(u,p;u,-p)\right)^{\frac{1}{2}}.
    \end{align*}
    Since $a(u,v)+b(u,q) = \mathcal{A}(u,p;v,q) - b(v,p) + t^2c(p,q)$, we readily see that
    \begin{align*}
        \norm{(u,p)}_{V\times Q} &\leq \left(\frac{1+\norm{a}}{\gamma}+\frac{1}{\beta}\right) \sup_{(v,q)\in V\times Q} \frac{A(u,p;v,q)}{\norm{(v,q)}_{V\times Q}} + \left(1+\frac{1+\norm{a}}{\gamma}\right)\left(\mathcal{A}(u,p;u,-p)\right)^{\frac{1}{2}}.
    \end{align*}
    Given that for all $0 < x,y,z$ the inequality $x\leq y+z$ implies that $x\leq 2y + \frac{z^2}{x}$, we have
    \begin{align*}
        \norm{(u,p)}_{V\times Q} &\leq 2\left(\frac{1+\norm{a}}{\gamma}+\frac{1}{\beta}\right) \sup_{(v,q)\in V\times Q} \frac{\mathcal{A}(u,p;v,q)}{\norm{(v,q)}_{V\times Q}} + \left(1+\frac{1+\norm{a}}{\gamma}\right)^2\frac{\mathcal{A}(u,p;u,-p)}{\norm{(u,p)}_{V\times Q}}.
    \end{align*}
  The proof concludes by taking the supremum for all $(v,q)\in V\times Q$.
\end{proof}
\section{Virtual element  discretisation}\label{sec:vem}
This section introduces the 2D VEM from  \cite{khot2024}, based on \cite{beirao13,daveiga15-reaction-diffusion,daveiga15-stokes} with the estimates involving suitable polynomial projections adapted to the parameter-robust setting. Finally, we refer to  \cite[Section 4]{khot2024} for further details on the well-posedness of the discrete problem.

\medskip 
\noindent\textbf{Mesh assumptions.} Let $\mathcal{T}^h$ be a decomposition of $\Omega$ into polygonal elements $E$ with  diameter $h_E$, let $\mathcal{E}^h$ be the set of edges $e$ of $\mathcal{T}^h$ with length $h_e$. 
We assume that there exists a universal constant $\rho>0$ such that
\begin{enumerate}[label={\textbf{(M\arabic*)}}, align=left, leftmargin=*, labelwidth=!, labelsep=1em]
    \item \label{M1} Every polygonal element $E$ of diameter $h_E$ is star-shaped with respect to a disk of radius $\geq$ $\rho h_E$,
    \item \label{M2} Every edge $e$ of $E$ has length $\geq$ $\rho h_E$. 
\end{enumerate}
We split the set of all edges as $\mathcal{E}^h = \mathcal{E}^h_\Omega \cup \mathcal{E}^h_{\mathrm{D}} \cup \mathcal{E}^h_{\mathrm{N}}$, where $\mathcal{E}^h_\Omega = \{ e\in \mathcal{E}^h: e\subset \Omega\}$, $\mathcal{E}^h_{\mathrm{D}} = \{ e\in \mathcal{E}^h : e\subset \Gamma_{\mathrm{D}}\}$ and $\mathcal{E}^h_{\mathrm{N}} = \{ e\in \mathcal{E}^h : e\subset \Gamma_{\mathrm{N}}\}$. The set of edges of $E\in \mathcal{T}^h$ is denoted as $\mathcal{E}^h(E)$, the edges of $E$ which are not in the boundary $\partial \Omega$ are denoted by $\mathcal{E}^h_\Omega(E)$ and the ones that lie on the Dirichlet portion of the boundary (resp. Neumann) are denoted by $\mathcal{E}^h_{\mathrm{D}}(E)$ (resp. $\mathcal{E}^h_{\mathrm{N}}(E)$), the set of elements $E$ that share $e$ as an edge is denoted by $\mathcal{T}^h_e$. The normal and tangential jump operators are defined as usual by $\jump{\bu \cdot \bn_e}:= (\bu|_{E} - \bu|_{E'})|_e \cdot \bn_e$ and $\jump{\bzeta \cdot \bt_e}:= (\bzeta|_{E} - \bzeta|_{E'})|_e \cdot \bt_e$, where $E,E' \in \mathcal{T}^h_e$, and $\bn_e$ and $\bt_e$ are the outward normal and tangential counterclockwise vectors of $e$ with respect to $\partial E$. The space $D(E)$ denotes the union of elements in $\mathcal{T}^h$ intersecting $E$.

\medskip 
\noindent\textbf{Polynomial spaces.} Given  an integer $k\geq 0$ the space of polynomials of degree $\leq k$ on $E$ is denoted by $\mathcal{P}_k(E)$ (resp. for edges $e$). The space of the gradients of polynomials of grade $\leq k+1$ on $E$ is denoted as $\boldsymbol{\mathcal{G}}_k(E):= \nabla (\mathcal{P}_{k+1}(E))$ with standard notation $\mathcal{P}_{-1}(E)=\{0\}$ for $k=-1$. The space $\boldsymbol{\mathcal{G}}_k^\oplus(E)$ denotes the complement of the space $\boldsymbol{\mathcal{G}}_k(E)$ in the vector polynomial space $\boldsymbol{\mathcal{P}}_k(E)$,  that is, $\boldsymbol{\mathcal{P}}_k(E) = \boldsymbol{\mathcal{G}}_k(E) \oplus  \boldsymbol{\mathcal{G}}_k^\oplus(E)$. In particular, following \cite{veiga19}, we set $\boldsymbol{\mathcal{G}}_k^\oplus(E)= \bx^\perp \mathcal{P}_{k-1}(E)$ where $\bx^\perp = (x_2,-x_1)^{\tt t}$. Likewise, the space that defines the rotational of polynomials with degree $\leq k+1$ is denoted as $\boldsymbol{\mathcal{R}}_{k}(E) := \brot(\mathcal{P}_{k+1}(E))$ where the associated complement space $\boldsymbol{\mathcal{R}}_k^\oplus(E)$ fulfills the property $\boldsymbol{\mathcal{P}}_k(E) = \boldsymbol{\mathcal{R}}_k(E) \oplus  \boldsymbol{\mathcal{R}}_k^\oplus(E)$ with $\boldsymbol{\mathcal{R}}_k^\oplus(E) = \bx \mathcal{P}_{k-1}(E)$.

Let $\bx_E = (x_{1,E},x_{2,E})^{\tt t}$ denote the barycentre of $E$ and let $\mathcal{M}_k(E)$ be the set of scaled monomials
\[\mathcal{M}_k(E):=\left\{ \left( \frac{\bx-\bx_E}{h_E} \right)^{\boldsymbol{\alpha}}, 0\leq |\boldsymbol{\alpha}|\leq k \right\},\]
where $\boldsymbol{\alpha}=(\alpha_1,\alpha_2)^{\tt t}$ is a non-negative multi-index with $|\boldsymbol{\alpha}|=\alpha_1+\alpha_2$ and $\bx^{\boldsymbol{\alpha}}=x_1^{\alpha_1}x_2^{\alpha_2}$ for $\bx = (x_1,x_2)^{\tt t}$. In particular, we can take the basis of $\boldsymbol{\mathcal{G}}_{k}(E)$ and $\boldsymbol{\mathcal{G}}_{k}^\oplus(E)$ as $\boldsymbol{\mathcal{M}}_{k}^{\nabla}(E):=\nabla\mathcal{M}_{k+1}(E)\setminus\{\mathbf{0}\}$ and $\boldsymbol{\mathcal{M}}_{k}^{\oplus}(E):=\mathbf{m}^\perp \mathcal{M}_{k-1}(E)$, with $\mathbf{m}^\perp := (\frac{x_2-x_{2,E}}{h_E},\frac{x_{1,E}-x_1}{h_E})^{\tt t}$, $\mathbf{m} :=  \frac{\bx-\bx_E}{h_E}$, respectively. 

The spaces introduced below are constructed to address the variational forms of \eqref{eq:linear-momentum}–\eqref{eq:hooke-law} and \eqref{eq:diffusive-flux}–\eqref{eq:reaction-diffusion}, respectively. Their 
unisolvency in terms of \textit{Degrees of Freedom} (DoFs)  
and other properties, are detailed in \cite{daveiga15-stokes,daveiga15-reaction-diffusion}. 

\subsection{Discrete spaces}
For  $k_1\geq 2$, the discrete {displacement} space locally solves a Stokes problem:
\begin{align*}
\bV_1^{h,k_1}(E) := \{ \bv\in \bH^1(E) \colon &\bv|_{\partial E}\in \boldsymbol{\mathcal{B}}^{k_1}(\partial E),\; \vdiv\bv \in \mathcal{P}_{k_1-1}(E),\\ 
&-2\mu\bdiv\beps(\bv) - \nabla s \in \boldsymbol{\mathcal{G}}_{k_1-2}^\oplus(E),\text{ for some } s\in \mathrm{L}^2_0(E)\},
\end{align*}
where $\mathcal{B}^k(\partial E)$ is the continuous space of polynomials along the boundary $\partial E$ of $E$  defined as
$$\mathcal{B}^k(\partial E) := \left\{ v\in C^0(\partial E) \colon v|_{e} \in \mathcal{P}_k(e), \, \forall e\subset \partial E \right\}.$$
Observe that $\boldsymbol{\mathcal{P}}_{k_1}(E)\subseteq \bV_1^{h,k_1}(E)$. Then, the global discrete spaces are defined as 
\begin{align*}
    \bV_1^{h,k_1} := \{ \bv \in \bV_1 \colon \bv|_E \in \bV_1^{h,k_1}(E), \ \forall E\in \mathcal{T}^h\}, \quad 
    \mathrm{Q}_1^{h,k_1} := \{ {q}\in \mathrm{Q}_1 \colon {q}|_E\in \mathcal{P}_{k_1-1}(E), \ \forall E\in \mathcal{T}^h \}.
\end{align*}
The set of DoFs for $\bv_h \in \bV_1^{h,k_1}(E)$ and ${q}_h\in \mathrm{Q}_1^{h,k_1}(E)$ are selected as
\begin{align*}
&\bullet \text{The values of } \bv_h \text{ at the vertices of $E$}, \\
&\bullet \text{The values of } \bv_h \text{ at the $k_1-1$ internal Gauss-Lobatto quadrature points on each edge of}\;  E, \\
&\bullet \int_E (\vdiv\bv_h) m_{k_1-1}, \quad \forall m_{k_{1}-1} \in \mathcal{M}_{k_{1}-1}(E)\setminus \left\{\frac{1}{h_E}\right\},\\
&\bullet \int_E \bv_h \cdot \mathbf{m}_{k_{1}-2}^{\oplus}, \quad \forall \mathbf{m}_{k_{1}-2}^{\oplus} \in \boldsymbol{\mathcal{M}}_{k_{1}-2}^{\oplus}(E),\\
&\bullet \int_E {q}_h m_{k_{1}-1}, \quad \forall m_{k_{1}-1}\in \mathcal{M}_{k_{1}-1}(E).
\end{align*}

For the reaction-diffusion equation and for $k_2\geq 0$, the discrete flux space locally solves a $\vdiv$-$\vrot$ problem:  
\begin{align*}
\bV_2^{h,k_2}(E) := \{ \bxi\in \bH(\vdiv,E)\cap \bH(\vrot,E) \colon &\bxi\cdot \bn|_e\in \mathcal{P}_{k_2}(e), \, \forall e\subset \partial E,\, \vdiv\bxi \in \mathcal{P}_{k_2}(E),\; \vrot\bxi \in \mathcal{P}_{k_2-1}(E)\}.
\end{align*}
Note that $\boldsymbol{\mathcal{P}}_{k_2}(E)\subseteq \bV_2^{h,k_2}(E)$. In turn, the global discrete spaces are defined as follows
\begin{align*}
    \bV_2^{h,k_2} := \{ \bxi \in \bV_2 \colon \bxi|_E \in \bV_2^{h,k_2}(E), \, \forall E\in \mathcal{T}^h \}, \quad \mathrm{Q}_2^{h,k_2} := \{ \psi\in \mathrm{Q}_2 \colon \psi|_E\in \mathcal{P}_{k_2}(E), \, \forall E\in \mathcal{T}^h \}.
\end{align*}
The set of DoFs for $\bxi_h\in \bV_2^{h,k_2}(E)$ and $\psi_h\in \mathrm{Q}_2^{h,k_2}(E)$ can be taken as
\begin{align*}
&\bullet \text{The values of } \bxi_h\cdot \bn \text{ at the $k_2+1$ Gauss--Lobatto quadrature points of each edge of $E$}, \\
&\bullet  \int_E \bxi_h\cdot \mathbf{m}_{k_2-1}^\nabla, \quad \forall \mathbf{m}_{k_2-1}^{\nabla} \in \boldsymbol{\mathcal{M}}_{k_2-1}^{\nabla}(E),\\
&\bullet \int_E \bxi_h \cdot \mathbf{m}_{k_2}^{\oplus}, \quad \forall \mathbf{m}_{k_2}^{\oplus} \in \boldsymbol{\mathcal{M}}_{k_2}^{\oplus}(E), \\
&\bullet \int_E \psi_h m_{k_2}, \quad \forall m_{k_2}\in \mathcal{M}_{k_2}(E).
\end{align*}
Although $k_1$ and $k_2$ are in general arbitrary and independent, we select $k_1 = k_2 + 1$ with $k_2 \ge 1$ to achieve optimal convergence. This choice is justified later by Theorem~\ref{convergence-rates}.

\subsection{Projection and interpolation operators}\label{sec:proj_interp} 
The operators below are required for the discrete formulation and the subsequent error analysis. We follow \cite{daveiga15-stokes} for the elasticity problem and \cite{daveiga15-reaction-diffusion} for the reaction-diffusion problem.
Given $E\in \mathcal{T}^h$, the energy projection operator $\bPi_{1}^{\beps,k_1}: \bH^1(E)\rightarrow \boldsymbol{\mathcal{P}}_{k_1}(E)$ is defined, for all $\bv \in \bH^1(E)$, by
\begin{align*}
    \int_E \beps(\bv-\bPi_{1}^{\beps,k_1}\bv) \colon \beps(\mathbf{m}_{k_1}) &= 0, \quad \forall \mathbf{m}_{k_1} \in \boldsymbol{\mathcal{M}}_{k_1}(E),\\ 
    \int_{\partial E} (\bv-\bPi_{1}^{\beps,k_1} \bv) \cdot \mathbf{m}_{\textbf{RBM}} &=  0, \quad \forall\mathbf{m}_{\textbf{RBM}}\in \textbf{RBM}(E):=  
        \left\{ \begin{pmatrix} \frac{1}{h_E} \\ 0 \end{pmatrix}, \begin{pmatrix} 0 \\ \frac{1}{h_E} \end{pmatrix}, \begin{pmatrix} \frac{x_{2,E}-x_2}{h_E} \\ \frac{x_1-x_{1,E}}{h_E} \end{pmatrix} \right\},
\end{align*}
where $\textbf{RBM}(E)$ is the set of scaled rigid body motions. We also define the $\bL^2$-projection   $\bPi_{j}^{0,k}: \bL^2(E)\rightarrow \boldsymbol{\mathcal{P}_{k}}(E)$ by 
\begin{align*}
    \int_E (\bzeta-\bPi_{j}^{0,k}\bzeta)\cdot \mathbf{m}_{k} = 0, \quad \forall \mathbf{m}_{k} \in \boldsymbol{\mathcal{M}}_{k}(E), \forall \bzeta \in \bL^2(E), 
\end{align*}
and analogously for scalar functions. For  clarity, $\bPi_{1}^{0,k_1}$ is the projection associated with the elasticity problem, with polynomial degree $k_1$ (resp. $\bPi_{2}^{0,k_2}$ for the reaction-diffusion problem). Regarding computability of $\bPi_1^{\beps,k_1}, \bPi_1^{0,k_1-2}$ on $\bV_1^{h,k_1}$ and $\bPi_2^{0,k_2}$ on $\bV_2^{h,k_2}$ in terms of the respective DoFs, we refer to \cite[Section 3.2]{daveiga15-stokes} and \cite[Theorem 3.2]{veiga-Hdiv}. Next, we present a scaled version of classical polynomial approximation estimates  \cite{Brenner1994}.

\begin{lemma} \label{approximation-estimates} For $E\in \mathcal{T}^h$, let $\bv \in \bH^1(E)$, $q\in \mathrm{L}^2(E)$, $\bxi \in \bH^1(E)$, and $\psi \in \mathrm{L}^2(E)$. Then the following estimates hold:
    \begin{alignat*}{3}
    \norm{\bv-\bPi_1^{\beps,k_1}\bv}_{\bV_1(E)}&\lesssim |\bv|_{\bV_1(E)},\quad &\norm{q-\Pi_1^{0,k_1}q}_{\mathrm{Q}_1(E)}&\lesssim |q|_{\mathrm{Q}_1(E)},\\
    \norm{\bxi-\bPi_2^{0,k_2}\bxi}_{\bbM,E}&\lesssim h_E|\bxi|_{\bV_2(E)},\quad &\norm{\psi-\Pi_2^{0,k_2}\psi}_{\mathrm{Q}_2(E)}&\lesssim |\psi|_{\mathrm{Q}_2(E)}.
    \end{alignat*}
\end{lemma}

Finally, we can define a quasi--interpolation operator $\mathbf{I}_1^{Q,k_1}:\mathbf{H}^{1}(E)\rightarrow \bV_1^{h,k_1}(E)$ and a Fortin operator $\mathbf{I}_2^{F,k_2}:\mathbf{H}^{1}(E)\rightarrow \bV_2^{h,k_2}(E)$ satisfying the  estimates in the next lemma.  
\begin{lemma} \label{interpolation-estimates}
Let $E\in \mathcal{T}^h$, and $e\in \mathcal{E}^h(E)$, and let $\bv\in \bH^{1}(E)$ and $\bxi \in \bH^{1}(E)$. Then the following estimates hold:
\begin{subequations}
\begin{align}
\label{tt0-1}    h_e^{-\frac{1}{2}} 2\mu \norm{\bv-\mathbf{I}_1^{Q,k_1}\bv}_{0,e} + \norm{\bv-\mathbf{I}_1^{Q,k_1}\bv}_{\bV_1(E)}& \lesssim |\bv|_{\bV_1(D(E))},\\
\label{tt0-2}    h_e^{\frac{1}{2}} M \norm{(\bxi-\mathbf{I}_2^{F,k_2}\bxi)\cdot \bn_e}_{0,e} +\norm{\bxi - \mathbf{I}_2^{F,k_2}\bxi}_{\bbM,E} &\lesssim h_E|\bxi|_{\bV_2(E)}.
\end{align}\end{subequations}
\end{lemma}

The 2D construction of $\mathbf{I}_1^{Q,k_1}$ is  in \cite[Proposition 4.2]{beirao2020stokes}. 
It does not require extra regularity \cite{daveiga2022stability}, and \eqref{tt0-2} follows from   trace inequality. On the other hand, $\mathbf{I}_2^{F,k_2}$ can be defined directly through the   DoFs for $\bH^1(\Omega)$ functions. Moreover, the commutative property $\vdiv \mathbf{I}_2^{F,k_2}(\cdot) = \Pi_2^{0,k_2} \vdiv (\cdot)$ holds, leading to the discrete stability of the   reaction-diffusion system (see \cite[Section 3.2]{daveiga15-reaction-diffusion}). We refer to \cite[Lemma 5.2]{munar2024} for a proof of \eqref{tt0-1}.

\subsection{The virtual element formulation for the stress-assisted diffusion problem}\label{vem-problem-2d}
The discrete formulation for the fully-coupled problem reads as follows: Given $\fb\in \bL^2(\Omega)$, $g\in \mathrm{L}^2(\Omega)$, 
and $\varphi_{\mathrm{D}}\in \mathrm{H}^{\frac{1}{2}}(\Gamma_{\mathrm{D}})$, find 
$(\bu_h,{p}_h,\bzeta_h,\varphi_h) \in \bV_1^{h,k_1}\times \mathrm{Q}_1^{h,k_1} \times \bV_2^{h,k_2} \times \mathrm{Q}_2^{h,k_2}$ such that 
\begin{subequations}\label{eq:weak-discrete}
\begin{alignat}{2}
  a_1^{h}(\bu_h,\bv_h) + b_1(\bv_h,{p}_h) &= F_1^h(\bv_h), \qquad &\forall \bv_h \in \bV_1^{h,k_1}, \label{eq:sto1}\\
  b_1(\bu_h,{q}_h) - \frac{1}{\lambda} c_1({p}_h,{q}_h) &= G_1^{\varphi_h}({q}_h), \qquad &\forall {p}_h \in \mathrm{Q}_1^{h,k_1},\label{eq:sto2}\\
a_2^{\overline{\bu}_h,{p}_h,h}(\bzeta_h,\bxi_h) + b_2(\bxi_h,\varphi_h) &= F_2(\bxi_h), \qquad &\forall \bxi_h \in \bV_2^{h,k_2},\label{eq:mix1}\\
  b_2(\bzeta_h,\psi_h) - \theta c_2(\varphi_h,\psi_h) &= G_2(\psi_h), \qquad &\forall \psi_h \in \mathrm{Q}_2^{h,k_2},\label{eq:mix2}
\end{alignat}
\end{subequations}
where $\overline{\bu}_h := \bPi_1^{\beps,k_1}\bu_h$. The discrete bilinear and linear forms are defined by adding the local contributions, as
\begin{align*}
    &a_1^{h}(\bu_h,\bv_h) := \sum_{E\in \mathcal{T}^h} a_{1}^{h,E}(\bu_h,\bv_h) := \sum_{E\in \mathcal{T}^h} a_1^E(\bPi_1^{\beps,k_1}\bu_h,\bPi_1^{\beps,k_1}\bv_h) + \sum_{E\in \mathcal{T}^h} S_1^E(\bu_h-\bPi_1^{\beps,k_1}\bu_h,\bv_h-\bPi_1^{\beps,k_1}\bv_h),\\
    &a_2^{\overline{\bu}_h,{p}_h,h}(\bzeta_h,\bxi_h) := \sum_{E\in \mathcal{T}^h} a_{2}^{\overline{\bu}_h,{p}_h,h,E}(\bzeta_h,\bxi_h)\\
    & := \sum_{E\in \mathcal{T}^h} a_{2}^{\overline{\bu}_h,{p}_h,E}(\bPi_2^{0,k_2}\bzeta_h, \bPi_2^{0,k_2}\bxi_h) + \sum_{E\in \mathcal{T}^h} S_2^{\overline{\bu}_h,{p}_h,E}(\bzeta_h-\bPi_2^{0,k_2}\bzeta_h,\bxi_h-\bPi_2^{0,k_2}\bxi_h),\\
    &F_1^h(\bv_h) := \sum_{E\in \mathcal{T}^h} F_1^{h,E}(\bv_h) := \sum_{E\in \mathcal{T}^h} \int_E \fb\cdot\bPi_1^{0,k_1-2}\bv_h 
    = \sum_{E\in \mathcal{T}^h} \int_E \bPi_1^{0,k_1-2} \fb \cdot \bv_h. 
\end{align*}
The stabilisation $S_1^E(\cdot,\cdot)$ and $S_2^{\overline{\bu}_h,{p}_h,E}(\cdot,\cdot)$ are symmetric and positive definite bilinear forms such that
\begin{subequations}
\begin{alignat}{2}
    a_{1}^{E}(\bv_h,\bv_h)&\lesssim S_{1}^{E}(\bv_h,\bv_h)\lesssim a_{1}^{E}(\bv_h,\bv_h), \quad &\forall \bv_h\in \ker(\bPi_1^{\beps,k_1}), \label{stabilisation-elast} \\
    a_{2}^{\overline{\bu}_h,{p}_h,E}(\bxi_h,\bxi_h)&\lesssim S_{2}^{\overline{\bu}_h,{p}_h,E}(\bxi_h,\bxi_h)\lesssim a_{2}^{\overline{\bu}_h,{p}_h,E}(\bxi_h,\bxi_h), \quad &\forall \bxi_h\in \ker(\bPi_2^{0,k_2}). \label{stabilisation-reaction-diff} 
\end{alignat}\end{subequations}
We conclude by establishing the continuous dependence on data for \eqref{eq:weak-discrete} and the convergence result of the total error $\overline{\textnormal{e}}_h:=\norm{(\bu-\bu_h,{p}-{p}_h, \bzeta-\bzeta_h,\varphi-\varphi_h)}_{\bV_1\times Q_{1} \times \bV_2\times \mathrm{Q}_2}$ (see \cite[Sections~4-5]{khot2024} for a proof). Note that the convergence rate shown in Theorem~\ref{convergence-rates} is optimal when $k_1=k_2+1$ and $k_2\geq1$. For sake of simplicity, we introduce the constant $\overline{C}_3 := \max\left\{\frac{1}{\sqrt{2\mu}}, \sqrt{2\mu}\right\} \sqrt{M} L_\bbM \overline{C}_2 (\norm{\varphi_{\mathrm{D}}}_{\frac{1}{2},\Gamma_{\mathrm{D}}} + \norm{g}_{0,E})$. Similar to the continuous case (see Section~\ref{sec:robust-setting}), we assume that the following small data assumptions are given
\begin{subequations}
\begin{align}
    \overline{C}_1 L_\ell \overline{C}_3 \overline{C}_2 \sqrt{M^3} &< 1, \label{small_data_disc1}\\
    \overline{C}_1 \sqrt{M} L_\ell + \overline{C}_3 \overline{C}_2 M &< \frac{1}{2}. \label{small_data_disc2}
\end{align}\end{subequations}

\begin{theorem}\label{well-posedness-discrete}
Suppose that $1\leq \lambda$, $0<\mu$, $\theta \leq \frac{1}{M}$, and that the small data assumption \eqref{small_data_disc1} holds. 
Then, there exists a  unique solution $(\bu_h,{p}_h,\bzeta_h,\varphi_h)\in \bV_1^{h,k_1}\times \mathrm{Q}_1^{h,k_1} \times \bV_2^{h,k_2} \times \mathrm{Q}_2^{h,k_2}$ to \eqref{eq:weak-discrete} such that
    \begin{align*}
        \norm{(\bu_h,{p}_h)}_{\bV_1\times \mathrm{Q}_1} &\leq \overline{C}_1 \left( \norm{F_1^h}_{\bV'_1} + \norm{G^{\varphi_h}_1}_{Q'_1} \right),\\
\norm{(\bzeta_h,\varphi_h)}_{\bV_2\times \mathrm{Q}_2} &\leq \overline{C}_2 \left( \norm{F_2}_{\bV'_2} +\norm{G_2}_{Q'_2} \right),
    \end{align*}
where the  constants $\overline{C}_1$ and $\overline{C}_2$ do not depend on the physical parameters. 
\end{theorem}

\begin{theorem}\label{convergence-rates}
    Adopt the assumptions of Theorem~\ref{well-posedness} and~\ref{well-posedness-discrete}, along with the small data condition \eqref{small_data_disc2}. Let $(\bu,{p},\bzeta,\varphi)\in (\bH^{s_1+1}(\Omega)\cap \bV_1,|\cdot|_{s_1+1,\bV_1})\times (\mathrm{H}^{s_1}(\Omega)\cap Q_{b_1},|\cdot|_{s_1,Q_{b_1}}) \times (\bH^{s_2+1}(\Omega)\cap \bV_2,|\cdot|_{s_2+1,\bV_2}) \times (\mathrm{H}^{s_2+1}(\Omega)\cap Q_{b_2},|\cdot|_{s_2+1,Q_{b_2}})$, and $(\bu_h,{p}_h,\bzeta_h,\varphi_h)\in \bV_1^{h,k_1}\times \mathrm{Q}_1^{h,k_1}\times \bV_2^{h,k_2}\times \mathrm{Q}_2^{h,k_2}$ solve  respectively the continuous and discrete problems, with the data  $\fb\in (\bH^{s_1-1}\cap \bQ_{b_1},|\cdot|_{s_1-1,\bQ_{b_1}})$ and $g\in (\mathrm{H}^{s_2+1}(\Omega)\cap Q_{b_2},|\cdot|_{s_2+1,Q_{b_2}})$ where $0\leq s_1\leq k_1$ and $0\leq s_2\leq k_2$. Then, the total error $\overline{\textnormal{e}}_h$ decays with the following rate 
    \begin{align*}
         \overline{\textnormal{e}}_h \lesssim h^{\min \left\{s_1,s_2+1\right\}} (|\fb|_{s_1-1,\bQ_{b_1}} + |\bu|_{s_1+1,\bV_1} + |{p}|_{s_1,Q_{b_1}} + |g|_{s_2+1,Q_{b_2}} + |\bzeta|_{s_2+1,\bV_2}+|\varphi|_{s_2+1,Q_{b_2}}).
    \end{align*}
\end{theorem}
\section{A posteriori error analysis}\label{sec:aposteriori}
This section aims at deriving reliable and efficient residual-based a posteriori error estimators for the VEM of  Section~\ref{sec:vem}. The reliability is a consequence of the global inf-sup provided in Theorem~\ref{th:global-inf-sup}, together with the tools given in Subsection~\ref{toolkit}. The efficiency  is proven using bubble functions.

\subsection{Preliminary toolkit} \label{toolkit}
This subsection presents the necessary local estimates, Helmholtz decomposition, and bubble function results required for the a posteriori error analysis. First, an estimate involving the projection of the load term $\fb$ is presented (see \cite[Lemma 3.7]{daveiga15-stokes}).
\begin{lemma}\label{estimate-fb}
    Let $\fb \in L^{2}(E)$. Then, for all $\bv \in \bH^{1}(E)$, we have
    \begin{align*}
        \left|\left[F_1-F_1^h\right](\bv)\right| \lesssim h_E \norm{\fb-\bPi_1^{0,k_1-2}\fb}_{0,E}|\bv|_{1,E}.
    \end{align*}
\end{lemma}
Next, we provide an estimate for the local virtual approximation of the bilinear form $a_1^E(\cdot,\cdot)$.
\begin{lemma}\label{a_1-error}
    For all $\bu_h,\bv_h\in \bV_1^{h,k_1}(E)$, the following estimate holds:
    \begin{align*}
        \left|\left[a_1^{h,E}-a_1^E\right](\bu_h,\bv_h)\right|\lesssim \left(S_1^E(\bu_h-\bPi_1^{\beps,k_1}\bu_h,\bu_h-\bPi_1^{\beps,k_1}\bu_h)\right)^{\frac{1}{2}}\norm{\bv_h}_{\bV_1(E)}.
    \end{align*}
\end{lemma}
\begin{proof}
    The definition of the polynomial projection $\bPi_1^{\beps,k_1}$ leads to
    \begin{align*}
        &\left[a_1^{h,E}-a_1^E\right](\bu_h,\bv_h) = 2\mu \int_E \beps(\bPi_1^{\beps,k_1}\bu_h):\beps(\bPi_1^{\beps,k_1}\bv_h) + S_1^{E}(\bu_h - \bPi_1^{\beps,k_1}\bu_h, \bv_h - \bPi_1^{\beps,k_1}\bv_h) - 2\mu \int_E \beps(\bu_h):\beps(\bv_h) \\
        & = -2\mu \int_E \beps(\bu_h-\bPi_1^{\beps,k_1}\bu_h):\beps(\bv_h-\bPi_1^{\beps,k_1}\bv_h) + S_1^{E}(\bu_h - \bPi_1^{\beps,k_1}\bu_h, \bv_h - \bPi_1^{\beps,k_1}\bv_h).
    \end{align*}
    The proof is completed after using a Cauchy--Schwarz inequality, \eqref{stabilisation-elast}, and Lemma~\ref{approximation-estimates}.
\end{proof}

Two additional local estimates are required for the bilinear form $a_2^{\overline{\bu}_h,p_h,h,E}(\cdot,\cdot)$ and the nonlinear term $\bbM^{-1}(\beps(\overline{\bu}),p)\bzeta$. They are adapted from \cite[Lemma 4.3]{munar2024} to the robust case. 
\begin{lemma} \label{a_2-error}
    Given $\overline{\bu}_h \in \bV_1^{h,k_1}(E)$, $p_h\in \mathrm{Q}_1^{h,k_1}(E)$ and $\bzeta_h,\bxi_h \in \bV_2^{h,k_2}(E)$, the following estimate holds:
    \begin{align*}
        &\left| \left[a_2^{\overline{\bu}_h,p_h,h,E}-a_2^{\overline{\bu}_h,p_h,E}\right](\bzeta_h,\bxi_h) \right| \lesssim  \biggl(\left(S_2^E(\bzeta_h - \bPi_2^{0,k_2}\bzeta_h,\bzeta_h - \bPi_2^{0,k_2}\bzeta_h)\right)^{\frac{1}{2}}\\
        & \quad + \sqrt{M}\norm{\bbM^{-1}(\beps(\overline{\bu}_h),p_h) \bPi_2^{0,k_2}\bzeta_h - \bPi_2^{0,k_2}(\bbM^{-1}(\beps(\overline{\bu}_h),p_h) \bPi_2^{0,k_2}\bzeta_h)}_{0,E} \biggr) \norm{\bxi_h}_{\bV_2(E)}.
    \end{align*}
\end{lemma}
\begin{proof}
    The $\bL^2$-orthogonality of $\bPi_2^{0,k_1}$ shows that
    \begin{align*}
        &\left[a_2^{\overline{\bu}_h,p_h,h,E}-a_2^{\overline{\bu}_h,p_h,E}\right](\bzeta_h,\bxi_h) = a_2^{\overline{\bu}_h,p_h,E}(\bPi_2^{0,k_2}\bzeta_h, \bPi_2^{0,k_2}\bxi_h) - a_2^{\overline{\bu}_h,p_h,E}(\bzeta_h,\bxi_h)\\
        & \quad + S_2^E(\bzeta_h - \bPi_2^{0,k_2} \bzeta_h, \bxi_h - \bPi_2^{0,k_2}\bxi_h)\\
        & = -\int_E \left[\bbM^{-1}(\beps(\overline{\bu}_h),p_h) \bPi_2^{0,k_2}\bzeta_h -  \bPi_2^{0,k_2}( \bbM^{-1}(\beps(\overline{\bu}_h),p_h) \bPi_2^{0,k_2}\bzeta_h )\right]\cdot ( \bxi_h - \bPi_2^{0,k_2}\bxi_h) \\
        &\quad - \int_E \bbM^{-1}(\beps(\overline{\bu}_h),p_h)( \bzeta_h -  \bPi_2^{0,k_2}\bzeta_h )\cdot \bxi_h +  S_2^E(\bzeta_h - \bPi_2^{0,k_2} \bzeta_h, \bxi_h - \bPi_2^{0,k_2}\bxi_h).
    \end{align*}
    The Cauchy--Schwarz inequality, \eqref{stabilisation-reaction-diff}, the continuity of $\bPi_2^{0,k_2}$ and a proper scaling finish the proof. 
\end{proof}
\begin{lemma}\label{extra_estimates_error_M}
    Given $E\in \mathcal{T}^h$, $(\bu,{p})\in \bV_1\times Q$ and $(\overline{\bu}_h,{p})\in \bV_1^{h,k_1}\times \mathrm{Q}_1^{h,k_1}$ the following estimate holds:
    \begin{align*}
        & \norm{\bbM^{-1}(\beps(\bu),{p})\bzeta - \bPi_2^{0,k_2}(\bbM^{-1}(\beps(\overline{\bu}_h),{p}_h)\bPi_2^{0,k_2}\bzeta_h)}_{0,E} \lesssim \left(S_2^E(\bzeta_h - \bPi_2^{0,k_2}\bzeta_h,\bzeta_h - \bPi_2^{0,k_2}\bzeta_h)\right)^{\frac{1}{2}}\\
        & \quad + \overline{C}_3 \left(\norm{\bu-\bu_h}_{\bV_1(E)} + \norm{{p}-{p}_h}_{\mathrm{Q}_1(E)} + \left(S_1^E(\bu_h - \bPi_1^{\beps,k_1}\bu_h,\bu_h - \bPi_1^{\beps,k_1}\bu_h)\right)^{\frac{1}{2}}\right) \\
        & \quad + \norm{\bbM^{-1}(\beps(\overline{\bu}_h),p_h) \bPi_2^{0,k_2}\bzeta_h - \bPi_2^{0,k_2}(\bbM^{-1}(\beps(\overline{\bu}_h),p_h) \bPi_2^{0,k_2}\bzeta_h)}_{0,E}.
    \end{align*}
\end{lemma}
\begin{proof}
    The proof follows by adding and subtracting suitable terms,  
    applying the Cauchy--Schwarz inequality, \eqref{stabilisation-elast}, and invoking the well-posedness of the discrete reaction-diffusion equation. 
\end{proof}

To apply the Helmholtz decomposition in the forthcoming analysis (see \cite{munar2024}), we require the scalar space $\mathrm{V}_3 = \mathrm{H}^{1}_{\mathrm{N}}(\Omega):= \{ \chi \in \mathrm{H}^1(\Omega): \chi=0 \text{ on } \Gamma_{\mathrm{N}}\}$ and its local discrete virtual space
\begin{equation*}
    \mathrm{V}_3^{h,k_2+1}(E):= \{\chi\in \mathrm{H}_{\mathrm{N}}^1(E): \chi |_{\partial E} \in \mathcal{B}^{k_2+1}(\partial E) \text{ and } \Delta \chi \in \mathcal{P}_{k_2 - 1}(E) \}.
\end{equation*}
It is shown in \cite[Lemma 5.1]{munar2024} that $\brot \chi_h \in \bV_2^{h,k_2}$ for $\chi_h \in \mathrm{V}_3^{h,k_2+1}$, where the global space is given by 
$$\mathrm{V}_3^{h,k_2+1}:= \{ \chi \in \mathrm{V}_3 \colon \chi|_E \in \mathrm{V}_3^{h,k_2+1}(E), \, \forall E\in \mathcal{T}^h  \}.$$
This relation plays an important role in the argument used in Lemma~\ref{residual-bound-2}. 

Next, we introduce the quasi-interpolation error operator for functions in the auxiliary space $\mathrm{V}_3$ defined as $\mathrm{I}_3^{Q,k_2+1}: \mathrm{H}^1(E)\rightarrow \mathrm{V}^{h,k_2+1}_3(E)$ (see \cite[Proposition 4.2]{mora15} for details). The interpolation error involving the $\norm{\cdot}_{0,e}$ norm and functions in $\mathrm{H}^1(E)$ can be proven directly from \cite[Theorem 3.10]{ern21} and \cite[Section~1.3.2]{gatica14}.
\begin{lemma} \label{fortin-auxiliary}
    Let $E\in \mathcal{T}^h$, and $e\in \mathcal{E}^h(E)$. For $\chi \in \mathrm{H}^{1}(E)$, the following estimate holds:
    \begin{align*}
        h_e^{\frac{1}{2}} \norm{\chi-\mathrm{I}_3^{Q,k_2+1}\chi}_{0,e} + \norm{\chi-\mathrm{I}_3^{Q,k_2+1}\chi}_{0,E} \lesssim h_E |\chi|_{1,E}.
    \end{align*}
\end{lemma}
As a consequence of the mesh assumptions from Section~\ref{sec:vem}, a shape-regular triangulation $\widetilde{\mathcal{T}}^h$ of $\Omega$ can be constructed via a sub-triangulation of each polygon into triangles (see \cite{AINSWORTH1997} for the construction of bubble functions in triangles). Let us recall the following results from \cite{cangiani2017posteriori} regarding element and edge bubble functions supported in polygonal elements and edges (union of triangles sharing an edge), respectively.
\begin{lemma}\label{bubble-int}
    Let $E\in \mathcal{T}^h$ and $\Psi_E$ be an element bubble function. For all $\mathbf{q}_E \in \boldsymbol{\mathcal{P}}_{k}(E)$, the following bounds hold
    \begin{gather*}
        \norm{\mathbf{q}_E}_{0,E}^2 \lesssim \norm{\sqrt{\Psi_E} \mathbf{q}_E}_{0,E}^2 \lesssim \norm{\mathbf{q}_E}_{0,E}^2, \quad \norm{\mathbf{q}_E}_{0,E} \lesssim \norm{\Psi_E \mathbf{q}_E}_{0,E} + h_E |\Psi_E \mathbf{q}_E|_{1,E} \lesssim \norm{\mathbf{q}_E}_{0,E}. 
    \end{gather*}
\end{lemma}
\begin{lemma}\label{bubble-inner}
    For $E\in \mathcal{T}^h$, let $e \subset \partial E$ and $\Psi_e$ be the corresponding edge bubble function. For all 
    $\mathbf{q}_e \in \boldsymbol{\mathcal{P}}_{k}(e)$, we have
    \begin{gather*}
        \norm{\mathbf{q}_e}_{0,e}^2 \lesssim \norm{\sqrt{\Psi_e} \mathbf{q}_e}_{0,e}^2 \lesssim \norm{\mathbf{q}_e}_{0,e}^2, \quad \frac{1}{\sqrt{h_E}} \norm{\Psi_e \mathbf{q}_e}_{0,E} + \sqrt{h_E} |\Psi_e \mathbf{q}_e|_{1,E} \lesssim \norm{\mathbf{q}_e}_{0,e}, 
    \end{gather*}
    where $\mathbf{q}_e$ is also used to denote the constant prolongation of $\mathbf{q}_e$ in the direction normal to $e$.
\end{lemma}
\subsection{Error estimators}
We now define residual-based a posteriori error estimators for the VE formulation from Section~\ref{sec:vem}. First, we define the local error estimator $\Theta_E^2 := \Theta_{1,E}^2+\Theta_{2,E}^2$ for $E\in \mathcal{T}^h$ conformed by the elasticity contribution and the reaction-diffusion one. These are given explicitly by
\begin{align}\label{error-indicators}
    \Theta_{1,E}^2 := \frac{1}{2\mu}\Xi_{1,E}^2 + \frac{1}{2\mu}\eta_{1,E}^2 + 2\mu \Lambda_{1,E}^2 + S_{1,E}^2, \quad \Theta_{2,E}^2 := M\Xi_{2,E}^2 + M\eta_{2,E}^2 + M\Lambda_{2,E}^2 + S_{2,E}^2,
\end{align}
where $\Xi_{j,E}$ represents the local residual estimator, which encapsulates boundary and volume terms as
\begin{align*}
    &\Xi_{1,E}^2 :=  \, \sum_{e\in \mathcal{E}^h_{\mathrm{N}}(E)} h_e \norm{ (2\mu\varepsilon(\bPi_1^{\beps,k_1} \bu_h) - {p}_h \bbI )\bn_e}_{0,e}^2 + \sum_{e\in \mathcal{E}^h_\Omega(E)}  h_e \norm{ \jump{(2\mu\varepsilon(\bPi_1^{\beps,k_1} \bu_h)  - {p}_h \bbI )\bn_e} }_{0,e}^2 \\
    & \quad  + h_E^2\norm{\bPi_1^{0,k_1-2}\fb + 2\mu \bdiv(\beps(\bPi_1^{\beps,k_1} \bu_h)) - \nabla {p}_h}_{0,E}^2, \\
    &\Xi_{2,E}^2 := \sum_{e\in\mathcal{E}^h_{\mathrm{D}}(E)} h_e \norm{\varphi_{\mathrm{D}} - \varphi_h}_{0,e}^2 + \sum_{e\in\mathcal{E}^h_{\mathrm{D}}(E)} h_e \norm{(\nabla \varphi_{\mathrm{D}} - \bPi_2^{0,k_2}(\bbM^{-1}(\beps(\overline{\bu}_h),{p}_h) \bPi_2^{0,k_2} \bzeta_h) )\cdot \bt_e}_{0,e}^2 \\ 
    & \quad + h_E^2 \norm{\bPi_2^{0,k_2}(\bbM^{-1}(\beps(\overline{\bu}_h),{p}_h) \bPi_2^{0,k_2} \bzeta_h) - \nabla \varphi_h}_{0,E}^2 + h_E^2\norm{\vrot(\bPi_2^{0,k_2}(\bbM^{-1}(\beps(\overline{\bu}_h),{p}_h) \bPi_2^{0,k_2} \bzeta_h))}_{0,E}^2.
\end{align*}
The local data oscillation, mixed, and stabilisation estimators, are respectively given by 
\begin{gather*}
    \eta_{1,E}^2 := h_E^2 \norm{\fb - \bPi_{1}^{0,k_1-2}\fb}_{0,E}^2, \quad \eta_{2,E}^2 := \norm{\bbM^{-1}(\beps(\overline{\bu}_h),{p}_h)\bPi_2^{0,k_2}\bzeta_h - \bPi_2^{0,k_2}(\bbM^{-1}(\beps(\overline{\bu}_h),{p}_h)\bPi_2^{0,k_2}\bzeta_h)}_{0,E}^2, \\ 
    \Lambda_{1,E}^2 := \norm{\frac{1}{\lambda}\ell(\varphi_h) - \vdiv \bu_h + \frac{1}{\lambda} {p}_h}_{0,E}^2, \quad \Lambda_{2,E}^2 := \norm{-g - \vdiv \bzeta_h + \theta \varphi_h}_{0,E}^2,\\
    S_{1,E}^2 := S_1^E(\bu_h - \bPi_1^{\beps,k_1}\bu_h,\bu_h - \bPi_1^{\beps,k_1}\bu_h), \quad S_{2,E}^2 := S_2^E(\bzeta_h - \bPi_2^{0,k_2}\bzeta_h,\bzeta_h - \bPi_2^{0,k_2}\bzeta_h). 
\end{gather*}
Such definitions require a stronger assumption on the boundary term $\varphi_{\mathrm{D}}$ (see Lemma~\ref{residual-bound-2}). 
The global error, global residual, global mixed, global data oscillation, and global stabilisation estimators are defined as
\begin{gather*}
    \Theta^2 := \sum_{E\in \mathcal{T}^h} \Theta_E^2, \quad \Xi^2 := \sum_{E\in \mathcal{T}^h} \left(\frac{1}{2\mu}\Xi_{1,E}^2 + M \Xi_{2,E}^2\right), \quad
    \Lambda^2 := \sum_{E\in \mathcal{T}^h} \left(2\mu \Lambda_{1,E}^2 + M \Lambda_{2,E}^2\right), \\
    \eta^2 := \sum_{E\in \mathcal{T}^h} \left(\frac{1}{2\mu}\eta_{1,E}^2 + M \eta_{2,E}^2\right), \quad S^2 := \sum_{E\in \mathcal{T}^h} \left(S_{1,E}^2 + S_{2,E}^2\right).
\end{gather*}
Similar estimators can be found for the uncoupled divergence-free Stokes problem in \cite{wang20} and for the uncoupled mixed Poisson problem in \cite{munar2024}.

\subsection{Reliability} This subsection establishes an upper bound for the total error in terms of the error estimator. We begin by deriving estimates for the elasticity and reaction-diffusion partial errors in terms of a residual operator, pointing out that the non-linear terms satisfy the assumptions in Section~\ref{nonlinear-terms}.

Given the solution $(\bu_h,{p}_h,\bzeta_h,\varphi_h) \in \bV_1^{h,k_1}\times \mathrm{Q}_1^{h,k_1} \times \bV_2^{h,k_2} \times \mathrm{Q}_2^{h,k_2}$ to \eqref{eq:weak-discrete} and the fixed functions $\vartheta_h \in \mathrm{Q}_2^{h,k_2}$ and $(\bw_h,r_h)\in \bV_1^{h,k_1}\times \mathrm{Q}_1^{h,k_1}$,  we introduce the linear operators 
\begin{alignat*}{2}
  R_{F_1}(\bv) &:= F_1(\bv) - a_1(\bu_h,\bv) - b_1(\bv,{p}_h), \quad &\forall \bv\in \bV_1,\\
  R_{G_1^{\vartheta_h}}({q}) &:= G_1^{\vartheta_h}({q}) - b_1(\bu_h,{q}) + \frac{1}{\lambda}c_1({p}_h,{q}), \quad &\forall q\in \mathrm{Q}_1,\\
  R_{F_2}(\bxi) &:= F_2(\bxi) - a_2^{\overline{\bw}_h,r_h}(\bzeta_h,\bxi) - b_2(\bxi,\varphi_h), \quad &\forall \bxi\in \bV_2,\\ 
  R_{G_2}(\psi) &:= G_2(\psi) - b_2(\bzeta_h,\psi) + \theta c_2(\varphi_h,\psi), \quad &\forall \varphi \in \mathrm{Q}_2.
\end{alignat*}
\begin{lemma}\label{residual-elasticity}
Given $\vartheta \in \mathrm{Q}_2$ and $\vartheta_h \in \mathrm{Q}_2^{h,k_2}$, assume that $1\leq \lambda$ and $0<\mu$. Then, the following estimate holds:
    \begin{equation*}
        \norm{(\bu-\bu_h,{p}-{p}_h)}_{\bV_1\times \mathrm{Q}_1} \leq C_1\left\{ \norm{R_{F_1}}_{\bV'_1} + \norm{R_{G_1^{\vartheta_h}}}_{Q'_1} + \sqrt{M} L_\ell \norm{\vartheta-\vartheta_h}_{\mathrm{Q}_2} \right\}.
    \end{equation*}
\end{lemma}

\begin{proof}
    The global inf-sup from Theorem~\ref{th:global-inf-sup} applied to $(\bu-\bu_h,{p}-{p}_h)\in \bV_1 \times \mathrm{Q}_1$, the Lipschitz continuity of $\ell(\cdot)$, and a proper use of the term $G_1^{\vartheta_h}({q})$ lead to
    \begin{align*}
        \norm{(\bu-\bu_h,{p}-{p}_h)}_{\bV_1 \times \mathrm{Q}_1} &\leq  C_1 \sup_{(\bv,{q}) \in \bV_1\times \mathrm{Q}_1} \frac{\left|F_1(\bv)+G_1^{\vartheta}({q})-a_1(\bu_h,\bv)-b_1(\bv,{p}_h)-b_1(\bu_h,{q})+\frac{1}{\lambda}c_1({p}_h,{q})\right|}{\norm{(\bv,{q})}_{\bV_1 \times \mathrm{Q}_1}}\\
        &= C_1 \sup_{(\bv,{q})\in \bV_1\times \mathrm{Q}_1}\frac{\left|R_{F_1}(\bv)+ R_{G_1^{\vartheta_h}}({q}) +\left[G_1^{\vartheta}-G_1^{\vartheta_h}\right]({q})\right|}{\norm{(\bv,{q})}_{\bV_1 \times \mathrm{Q}_1}}\\
        & \leq C_1 \left\{ \norm{R_{F_1}}_{\bV'_1} + \norm{R_{G_1^{\vartheta_h}}}_{Q'_1} + \sqrt{M}L_\ell\norm{\vartheta-\vartheta_h}_{\mathrm{Q}_2}\right\}.
    \end{align*}
\end{proof}

\begin{lemma}\label{residual-diffusion}
     For the fixed pairs $(\bw,r)\in \bV_1\times \mathrm{Q}_1$ and $(\bw_h,r_h)\in \bV_1^{h,k_1}\times \mathrm{Q}_1^{h,k_1}$, suppose that $\theta \leq \frac{1}{M}$. Then, the following estimate holds:
    \begin{align*}
        &\norm{(\bzeta-\bzeta_h,\varphi-\varphi_h)}_{\bV_2\times \mathrm{Q}_2} \leq C_2 \biggl\{ \norm{R_{F_2}}_{\bV'_2} + \norm{R_{G_2}}_{Q'_2} + \overline{C}_3 \norm{(\bw-\bw_h,r-r_h)}_{\bV_1\times \mathrm{Q}_1} \biggr\}.
    \end{align*}
\end{lemma}

\begin{proof}
The proof proceeds in the same way as for Lemma~\ref{residual-elasticity}.
\end{proof}

We finalise with an upper bound for the total error $\overline{\textnormal{e}}_h$ in terms of the residuals defined in Lemmas~\ref{residual-elasticity}--\ref{residual-diffusion}. The following result is a consequence of the fixed point argument discussed in \cite{khot2024} and Theorem~\ref{th:global-inf-sup}.
\begin{theorem}\label{residual-total}
    Assume that $1\leq \lambda$, $0<\mu$, $\theta \leq \frac{1}{M}$. Furthermore, suppose that   $C_1 \sqrt{M} L_\ell + C_2 \overline{C}_3 < \frac{1}{2}$. For $\varphi\in \mathrm{W}$ and $\varphi_h \in \mathrm{W}^h$ (continuous and discrete balls defined in Theorems~\ref{well-posedness} and \ref{well-posedness-discrete}), the following estimate holds:
    $$\overline{\textnormal{e}}_h \lesssim \max\left\{C_1,C_2\right\} \left\{ \norm{R_{F_1}}_{\bV'_1} + \norm{R_{G_1^{\varphi_h}}}_{Q'_1} + \norm{R_{F_2}}_{\bV'_2} + \norm{R_{G_2}}_{Q'_2} \right\}.$$
\end{theorem}
Next, we aim to estimate locally the residual operators from Theorem~\ref{residual-total} in terms of \eqref{error-indicators}. First, the residuals of the elasticity problem given in Lemma~\ref{residual-elasticity} satisfy the following result.
\begin{lemma}\label{residual-bound-1}
 The following bound holds:
    \begin{equation*}
        \norm{R_{F_1}}_{\bV'_1}^2 + \norm{R_{G_1^{\varphi_h}}}_{Q'_1}^2 \lesssim  \sum_{E\in \mathcal{E}^h} \Theta_{1,E}^2.
    \end{equation*}
\end{lemma}
\begin{proof}
    Let $\bv \in \bV_1$ and  $\bv_h = \mathbf{I}_1^{Q,k_1}\bv$ satisfying Lemma~\ref{interpolation-estimates}. Then, the residual $R_{F_1}(\cdot)$ can be rewritten as    
    \begin{align}\label{residual_F1_rewrite}
        R_{F_1}(\bv) &= \sum_{E\in \mathcal{T}^h} \left( \left[F_1^E-F_1^{h,E}\right](\bv) + \left[a_1^{h,E}-a_1^E\right](\bu_h,\bv_h) + F_1^{h,E}(\bv-\bv_h) - a_1^E(\bu_h,\bv-\bv_h) - b_1^E(\bv-\bv_h,{p}_h)\right) \notag \\
        &=: \sum_{E\in \mathcal{T}^h} \left( \mathrm{T}_1^{1,E} + \mathrm{T}_1^{2,E} + \mathrm{T}_1^{3,E} + \mathrm{T}_1^{4,E} + \mathrm{T}_1^{5,E}\right).
    \end{align}
    Lemma~\ref{estimate-fb}, the continuity of the interpolation operator,and Lemma~\ref{a_1-error} imply that
    \begin{align}
        &\left|\sum_{E\in \mathcal{T}^h} \mathrm{T}_1^{1,E}\right| \lesssim \sum_{E\in \mathcal{T}^h} \frac{1}{\sqrt{2\mu}} \eta_{1,E} \norm{\bv}_{\bV_1(E)}, \quad
        \left|\sum_{E\in \mathcal{T}^h} \mathrm{T}_1^{2,E}\right|\lesssim \sum_{E\in \mathcal{T}^h} S_{1,E} \norm{\bv}_{\bV_1(E)} \label{bound-err-F_bound-err-a1}.
    \end{align}
    To address $\mathrm{T}_1^{3,E} + \mathrm{T}_1^{4,E} + \mathrm{T}_1^{5,E}$ we shall use the following integration by parts formulae 
    \begin{subequations}\begin{align}
        \int_E \beps(\bPi_1^{\beps,k_1} \bu_h):\beps(\bv-\bv_h) &= - \int_E \bdiv(\beps(\bPi_1^{\beps,k_1} \bu_h)) \cdot (\bv-\bv_h) + \int_{\partial E} \beps(\bPi_1^{\beps,k_1} \bu_h)\bn \cdot (\bv - \bv_h), \label{parts-beps} \\
        \int_E \vdiv(\bv-\bv_h) {p}_h &= - \int_E \nabla {p}_h \cdot (\bv - \bv_h) + \int_{\partial E} ({p}_h \bbI) \bn \cdot (\bv-\bv_h). \label{parts-div}
    \end{align}\end{subequations}
    The identities \eqref{parts-beps}-\eqref{parts-div}, Lemma~\ref{interpolation-estimates}, and a Cauchy--Schwarz inequality lead to
    \begin{align}\label{bound-last-part}
        &\biggl|\sum_{E\in \mathcal{T}^h} \mathrm{T}_1^{3,E} + \mathrm{T}_1^{4,E} + \mathrm{T}_1^{5,E}\biggr| = \biggl| -  \sum_{e\in \mathcal{E}^h_{\mathrm{N}}} \int_{e} (2\mu \beps(\bPi_1^{\beps,k_1} \bu_h) -{p}_h \bbI)\bn \cdot (\bv - \bv_h) - \sum_{E\in \mathcal{T}^h} a_1^E(\bu_h - \bPi_1^{\beps,k_1} \bu_h, \bv-\bv_h) \notag \\
        & \quad - \sum_{e\in \mathcal{E}^h_\Omega} \int_{e} (2\mu \beps(\bPi_1^{\beps,k_1} \bu_h)-{p}_h \bbI) \bn \cdot (\bv - \bv_h) + \sum_{E\in \mathcal{T}^h} \int_E ( \bPi_1^{0,k_1-2}\fb +  2\mu \bdiv(\beps(\bPi_1^{\beps,k_1} \bu_h)) - \nabla {p}_h ) \cdot (\bv-\bv_h)  \biggr| \notag \\
        & \lesssim  \sum_{E\in \mathcal{T}^h} ( \frac{1}{\sqrt{2\mu}} \Xi_{1,E} + \, S_{1,E}) \norm{\bv}_{\bV_1(E)}.
    \end{align}
    Finally, the residual $R_{G_1^{\varphi_h}}^E(\cdot)$ is handled using a Cauchy--Schwarz inequality as follows
    \begin{align}\label{bound-G1}
        \biggl| R_{G_1^{\varphi_h}}({q})\biggr| \lesssim \sum_{E\in \mathcal{T}^h} \sqrt{2\mu} \Lambda_{1,E} \norm{{q}}_{\mathrm{Q}_1(E)}.
    \end{align}
    Summing the estimates \eqref{bound-err-F_bound-err-a1}, \eqref{bound-last-part}, \eqref{bound-G1}, taking the supremum for all $\bv\in \bV_1$ and all ${q} \in \mathrm{Q}_1$, and applying the Cauchy--Schwarz inequality conclude the proof.
\end{proof}
\begin{remark}
   It is possible to construct a Fortin interpolation  for Stokes-like spaces. Thus, from the commutative property   applied to \eqref{parts-div}, one can eliminate $p$ from the estimator $\Xi_{1,E}$. However, the  momentum balance  \eqref{eq:linear-momentum} also drops from 
   the estimator, leading to convergence issues as the right-hand side $\boldsymbol{f}$ is not recovered. An alternative is proposed in \cite{Lederer2019}, where linear momentum conservation is ensured in a pressure-free formulation.
\end{remark}
Now, we will concentrate on the residuals for the reaction-diffusion problem defined in Lemma~\ref{residual-diffusion}.
\begin{lemma}\label{residual-bound-2}
    Assume that $\Omega \subset \mathbb{R}^2$ is a connected domain and that $\Gamma_{\mathrm{N}}$ is contained in the boundary of a convex part of $\Omega$, that is, there exists a convex domain $B$ such that $\overline{\Omega}\subseteq B$ and $\Gamma_{\mathrm{N}}\subseteq \partial B$. Assume further that $\varphi_{\mathrm{D}}\in \mathrm{H}^1(\partial \Omega)$. Then, the following bound holds:
     \[       \norm{R_{F_2}}_{\bV'_2}^2 + \norm{R_{G_2}}_{Q'_2}^2 \lesssim   \sum_{E\in \mathcal{T}^h} \Theta_{2,E}^2.\]
\end{lemma}
\begin{proof}
Let $\bxi \in \bV_2$. We construct  
$\bxi_h \in \bV_2^{h,k_2}$ combining the continuous Helmholtz decomposition (see \cite[Lemma 5.1]{Cascon07}) and the interpolators $\mathbf{I}_2^{F,k_2},\mathrm{I}_3^{Q,k_2+1}$. Indeed, there exist $\bz \in \bH^1(\Omega)$ and $\chi \in \mathrm{H}^1_{\mathrm{N}}(\Omega)$ such that
\begin{equation} \label{bound-decomposition}
    \bxi = \bz + \brot \chi \text{ in } \Omega \text{, with } \frac{1}{\sqrt{M}} \left( \norm{\bz}_{1,\Omega} + \norm{\chi}_{1,\Omega} \right) \lesssim \norm{\bxi}_{\bV_2}.
\end{equation} 
Therefore, we set $\bxi_h := \bz_h + \brot \chi_h$, where $\bz_h = \mathbf{I}_2^{F,k_2}\bz$ and $\chi_h = \mathrm{I}_3^{Q,k_2+1}\chi$. Note that Lemma~\ref{interpolation-estimates} and Lemma~\ref{fortin-auxiliary} turn the bound in \eqref{bound-decomposition} into
\begin{align} \label{bound_xi_h}
    \frac{1}{\sqrt{M}} \sum_{E\in \mathcal{T}^h} \norm{\bxi_h}_{0,E} &\leq \frac{1}{\sqrt{M}} \sum_{E\in \mathcal{T}^h} \left(\norm{\bz - \bz_h}_{0,E} + \norm{\bz}_{0,E} + |\chi - \chi_h|_{1,E} + |\chi|_{1,E}\right) \notag \\
    &\lesssim \frac{1}{\sqrt{M}} \sum_{E\in \mathcal{T}^h} \left(\norm{\bz}_{1,E} + \norm{\chi}_{1,E}\right) \lesssim \norm{\bxi}_{\bV_2}.
\end{align}
We can assert that 
\begin{equation}\label{bound-helmhotz}
    \bxi - \bxi_h = \bz - \bz_h + \brot(\chi-\chi_h)\text{ in } \Omega, \quad  \text{with}\quad  \frac{1}{\sqrt{M}} \sum_{E\in \mathcal{T}^h} \norm{\bxi-\bxi_h}_{0,E} \lesssim \norm{\bxi}_{\bV_2}.
\end{equation}
Next, given $E\in \mathcal{T}^h$, we rewrite the residual of $R_{F_2}(\cdot)$ as 
\begin{align*}
    R_{F_2}(\bxi) & = \sum_{E\in \mathcal{T}^h} \left( \left[a_2^{\overline{\bu}_h,{p}_h,h,E}-a_2^{\overline{\bu}_h,{p}_h,E}\right](\bzeta_h, \bxi_h) + F_2^E(\bxi - \bxi_h) -  a_2^{\overline{\bu}_h,{p}_h,E}(\bzeta_h, \bxi-\bxi_h) - b_2^E(\bxi - \bxi_h, \varphi_h) \right)\\
    & =: \sum_{E\in \mathcal{T}^h} \left( \mathrm{T}_2^{1,E} + \mathrm{T}_2^{2,E} + \mathrm{T}_2^{3,E} + \mathrm{T}_2^{4,E} \right). 
\end{align*}
For $\mathrm{T}_2^{1,E}$, Lemma~\ref{a_2-error} and \eqref{bound_xi_h} imply that
\begin{align}\label{bound_error_a_2}
    \biggl|\sum_{E\in \mathcal{T}^h} \mathrm{T}_2^{1,E}\biggr| \lesssim \sum_{E\in \mathcal{T}^h} \left(\sqrt{M} \eta_{2,E} + S_{2,E}\right) \norm{\bxi}_{\bV_2(E)}.
\end{align}
For $\mathrm{T}_2^{2,E}$, substituting \eqref{bound-helmhotz} and applying \cite[Lemma 3.5]{dominguez15} for $\chi-\chi_h \in \mathrm{H}^1_{\mathrm{N}}(\Omega)$ and a suitable extension of the data $\varphi_{\mathrm{D}}\in \mathrm{H}^1(\Omega)$ such that $\varphi_{\mathrm{D}}=0$ on $\Gamma_{\mathrm{N}}$ (see \cite[Lemma 2.2]{galvis07}), we obtain
\begin{align*}
    \sum_{E\in \mathcal{T}^h} \mathrm{T}_2^{2,E} & = \sum_{E\in \mathcal{T}^h} \langle (\bz-\bz_h)\cdot \bn_e, \varphi_{\mathrm{D}} \rangle_{\partial E\cap \Gamma_{\mathrm{D}}} + \sum_{E\in \mathcal{T}^h} \langle \brot(\chi-\chi_h)\cdot \bn_e, \varphi_{\mathrm{D}} \rangle_{\partial E\cap \Gamma_{\mathrm{D}}} \notag\\
    & = \sum_{E\in \mathcal{T}^h} \langle (\bz-\bz_h)\cdot \bn_e, \varphi_{\mathrm{D}} \rangle_{\partial E\cap \Gamma_{\mathrm{D}}} - \sum_{E\in \mathcal{T}^h} \langle \chi-\chi_h, \nabla \varphi_{\mathrm{D}} \cdot \bt_e \rangle_{\partial E\cap \Gamma_{\mathrm{D}}}. 
\end{align*}
For $\mathrm{T}_2^{3,E}$, the addition and subtraction of the terms $\bPi_2^{0,k_2}(\bbM^{-1}(\beps(\overline{\bu}_h),{p}_h) \bPi_2^{0,k_2} \bzeta_h)$ and $\bbM^{-1}(\beps(\overline{\bu}_h),{p}_h) \bPi_2^{0,k_2} \bzeta_h$, and for $\mathrm{T}_2^{4,E}$ an application of integration by parts lead to
\begin{align*}
    & \sum_{E\in \mathcal{T}^h} \mathrm{T}_2^{3,E} = \sum_{E\in \mathcal{T}^h} \biggl( \int_E \bPi_2^{0,k_2}(\bbM^{-1}(\beps(\overline{\bu}_h),{p}_h) \bPi_2^{0,k_2} \bzeta_h) \cdot (\bxi -\bxi_h) + \int_E \bbM^{-1}(\beps(\overline{\bu}_h),{p}_h)(\bzeta_h-\bPi_2^{0,k_2} \bzeta_h) \cdot (\bxi - \bxi_h) \notag \\ 
    & \quad + \int_E \left[ \bbM^{-1}(\beps(\overline{\bu}_h),{p}_h) \bPi_2^{0,k_2} \bzeta_h - \bPi_2^{0,k_2}(\bbM^{-1}(\beps(\overline{\bu}_h),{p}_h) \bPi_2^{0,k_2} \bzeta_h) \right] \cdot (\bxi - \bxi_h)\biggr), \\
    & \sum_{E\in \mathcal{T}^h} \mathrm{T}_2^{4,E} = - \sum_{E\in \mathcal{T}^h} \int_{E} \varphi_h \vdiv(\bxi - \bxi_h) = \sum_{E\in \mathcal{T}^h} \int_{E} \nabla \varphi_h \cdot (\bxi - \bxi_h) - \sum_{e\in \mathcal{E}^h}\int_{e} \varphi_h (\bxi - \bxi_h)\cdot \bn_e. 
\end{align*}
Thus, integration by parts, the Cauchy--Schwarz inequality, Lemma~\ref{fortin-auxiliary}, Lemma~\ref{interpolation-estimates}, and \eqref{stabilisation-reaction-diff} imply that
\begin{align}\label{bound_three_terms_splitted}
    &\biggl| \sum_{E\in \mathcal{T}^h} \mathrm{T}_2^{2,E} + \mathrm{T}_2^{3,E} + \mathrm{T}_2^{4,E} \biggr|  
     = \biggl|\sum_{e\in \mathcal{E}^h_{\mathrm{D}}} \langle (\bz-\bz_h)\cdot \bn_e, \varphi_{\mathrm{D}}-\varphi_h \rangle_{e} - \sum_{e\in \mathcal{E}^h_{\mathrm{D}}} \langle \chi-\chi_h, \nabla \varphi_{\mathrm{D}} \cdot \bt_e \rangle_{e} \notag \\ 
    & \quad +  \sum_{e\in \mathcal{E}^h} \langle \chi - \chi_h, \bPi_2^{0,k_2}(\bbM^{-1}(\beps(\overline{\bu}_h),{p}_h) \bPi_2^{0,k_2} \bzeta_h) \cdot \bt_e \rangle_e \notag \\
    & \quad - \sum_{E\in \mathcal{T}^h} \int_E  \left[ \bPi_2^{0,k_2}(\bbM^{-1}(\beps(\overline{\bu}_h),{p}_h) \bPi_2^{0,k_2} \bzeta_h)-\nabla \varphi_h \right] \cdot (\bz-\bz_h) \notag \\
    & \quad - \sum_{E\in \mathcal{T}^h} \int_E \vrot\left(\bPi_2^{0,k_2}(\bbM^{-1}(\beps(\overline{\bu}_h),{p}_h) \bPi_2^{0,k_2} \bzeta_h)\right) (\chi - \chi_h) \notag \\ 
    & \quad - \sum_{E\in \mathcal{T}^h} \int_E \bbM^{-1}(\beps(\overline{\bu}_h),{p}_h)(\bzeta_h-\bPi_2^{0,k_2} \bzeta_h) \cdot (\bxi - \bxi_h) \notag \\
    & \quad - \sum_{E\in \mathcal{T}^h} \int_E \left[ \bbM^{-1}(\beps(\overline{\bu}_h),{p}_h) \bPi_2^{0,k_2} \bzeta_h - \bPi_2^{0,k_2}(\bbM^{-1}(\beps(\overline{\bu}_h),{p}_h) \bPi_2^{0,k_2} \bzeta_h) \right] \cdot (\bxi - \bxi_h) \biggr| \notag \\
    &\lesssim \sum_{E\in \mathcal{T}^h} (\sqrt{M} \Xi_{2,E} + \, S_{2,E} + \, \sqrt{M}\eta_{2,E}) \norm{\bxi}_{\bV_2(E)}.
\end{align}
Finally, concerning the residual $R_{G_2}(\cdot)$, the Cauchy--Schwarz inequality and proper scaling give us that
\begin{align}\label{bound-G2}
    \biggl| R_{G_2}(\psi)\biggr| \lesssim \sum_{E\in \mathcal{T}^h} \sqrt{M} \Lambda_{2,E} \norm{\psi}_{\mathrm{Q}_2(E)}.
\end{align}
Adding  \eqref{bound_error_a_2}, \eqref{bound_three_terms_splitted}, and \eqref{bound-G2}, taking the supremum for all $\bxi\in \bV_2$ and all $\psi \in \mathrm{Q}_2$, and applying the Cauchy--Schwarz inequality complete the proof.
\end{proof}
To finalise, the following reliability of the total residual global estimator $\Theta$ is a direct consequence of Lemmas~\ref{residual-bound-1}-\ref{residual-bound-2}.
\begin{theorem}\label{th:upper-bound}
    Under the assumption of Theorem~\ref{residual-total} and Lemma~\ref{residual-bound-2}, the following bound holds
    $$\overline{\textnormal{e}}_h \lesssim \max \left\{ C_1, C_2 \right\} \Theta.$$
\end{theorem}
\subsection{Efficiency}\label{sec:efficiency}
This section aims to show the efficiency of the global residual and mixed estimators $\Xi$ and $\Lambda$  
up to data oscillation $\eta$, the stabilisation estimator $S$ and higher-order terms 
($\mathrm{HOTs}$). In what follows, we use the properties of bubble functions, the Lipschitz continuity of the nonlinear terms from Section~\ref{nonlinear-terms}, and the strong mixed formulation in \eqref{eq:mixed-formulation}. The main result is then a consequence of Lemmas~\ref{efficiency-elasticity}-\ref{efficiency-reaction-diffusion}.
\begin{lemma}\label{efficiency-elasticity}
    The following bound holds 
    \begin{align*}
        \sum_{E\in \mathcal{T}^h} \left(\frac{1}{2\mu} \Xi_{1,E}^2 + 2\mu \Lambda_{1,E}^2\right) \lesssim \max \left\{ 1, 2\mu M \right\} \sum_{E\in \mathcal{T}^h} \left( \overline{\textnormal{e}}_h^2 + \frac{1}{2\mu}\eta_{1,E}^2 + S_{1,E}^2 \right).
    \end{align*}
\end{lemma}
\begin{proof}
    Let $E\in \mathcal{T}^h$, $\mathbf{q}_E := \bPi_1^{0,k_1-2}\fb + 2\mu \bdiv(\beps(\bPi_1^{\beps,k_1} \bu_h)) - \nabla {p}_h \in \boldsymbol{\mathcal{P}}_{k_1-2}(E)$, and $\overline{\bv}_E := \Psi_E \mathbf{q}_E$ where $\Psi_E$ is an element bubble function from Lemma~\ref{bubble-int}. Note that \eqref{weak-1} leads to $R_{F_1}(\bv) = a_1(\bu-\bu_h,\bv)+b_1(\bv,{p}-{p}_h)$ for all $\bv\in \bV_1$. Then, we set $\bv=\overline{\bv}_E$ and $\bv_h = \bzero$ in \eqref{residual_F1_rewrite}. Thus, given that $\overline{\bv}_E$ vanishes on $\partial E$ and outside $E$, we can apply  integration by parts to arrive at 
    \begin{align*}
        &\norm{\mathbf{q}_E}_{0,E}^2 \lesssim \left| \int_E \mathbf{q}_E \cdot \overline{\bv}_E \right| =  \left| a_1^E(\bu-\bu_h,\overline{\bv}_E) + b_1^E(\overline{\bv}_E,{p}-{p}_h) -  \left[F_1^E-F_1^{h,E} \right](\overline{\bv}_E) + a_1^E(\bu_h - \bPi_1^{\beps,k_1} \bu_h, \overline{\bv}_E) \right|.
    \end{align*}
    Next, the continuity of $a_1(\cdot,\cdot)$ and $b_1(\cdot,\cdot)$ together with the Cauchy--Schwarz inequality and \eqref{stabilisation-elast} lead to
    \begin{align*}
        &\norm{\mathbf{q}_E}_{0,E}^2 \lesssim (\norm{\bu-\bu_h}_{\bV_1(E)} + \norm{{p}-{p}_h}_{\mathrm{Q}_1(E)}+ S_{1,E}) \norm{\overline{\bv}_E}_{\bV_1(E)} + \norm{\fb - \bPi_{1}^{0,k_1-2}\fb}_{0,E} \norm{\overline{\bv}_E}_{0,E}.
    \end{align*}
    From the inequality $\norm{\beps(\cdot)}_{0,E}\leq |\cdot|_{1,E}$ and Lemma~\ref{bubble-int} we obtain
    \begin{align*}
        &\norm{\mathbf{q}_E}_{0,E}^2 \lesssim \frac{\sqrt{2\mu}}{h_E} (\norm{\bu-\bu_h}_{\bV_1(E)} + \norm{{p}-{p}_h}_{\mathrm{Q}_1(E)}+ S_{1,E}) \norm{\mathbf{q}_E}_{0,E} + \norm{\fb - \bPi_{1}^{0,k_1-2}\fb}_{0,E} \norm{\mathbf{q}_E}_{0,E},
    \end{align*}
    and consequently, we have
    \begin{align} \label{bound-q_E}
        \frac{1}{2\mu}h_E^2\norm{\mathbf{q}_E}_{0,E}^2 \lesssim \left(\norm{\bu-\bu_h}_{\bV_1(E)}^2 + \norm{{p}-{p}_h}_{\mathrm{Q}_1(E)}^2 + S_{1,E}^2 + \frac{1}{2\mu}\eta_{1,E}^2\right).
    \end{align}
    Now, we define the edge polynomial $\mathbf{q}_e = \jump{(2\mu\varepsilon(\bPi_1^{\beps,k_1} \bu_h) -{p}_h \bbI)\bn_e} \in \boldsymbol{\mathcal{P}}_{k_1-1}(e)$ and  the edge bubble function $\overline{\bv}_e=\Psi_e \mathbf{q}_e$  as in Lemma~\ref{bubble-inner}. Note that, this polynomial can be extended to $\mathcal{T}^h_e$ by the techniques used in \cite[Remark 3.1]{MORA2017}. Choosing $\bv = \overline{\bv}_e$ and $\bv_h = \boldsymbol{0}$ in \eqref{residual_F1_rewrite}, we readily see that $\norm{\mathbf{q}_e}_{0,e}^2 \lesssim \left|\int_e \mathbf{q}_e \cdot \overline{\bv}_e \right|$, and
    \begin{align*}
        &\biggl|\int_e \mathbf{q}_e \cdot \overline{\bv}_e \biggr| \leq \sum_{E\in \mathcal{T}^h_e}  \biggl| \int_E \mathbf{q}_E \cdot \overline{\bv}_e + \left[F_1^E-F_1^{h,E} \right](\overline{\bv}_e) - a_1^E(\bu_h-\bPi_1^{\beps,k_1}\bu_h,\overline{v}_e) -a_1^E(\bu-\bu_h,\overline{\bv}_e) - b_1^E(\overline{\bv}_e,{p}-{p}_h) \biggr|.
    \end{align*}
    Similarly, note that
    \begin{align*}
        & \norm{\mathbf{q}_e}_{0,e}^2 \lesssim \sum_{E\in\mathcal{T}^h_e} \left(\norm{\bu-\bu_h}_{\bV_1(E)} + \norm{{p}-{p}_h}_{\mathrm{Q}_1(E)}+ S_{1,E}\right) \norm{\overline{\bv}_e}_{\bV_1(E)} + \left(\norm{\fb - \bPi_{1}^{0,k_1-2}\fb}_{0,E} + \norm{\mathbf{q}_E}_{0,E}\right) \norm{\overline{\bv}_e}_{0,E},
    \end{align*}
    which together with Lemma~\ref{bubble-inner} and the fact that $h_e\leq h_E$, imply that
    \begin{align*}
        \sqrt{h_e} \norm{\mathbf{q}_e}_{0,e} \lesssim \sum_{E\in \mathcal{T}^h_e} \biggl[ &\sqrt{2\mu} (\norm{\bu-\bu_h}_{\bV_1(E)} + \norm{{p}-{p}_h}_{\mathrm{Q}_1(E)}+ S_{1,E})  + h_E (\norm{\fb - \bPi_{1}^{0,k_1-2}\fb}_{0,E} + \norm{\mathbf{q}_E}_{0,E})\biggr].
    \end{align*}
    Therefore, 
    \begin{align}
        &\frac{1}{2\mu} h_e \norm{\mathbf{q}_e}_{0,e}^2 \lesssim  \sum_{E\in \mathcal{T}^h_e} \left[ \norm{\bu-\bu_h}_{\bV_1(E)}^2 + \norm{{p}-{p}_h}_{\mathrm{Q}_1(E)}^2 + S_{1,E}^2  \quad + \frac{1}{2\mu}\eta_{1,E}^2 + \frac{1}{2\mu}h_E^2 \norm{\mathbf{q}_E}_{0,E} \right].\label{bound-q_e}
    \end{align}
    Finally, from \eqref{eq:hooke-law}, the Lipschitz continuity of $\ell(\cdot)$, the inequality $\norm{\vdiv(\cdot)}_{0,E}\leq |\cdot|_{1,E}$, K\"orn's inequality, and given that $\frac{1}{\lambda} \leq 1$, we easily see that 
    \begin{align}\label{bound-poly-appx}
        \sqrt{2\mu}\Lambda_{1,E} &= \sqrt{2\mu}\norm{\frac{1}{\lambda}(\ell(\varphi)-\ell(\varphi_h)) + \vdiv(\bu-\bu_h)  + \frac{1}{\lambda}({p}-{p}_h)}_{0,E} \notag \\ &\lesssim \max \left\{ \sqrt{2\mu M}, 1 \right\} ( \norm{\varphi-\varphi_h}_{\mathrm{Q}_2(E)} + \norm{\bu - \bu_h}_{\bV_1(E)} + \norm{{p}-{p}_h}_{\mathrm{Q}_1(E)}).
    \end{align}
    Summing  the bounds in \eqref{bound-q_E}-\eqref{bound-poly-appx} for all $E$, and \eqref{bound-q_e} for all $e\subset \partial E$ concludes the proof.
\end{proof}

\begin{lemma}\label{efficiency-reaction-diffusion} The following bound holds
    \begin{align*}
        \sum_{E\in \mathcal{T}^h} M\Xi_{2,E}^2 + M\Lambda_{2,E}^2 & \lesssim \max \left\{ M^2, M \overline{C}_3 \right\} \sum_{E\in \mathcal{T}^h} ( \overline{\textnormal{e}}_h^2 + M\eta_{2,E}^2 + S_{2,E}^2 + S_{1,E}^2 + \mathrm{HOTs} ),
    \end{align*}
with the high-order terms given by $\mathrm{HOTs} := h_E^2 \left[ \overline{C}_3 \left(\norm{\bu-\bu_h}_{\bV_1(E)}^2 + \norm{{p}-{p}_h}_{\mathrm{Q}_1(E)}^2 + S_{1,E}^2 \right) + S_{2,E}^2 + \eta_{2,E}^2 \right] + \sum_{e\in \mathcal{E}^h(E)} h_e \norm{\nabla \varphi_\mathrm{D} \cdot \bt_e - \Pi_{2,e}^{0,k_2} \left(\nabla \varphi_\mathrm{D} \cdot \bt_e\right)}_{0,e}^2.$
\end{lemma}
\begin{proof}
    We start by defining $\mathbf{q}_E = \bPi_2^{0,k_2}(\bbM^{-1}(\beps(\overline{\bu}_h),{p}_h) \bPi_2^{0,k_2} \bzeta_h) - \nabla \varphi_h \in \boldsymbol{\mathcal{P}}_{k_2}(E)$, and $\overline{\bxi}_E = \Psi_E \mathbf{q}_E$ with a given element bubble function $\Psi_E$. Lemma~\ref{bubble-inner}, and the definition of diffusive-flux in \eqref{eq:diffusive-flux} lead to
    \begin{align*}
        & \norm{\mathbf{q}_E}_{0,E}^2 \lesssim \biggl|\int_E \overline{\bxi}_E \cdot \mathbf{q}_E \biggr| = \biggl| \int_E \overline{\bzeta}_E \cdot \biggl( \nabla(\varphi_h-\varphi) + \bbM^{-1}(\beps(\bu),{p}) \bzeta - \bPi_2^{0,k_2}(\bbM^{-1}(\beps(\bu_h),{p}_h)\bPi_2^{0,k_2} \bzeta_h) \biggr) \biggr|.
    \end{align*}
    The Cauchy--Schwarz inequality, Lemma~\ref{extra_estimates_error_M}, an integration by parts, and Lemma~\ref{bubble-inner} show that
    \begin{align}\label{4.29}
        M h_E^2 \norm{\mathbf{q}_E}_{0,E}^2 \lesssim  M^2 (\norm{\varphi-\varphi_h}_{\mathrm{Q}_2(E)}^2 + \mathrm{HOTs}).
    \end{align}
    Similarly, let $q_E = \vrot(\bPi_2^{0,k_2}(\bbM^{-1}(\beps(\overline{\bu}_h),{p}_h) \bPi_2^{0,k_2} \bzeta_h)) \in \mathcal{P}_{k_2-1}(E)$, and $\overline{\xi}_E = \Psi_E q_E$.  Lemma~\ref{bubble-inner},  the observation that $\vrot(\bbM^{-1}(\beps(\bu),{p})\bzeta) = \vrot(\nabla \varphi) = 0$, and an integration by parts imply that
    \begin{align*}
        \norm{q_E}_{0,E}^2 \lesssim \left| \int_E \overline{\xi}_E q_E  \right| &= \left| \int_E \brot(\overline{\xi}_E) \cdot \left(\bPi_2^{0,k_2}(\bbM^{-1}(\beps(\overline{\bu}_h),{p}_h) \bPi_2^{0,k_2} \bzeta_h) - \bbM^{-1}(\beps(\bu),{p})\bzeta\right) \right| \\ &\lesssim \norm{\bPi_2^{0,k_2}(\bbM^{-1}(\beps(\overline{\bu}_h),{p}_h) \bPi_2^{0,k_2} \bzeta_h) - \bbM^{-1}(\beps(\bu),{p})\bzeta}_{0,E} |\overline{\xi}_E|_{1,E}.
    \end{align*}
    This and Lemma~\ref{extra_estimates_error_M} result in
    \[
        M h_E^2 \norm{q_E}_{0,E}^2 \lesssim M \max \left\{ 1, \overline{C}_3 \right\} \bigl(\norm{\bu-\bu_h}_{\bV_1(E)}^2 + \norm{{p}-{p}_h}_{\mathrm{Q}_1(E)}^2 +  M\eta_{2,E}^2 + \sum_{i=1}^{2} S_{i,E}^2\bigr).
    \]
Now, for  $e\in \mathcal{E}^h(E)$, we  define $q_e = \jump{\bPi_2^{0,k_2}(\bbM^{-1}(\beps(\overline{\bu}_h),{p}_h) \bPi_2^{0,k_2} \bzeta_h)\cdot \bt_e} \in \mathcal{P}_{k_2}(e)$, and $\overline{\xi}_e = \Psi_e q_e$, with a given edge bubble function $\Psi_e$. Notice that $\jump{\bbM^{-1}(\beps(\bu),{p})\bzeta\cdot \bt_e} = 0$ for all $e$. Then, from Lemma~\ref{bubble-inner} and integration by parts, we readily see that
    \begin{align*}
        \norm{q_e}_{0,e}^2 \lesssim \left| \int_e \overline{\xi}_e q_e \right| &= \biggl| \sum_{E\in \mathcal{T}^h_e} \int_E \brot(\overline{\xi}_e) \cdot \left(\bPi_2^{0,k_2}(\bbM^{-1}(\beps(\overline{\bu}_h),{p}_h) \bPi_2^{0,k_2} \bzeta_h) - \bbM^{-1}(\beps(\bu),{p})\bzeta)\right) \\ 
        & \quad + \sum_{E\in \mathcal{T}^h_e} \int_E \overline{\xi}_e \vrot(\bPi_2^{0,k_2}(\bbM^{-1}(\beps(\overline{\bu}_h),{p}_h) \bPi_2^{0,k_2} \bzeta_h) \biggr|\\
        & \lesssim \sum_{E\in \mathcal{T}^h_e} \norm{\bPi_2^{0,k_2}(\bbM^{-1}(\beps(\overline{\bu}_h),{p}_h) \bPi_2^{0,k_2} \bzeta_h) - \bbM^{-1}(\beps(\bu),{p})\bzeta}_{0,E} |\overline{\xi}_e|_{1,E} \\
        & \quad + \sum_{E\in \mathcal{T}^h_e} \norm{\vrot(\bPi_2^{0,k_2}(\bbM^{-1}(\beps(\overline{\bu}_h),{p}_h) \bPi_2^{0,k_2} \bzeta_h)}_{0,E} \norm{\overline{\xi}_e}_{0,E}.
    \end{align*}
    Hence, Lemma~\ref{extra_estimates_error_M} implies that 
    \begin{align}
        &M h_e\norm{q_e}_{0,e} \lesssim 
        M \max\left\{ 1, \overline{C}_3 \right\} \left(\norm{\bu-\bu_h}_{\bV_1(E)}^2 + \norm{{p}-{p}_h}_{\mathrm{Q}_1(E)}^2 +  M\eta_{2,E}^2 + \sum_{i=1}^{2} S_{i,E}^2 + M h_E^2 \norm{q_E}_{0,E}^2\right).\label{4.31}
    \end{align}
    The trace inequality  yields  $\norm{\varphi_{\mathrm{D}} - \varphi_h}_{0,e}^2 = \norm{\varphi - \varphi_h}_{0,e}^2 \lesssim h_e^{-1} \norm{\varphi - \varphi_h}_{0,E}^2 + h_e |\varphi - \varphi_h|_{1,e}$, and consequently,
    \begin{align*}
        \norm{\varphi_{\mathrm{D}} - \varphi_h}_{0,e}^2 & \lesssim h_e^{-1} \norm{\varphi-\varphi_h}_{0,e} + h_E \norm{\bPi_2^{0,k_2}(\bbM^{-1}(\beps(\overline{\bu}_h),{p}_h) \bPi_2^{0,k_2} \bzeta_h) - \nabla \varphi_h}_{0,E}^2 \\
        & \quad + h_E \norm{\bPi_2^{0,k_2}(\bbM^{-1}(\beps(\overline{\bu}_h),{p}_h) \bPi_2^{0,k_2} \bzeta_h) - \bbM^{-1}(\beps(\bu),{p})\bzeta}_{0,E}.
    \end{align*}
    This and a scaling argument with the inclusion of $\mathrm{HOTs}$ lead to the following bound
    \begin{align*}
        M h_e \norm{\varphi_{\mathrm{D}} - \varphi_h}_{0,e}^2 \lesssim  M^2  (\norm{\varphi-\varphi_h}_{\mathrm{Q}_2(E)} + M h_E^2\norm{\mathbf{q}_E}_{0,E} + \mathrm{HOTs}).
    \end{align*}
    In addition, we define $q_e^*= \Pi_{2,e}^{0,k_2}(\nabla \varphi_{\mathrm{D}}\cdot \bt_e) - \bPi_2^{0,k_2}(\bbM^{-1}(\beps(\overline{\bu}_h),{p}_h) \bPi_2^{0,k_2} \bzeta_h) \cdot \bt_e \in \mathcal{P}_{k_2}(e)$, where $\Pi_{2,e}^{0,k_2}$ is defined as the polynomial projection on the edge $e$, and $\overline{\xi}_e^* = \Psi_e q_e^*$. Lemma~\ref{bubble-inner}, $\nabla \varphi \cdot \bt_e = \nabla \varphi_{\mathrm{D}} \cdot \bt_e$ on $\Gamma_{\mathrm{D}}$, \eqref{eq:diffusive-flux}, and an integration by parts imply that $\norm{q_e^*}_{0,e}^2 \lesssim \left| \int_e  \overline{\xi}_e^* q_e^* \right|$. Furthermore,
    \begin{align*}
        \left| \int_e  \overline{\xi}_e^* q_e^* \right| & = \biggl| \sum_{E\in \mathcal{T}^h_e} \int_E  \brot(\overline{\xi}_e^*) \cdot (\bPi_2^{0,k_2}(\bbM^{-1}(\beps(\overline{\bu}_h),{p}_h) \bPi_2^{0,k_2} \bzeta_h) - \bbM^{-1}(\beps(\bu),{p})\bzeta) \\ 
        & \quad  + \sum_{E\in \mathcal{T}^h_e} \int_E \overline{\xi}_e^* \vrot(\bPi_2^{0,k_2}(\bbM^{-1}(\beps(\overline{\bu}_h),{p}_h) \bPi_2^{0,k_2} \bzeta_h)) + \int_e \overline{\xi}_e^* \left(\nabla \varphi_{\rm{D}} \cdot \bt_e - \Pi_{2,e}^{0,k_2} \left(\nabla \varphi_{\rm{D}} \cdot \bt_e\right) \right) \biggr| \\
        & \lesssim \norm{\bPi_2^{0,k_2}(\bbM^{-1}(\beps(\overline{\bu}_h),{p}_h) \bPi_2^{0,k_2} \bzeta_h) - \bbM^{-1}(\beps(\bu),{p})\bzeta}_{0,E} |\overline{\xi}_e^*|_{1,E} \\ 
        & \quad + \norm{\vrot(\bPi_2^{0,k_2}(\bbM^{-1}(\beps(\overline{\bu}_h),{p}_h) \bPi_2^{0,k_2} \bzeta_h))}_{0,E} \norm{\overline{\xi}_e^*}_{0,E} + \norm{\nabla \varphi_{\rm{D}} \cdot \bt_e - \Pi_{2,e}^{0,k_2} (\nabla \varphi_{\rm{D}} \cdot \bt_e)}_{0,e}\norm{\overline{\xi}_e^*}_{0,e},
    \end{align*}
    which results in the following bound
    \begin{align}
        &M h_e \norm{q_e}_{0,e}^2 \lesssim M \max\left\{1,\overline{C}_3\right\} \bigl(\norm{\bu-\bu_h}_{\bV_1(E)}^2 + \norm{{p}-{p}_h}_{\mathrm{Q}_1(E)}^2 +  M \eta_{2,E}^2 + \sum_{i=1}^{2} S_{i,E}^2 + M h_E^2 \norm{q_E}_{0,E}^2 + \mathrm{HOTs}\bigr).\label{4.33}
    \end{align}
    Finally, the equation \eqref{eq:reaction-diffusion} shows that
    \begin{align}
        & \sqrt{M}\Lambda_{2,E} = \sqrt{M}\norm{\vdiv(\bzeta-\bzeta_h)  + \theta(\varphi-\varphi_h)}_{0,E} \lesssim M \left( \norm{\bzeta-\bzeta_h}_{\bV_2(E)} + \norm{\varphi-\varphi_h}_{\mathrm{Q}_2(E)}\right).\label{4.34}
    \end{align}
    Summing for all $e\in \partial E$ \eqref{4.31}-\eqref{4.33} and for all $E\in \mathcal{T}^h$ \eqref{4.29}-\eqref{4.34} concludes the proof.
\end{proof}
The efficiency of the residual and mixed estimators (up to data oscillation  and stabilisation) is summarised below, and we  recall that it is a direct consequence of Lemmas~\ref{efficiency-elasticity}-\ref{efficiency-reaction-diffusion}.

\begin{theorem}\label{th:lower-bound} Under the assumptions of Theorem~\ref{residual-total}, the following bound holds 
\[\Xi^2 + \Lambda^2 \lesssim \max \left\{  2\mu M, M^2, M\overline{C}_3 \right\} (\overline{\textnormal{e}}_h^2 + \eta^2 + S^2+ \mathrm{HOTs}).\] 
\end{theorem}
\begin{remark}
    Notice that Theorem~\ref{th:lower-bound} can be extended to a parameter-free bound for the regime $\mu\geq\frac{1}{2}$. Indeed, the small data assumption \eqref{small_data_disc1} implies that
    \begin{align*}
        \max\left\{\frac{1}{2\mu},2\mu\right\}M^4 = 2\mu M^4 \leq  \frac{1}{C_*^2}, \text{ with } C_* = \overline{C}_1\overline{C}_2^2 L_\ell L_\bbM \left(\norm{\varphi_{\mathrm{D}}}_{\frac{1}{2},\Gamma_{\mathrm{D}}} + \norm{g}_{0,\Omega}\right).
    \end{align*}
    This observation together with the assumption $M\geq 1$ (from Section~2) readily shows that
    \begin{align*}
        \max\left\{2\mu M,M^2,M\overline{C}_3\right\}  \leq  \max\left\{1, C^*\right\}2\mu M^4 \leq \max\left\{\frac{1}{C_*^2},\frac{C^*}{C_*^2}\right\},
    \end{align*}
    with $C^*=L_\bbM \overline{C}_2 \left(\norm{\varphi_{\mathrm{D}}}_{\frac{1}{2},\Gamma_{\mathrm{D}}} + \norm{g}_{0,\Omega}\right)$. Noticing that both the constants $C_*$ and $C^*$ do not depend on parameters, the (parameter-free) efficiency reads 
    \begin{equation*}
        \Xi^2 + \Lambda^2 \lesssim \max \left\{ \frac{1}{C_*^2},\frac{C^*}{C_*^2}\right\} (\overline{\textnormal{e}}_h^2 + \eta^2 + S^2+ \mathrm{HOTs}).
    \end{equation*}
\end{remark}

\section{3D virtual element formulation}\label{sec:3d} In this section, we extend the discretisation to the 3D case following \cite{beirao2020stokes,veiga-Hdiv}. We introduce two novel quasi-interpolators for $\bH^1$ functions in Stokes-like and edge  VE spaces, and an extended Helmholtz decomposition for 3D mixed VEM. Although we only sketch the proof of reliability to avoid repeating arguments from the two-dimensional case, we present a detailed proof of the interpolation estimates, as these constitute essential tools and may be of independent interest.

\subsection{Considerations for 3D discretisation} Let $\mathcal{T}^h$ be a decomposition of $\Omega$ into polyhedral elements $P$ with diameter $h_P$ and let $\mathcal{F}^h$ be the set of faces $f$ with diameter $h_f$. Natural extensions of the mesh assumptions from Section~\ref{sec:vem} are given as follows:
\begin{enumerate}[label={(\bfseries M\arabic*)}]
    \item \label{M3} each polyhedral element $P$ is star-shaped with respect to a ball of radius $\geq$ $\rho h_P$,
    \item \label{M13d} every face $f$ of $P$ is star-shaped with respect to a disk of radius $\geq$ $\rho h_P$,
    \item \label{M23d} every edge $e$ of $P$ has length $\geq$ $\rho h_P$.
\end{enumerate}

Note that the polynomial decompositions discussed in Section~\ref{sec:vem} also follow in the 3D case for $\boldsymbol{\mathcal{P}}_{k}(P)$ with the spaces $\boldsymbol{\mathcal{G}}_k^\oplus(P):= \bx\wedge(\mathcal{P}_{k-1}(P))^3$, $\boldsymbol{\mathcal{R}}_{k}(P):= \bcurl(\mathcal{P}_{k+1}(P))$, and $\boldsymbol{\mathcal{R}}_k^\oplus(P) := \bx \mathcal{P}_{k-1}(P)$, where $\bx := (x_1,x_2,x_3)^{\tt t}$, and $\wedge$ the usual external product.

Following Section~\ref{sec:aposteriori}, the set of faces is divided as $\mathcal{F}^h = \mathcal{F}^h_\Omega \cup \mathcal{F}^h_D \cup \mathcal{F}^h_N$, where $\mathcal{F}^h_\Omega = \{ f\in \mathcal{F}^h: f\subset \Omega\}$, $\mathcal{F}^h_D = \{ f\in \mathcal{F}^h : f\subset \Gamma_D\}$ and $\mathcal{F}^h_N = \{ f\in \mathcal{F}^h : f\subset \Gamma_N\}$. Furthermore, the set of faces of $P$ is denoted by $\mathcal{F}^h(P)$, the set of faces of $P$ which are not in the boundary $\partial \Omega$ is denoted by $\mathcal{F}^h_\Omega(P)$, and the ones that lie on the Dirichlet  (resp. Neumann) portion of the boundary are denoted by $\mathcal{F}^h_D(P)$ (resp. $\mathcal{F}^h_N(P)$). Also, the set of elements $P$ that share $f$ as a common face is denoted by $\mathcal{T}^h_f$ and the set of faces $f$ of $P$ that share a common edge $e$ is denoted by $\mathcal{F}^h_e(P)$. The normal and tangential jump operators are defined as usual by $\jump{\bu \cdot \bn^f_P}:= (\bu|_{P} - \bu|_{P'})|_f \cdot \bn^f_P$ and $\jump{\bzeta \times \bn^f_P}:= (\bzeta|_{P} - \bzeta|_{P'})|_f \times \bn^f_P$, where $P$ and $P'$ are elements in $\mathcal{T}^h$ with a common face $f$, also $\bn^f_P$ and  $\bt^{f,1}_P,\bt^{f,2}_P$ are the outward normal and tangential vectors of $P$ with respect to the plane defined by $f$. In addition, for a smooth enough vector-valued function $\bu$ on $P$ we define the tangential component with respect to $f$ as $\bu_f := \bu - (\bu\cdot \bn_P^f)\bn_P^f$. We also set $\bu_t |_f := \bu_f$.

We recall that the same definitions from Section~\ref{sec:proj_interp} hold for the 3D case taking into account that the set of rigid body motions for a polyhedral $P$ is given by
\begin{align*}
    \textbf{RBM}(P) &= 
        \left\{ \begin{pmatrix} \frac{1}{h_P} \\ 0 \\ 0 \end{pmatrix}, 
\begin{pmatrix} 0 \\ \frac{1}{h_P} \\ 0 \end{pmatrix}, \begin{pmatrix} 0 \\ 0 \\ \frac{1}{h_P} \end{pmatrix}, \begin{pmatrix} \frac{x_{2,P}-x_2}{h_P} \\ \frac{x_1-x_{1,P}}{h_P} \\ 0 \end{pmatrix}, \begin{pmatrix} 0 \\ \frac{x_{3,P}-x_3}{h_P} \\ \frac{x_2-x_{2,P}}{h_P} \end{pmatrix}, \begin{pmatrix} \frac{x_3-x_{3,P}}{h_P} \\ 0 \\ \frac{x_{1,P}-x_1}{h_P} \end{pmatrix}\right\}.
\end{align*}
We shall use the notation $\bPi_1^{\beps,k_1,P},\Pi_1^{\beps,k_1,P},\bPi_j^{0,k_j,P},\Pi_j^{0,k_j,P}$ (resp. $\bPi_1^{\beps,k_1,f}$, $\Pi_1^{\beps,k_1,f}$, $\bPi_j^{0,k_j,f}$, $\Pi_j^{0,k_j,f}$) for vector and scalar valued polyhedral (resp. face) projections. We refer  to \cite[Proposition 5.1]{beirao2020stokes} and \cite[Theorem 3.2]{veiga-Hdiv} for the computability of the 3D projections in Stokes and $\textrm{H}(\vdiv)-\textrm{H}(\bcurl)$ spaces, respectively. Furthermore, Lemma~\ref{approximation-estimates} can be extended to the 3D case by  classical techniques for polynomial projections and a scaling argument.

\subsection{Discrete spaces}\label{sec:discrete_spaces_3D} Given $k_1\geq 2$, we first define a local VE space \cite{beirao2020stokes} as
\begin{align*}
    \bV^{h,k_1}_1(P):= \{
            \bv\in \bH^1(P) \colon &\bv|_{\partial P}\in \widetilde{\boldsymbol{\mathcal{B}}}^{h,k_1}_1(\partial P),\;\vdiv \bv\in\mathcal{P}_{k_1-1}(P), \\
            & -2\mu \bdiv\beps(\bv) -\nabla s\in \boldsymbol{\mathcal{G}}_{k_1}^\oplus(P) \text{, for some } s\in \mathrm{L}^2_0(P)\},
\end{align*}
where the boundary space of VE functions along the boundary $\partial P$ of $P$, is defined as follows
$$\widetilde{\mathcal{B}}^{h,k_1}_1(\partial P):=\{v\in C^0(\partial P)\colon v|_f\in \widetilde{\mathcal{B}}^{h,k_1}_1(f), \, \forall f\subset\partial P\},$$
and for each face $f\in\partial P$, the enhanced VE space $\widetilde{\mathcal{B}}^{h,k_1}_1(f)$ locally solves the Poisson equation with Dirichlet boundary conditions, and it is defined by (here $\Delta_f$ denotes the tangential Laplacian on $f$)
\begin{align*}\label{VE:face}
    \widetilde{\mathcal{B}}^{h,k_1}_1(f):= \{
        v\in \mathrm{H}^1(f) \colon &v|_{\partial f}\in C^0(\partial f),\; v|_e\in\mathcal{P}_{k_1}(e),\, \forall e\subset\partial f, \; \Delta_f v\in\mathcal{P}_{k_1+1}(f),\\
        & \int_f (v-\Pi_1^{\beps,k_1,f}v) m_{k_1+1} = 0, \; \forall m_{k+1}\in\mathcal{M}_{k_1+1}(f)\setminus\mathcal{M}_{k_1-2}(f)\}.
\end{align*}
Notice that the enhanced property on faces is needed in 3D to compute exactly the  $L^2$-orthogonal projection $\Pi_1^{0,k_1+1,f}$ (see \cite[Remark 5]{ahmad13} for further details).  Moreover, $\boldsymbol{\mathcal{P}}_{k_1}(P)\subseteq \bV_1^{h,k_1}(P)$. The global discrete spaces are defined by
\begin{align*}
\bV^{h,k_1}_1 := \{\bv\in \bV_1: \bv|_P\in\bV^{h,k_1}_1(P), \, \forall P\in\mathcal{T}^h\}, \quad
\mathrm{Q}_1^{h,k_1} := \{ q\in \mathrm{Q}_1 \colon q|_P\in \mathcal{P}_{k_1-1}(P), \, \forall P\in \mathcal{T}^h \}.
\end{align*}
The DoFs for $\bv_h\in\bV_1^{h,k_1}(P)$ and ${q}_h\in \mathrm{Q}_1^{h,k_1}(P)$ are as follows
\begin{align*}
    &\bullet \text{The values of $\bv_h$ at the vertices of $P$},\\
    &\bullet \text{The values of $\bv_h$ at the $k_1-1$ internal quadrature points on each edge of $P$},\\
    &\bullet \dashint_f (\bv_h\cdot\bn_P^f)m_{k_1-2}, \quad \forall m_{k_1-2}\in\mathcal{M}_{k_1-2}(f),\\
    &\bullet \dashint_f (\bv_h)_t\cdot\mathbf{m}_{k_1-2}, \quad \forall \mathbf{m}_{k_1-2}\in \boldsymbol{\mathcal{M}}_{k_1-2}(f),\\
    &\bullet \int_P \bv_h\cdot \mathbf{m}_{k_1-2}^\oplus, \quad \forall \mathbf{m}_{k_1-2}^\oplus\in\boldsymbol{\mathcal{M}}_{k_1-2}^\oplus(P),\\
    &\bullet \int_P (\vdiv \bv_h) m_{k_1-1}, \quad \forall m_{k_1-1}\in \mathcal{M}_{k_1-1}(P)\setminus \left\{ \frac{1}{h_P} \right\},\\
    &\bullet \int_P {q}_h m_{k_1-1}, \quad \forall m_{k_1-1}\in \mathcal{M}_{k_1-1}(P).
\end{align*}

 The construction of the $\bH(\vdiv,\Omega)$ conforming 3D VE space naturally follows its 2D counterpart \cite{veiga-Hdiv}. The definition differs from the 2D version by setting $k_2\geq 1$ to ensure continuity of the normal components across faces. The discrete VE space locally solves a $\nabla(\vdiv)$-$\bcurl$ problem as follows
\begin{align*}
    \bV_2^{h,k_2}(P) := \{ \bxi \in \bH(\vdiv,P)\cap \bH(\bcurl,P) \colon &\bxi \cdot \bn_P^f|_f \in \mathcal{P}_{k_2}(f), \, \forall f\in \partial P, \\
    & \nabla (\vdiv \bxi) \in \boldsymbol{\mathcal{G}}_{k_2-2}(P), \; \bcurl \bxi \in \boldsymbol{\mathcal{R}}_{k_2-1}(P)\}.
\end{align*}
Observe that $\boldsymbol{\mathcal{P}}_{k_2}(P)\subseteq\bV_2^{h,k_2}(P)$. Then, the discrete global spaces are defined by
\begin{align*}
    \bV_2^{h,k_2} := \{ \bxi \in \bV_2 \colon \bxi|_P \in \bV_2^{h,k_2}(P), \, \forall P\in \mathcal{T}^h \}, \quad
    \mathrm{Q}_2^{h,k_2} := \{ \psi\in \mathrm{Q}_2 \colon \psi|_P\in \mathcal{P}_{k_2-1}(P), \, \forall P\in \mathcal{T}^h \}.
\end{align*}
The set of DoFs for $\bxi_h\in \bV_2^{h,k_2}(P)$ and $\psi_h\in \mathrm{Q}_2^{h,k_2}(P)$ can be taken as
\begin{align*}
&\bullet \text{The values of } \bxi_h\cdot \bn_{P}^{f} \text{ at the $k_2+1$ quadrature points on each face of $P$}, \\
&\bullet  \int_P \bxi_h\cdot \mathbf{m}_{k_2-2}^\nabla, \quad \forall \mathbf{m}_{k_2-2}^{\nabla} \in \boldsymbol{\mathcal{M}}_{k_2-2}^{\nabla}(P),\\
&\bullet \int_P \bxi_h \cdot \mathbf{m}_{k_2}^{\oplus}, \quad \forall \mathbf{m}_{k_2}^{\oplus} \in \boldsymbol{\mathcal{M}}_{k_2}^{\oplus}(P), \\
&\bullet \int_P \psi_h m_{k_2}, \quad \forall m_{k_2}\in \mathcal{M}_{k_2}(P).
\end{align*}

\subsection{Interpolation operators}\label{sec:interp_3d}
Our goal is to define a quasi--interpolator $\mathbf{I}_1^{Q,k_1}:\mathbf{H}^{1}(P)\rightarrow \bV_1^{h,k_1}(P)$ and Fortin operator $\mathbf{I}_2^{F,k_2}:\mathbf{H}^{1}(P)\rightarrow \bV_2^{h,k_2}(P)$, which are essential in extending the reliability result (see Theorem~\ref{th:upper-bound}) to the 3D case. The Fortin operator $\mathbf{I}_2^{F,k_2}$ defined through the DoFs (similarly to the 2D case) is presented in \cite[Section 4.1]{Beirao22}. On the other hand, the construction of $\mathbf{I}_1^{Q,k_1}$ is more involved due to the minimal regularity, which is stated next.  
\begin{proposition}\label{prop:int}
   Let $\bv\in \bH^{s_1+1}(\Omega)$, and $0\leq s_1 \leq k_1$. Under the mesh assumptions, there exists $\mathbf{I}_1^{Q,k_1}\bv\in\bV_1^{h,k_1}$ such that, for all $P\in \mathcal{T}^h$ 
   \[\norm{\bv-\mathbf{I}_1^{Q,k_1}\bv}_{0,P}+h_P\norm{\beps(\bv-\mathbf{I}_1^{Q,k_1}\bv)}_{0,P}\lesssim 
   h_P^{s_1+1}|\bv|_{s_1+1,D(P)},\]
   where $D(P)$ denote the union of the polyhedral elements in $\mathcal{T}^h$ intersecting $P$. 
\end{proposition}
\begin{proof}
\textit{Step 1 (Existence of interpolation in the space $\bV_1^{h,k_1}$)}. Let $P\in\mathcal{T}^h$, we start considering the super-enhanced version of the enhanced VE space $\widetilde{\boldsymbol{\mathcal{B}}}^{h,k_1}_1(f)$  from \cite{cangiani2017posteriori} which locally solves a Poisson problem with Dirichlet boundary conditions. Let $\mathbf{I}_1^{C,k_1}\bv\in\bH^1_D(\Omega)$ be the Cl\'ement-type VE interpolation in $\widetilde{\boldsymbol{\mathcal{B}}}^{h,k_1}_1(f)$ satisfying the following estimate for all $P\in \mathcal{T}^h$ (see \cite[Theorem 11]{cangiani2017posteriori})
\begin{align}\label{5.1.1}
\norm{\bv-\mathbf{I}_1^{C,k_1}\bv}_{0,P}+h_P\norm{\beps(\bv-\mathbf{I}_1^{C,k_1}\bv)}_{0,P}\lesssim  
h_P^{s_1+1}|\bv|_{s_1+1,D(P)}.
\end{align}
Let $\bv_\pi=\bPi_1^{\beps,k_1}\bv$ be the polynomial approximation of $\bv$. 
The decomposition $\boldsymbol{\mathcal{P}}_{k_1-2}(P) = \boldsymbol{\mathcal{G}}_{k
_1-2}(P) \oplus  \boldsymbol{\mathcal{G}}_{k_1-2}^\oplus(P)$ guarantees the existence of $p_\pi\in\mathcal{P}_{k_1-1}(P)$ and $\bg_\pi^\oplus\in\boldsymbol{\mathcal{G}}_{k_1-2}^\oplus(P)$ such that 
\begin{align}
    - 2\mu \bdiv \beps(\bv_\pi) = \nabla p_\pi+\bg_\pi^\oplus. \label{5.1.2}
\end{align}
For $q_{k_1-1}|_P = \Pi_1^{0,k_1-1}(\vdiv \mathbf{I}_1^{C,k_1}\bv)$, we introduce the following Stokes problem for all $P\in\mathcal{T}^h$ as 
\begin{align}\label{5.1.3}
    \begin{cases}
        -2\mu \bdiv (\beps(\widetilde{\bv}_I)) -\nabla s = \bg_\pi^\oplus&\quad\text{in}\;P,\\
        \vdiv \widetilde{\bv}_I = q_{k_1-1}&\quad\text{in}\;P,\\
        \widetilde{\bv}_I = \mathbf{I}_1^{C,k_1}\bv&\quad\text{on}\;\partial P.
    \end{cases}
\end{align}
Observe that $\widetilde{\bv}_I\in\bV_1^{h,k_1}(P)$ from the definition of $\bV_1^{h,k_1}(P)$. This and $\widetilde{\bv}_I=\mathbf{I}_1^{C,k_1}\bv\in\bV_1$ on each boundary of $P$ imply that $\widetilde{\bv}_I\in\bV_1^{h,k_1}$.

\noindent \textit{Step 2 (Error estimate)}. Note that \eqref{5.1.2} imply that $\bv_\pi$ solves the following local Stokes problem 
\begin{align}\label{5.1.4}
    \begin{cases}
        -2\mu \bdiv (\beps(\bv_\pi))-\nabla p_\pi = \bg_\pi^\oplus&\quad\text{in}\;P,\\
        \vdiv \bv_\pi = \vdiv\bv_\pi&\quad\text{in}\;P,\\
       \bv_\pi = \bv_\pi&\quad\text{on}\;\partial P.
    \end{cases}
\end{align}
On the other hand, we define the auxiliary problem
\begin{align}\label{5.1.4.1}
    \begin{cases}
        -2\mu \bdiv (\beps(\widetilde{\bv})) -\nabla \widetilde{s} = \bg_\pi^\oplus&\quad\text{in}\;P,\\
        \vdiv \widetilde{\bv} = \vdiv \mathbf{I}_1^{C,k_1}\bv&\quad\text{in}\;P,\\
        \widetilde{\bv} = \mathbf{I}_1^{C,k_1}\bv&\quad\text{on}\;\partial P.
    \end{cases}
\end{align}
Then, subtracting \eqref{5.1.4} and \eqref{5.1.4.1} lead to
\begin{align}\label{5.1.5}
    \begin{cases}
        -2\mu \bdiv (\beps(\bv_\pi-\widetilde{\bv}))-\nabla (p_\pi-\widetilde{s}) = \bzero&\quad\text{in}\;P,\\
        \vdiv (\bv_\pi-\widetilde{\bv}) = \vdiv(\bv_\pi-\mathbf{I}_1^{C,k_1}\bv)&\quad\text{in}\;P,\\
       \bv_\pi- \widetilde{\bv} = \bv_\pi-\mathbf{I}_1^{C,k_1}\bv&\quad\text{on}\;\partial P.
    \end{cases}
\end{align}
Since $\bz-(\bv_\pi- \widetilde{\bv})\in \bH^1_{0}(P)$ for any $\bz\in \bH^1(P)$ with $\vdiv \bz = \vdiv (\bv_\pi-\mathbf{I}_1^{C,k_1}\bv)$ in $P$ and $\bz = \bv_\pi-\mathbf{I}_1^{C,k_1}\bv$ on $\partial P$, an integration by parts twice and \eqref{5.1.5} show 
\begin{align*}
    &2\mu \int_P \beps(\bv_\pi- \widetilde{\bv}):\beps(\bz-(\bv_\pi- \widetilde{\bv}))=\int_P \nabla (p_\pi - \widetilde{s}) \cdot (\bz-(\bv_\pi- \widetilde{\bv}))=-\int_P(p_\pi-\widetilde{s})\vdiv(\bz-(\bv_\pi- \widetilde{\bv}))=0.
\end{align*}
This results in $\norm{\beps(\bv_\pi- \widetilde{\bv})}_{0,P}\leq \norm{\beps(\bz)}_{0,P}$ and hence for $\bz = \bv_\pi-\mathbf{I}_1^{C,k_1}\bv$ we have 
\begin{align}
    &\norm{\beps(\bv_\pi- \widetilde{\bv})}_{0,P}\leq \norm{\beps(\bv_\pi-\mathbf{I}_1^{C,k_1}\bv)}_{0,P} \leq \norm{\beps(\bv - \bv_\pi)}_{0,P}+\norm{\beps(\bv-\mathbf{I}_1^{C,k_1}\bv)}_{0,P}
    \lesssim h_P^{s_1}|\bv|_{s_1+1,D(P)},\label{estimate_aux_poly_st}
\end{align}
where the triangle inequality, $\norm{\beps(\cdot)}_{0,P}\leq |\cdot|_{1,P}$, Lemma~\ref{approximation-estimates} and \eqref{5.1.1} were applied. Next, to estimate $\widetilde{\bv}-\widetilde{\bv}_I$ we consider the auxiliary problem arising from subtracting \eqref{5.1.4.1} and \eqref{5.1.3}
\begin{align*}
    \begin{cases}
        -2\mu \bdiv(\beps(\widetilde{\bv}-\widetilde{\bv}_I)) -\nabla (\widetilde{s}-s) = \bzero &\quad\text{in}\;P,\\
        \vdiv (\widetilde{\bv}-\widetilde{\bv}_I) = \vdiv(\mathbf{I}_1^{C,k_1}\bv)-q_{k-1}&\quad\text{in}\;P,\\
        \widetilde{\bv}-\widetilde{\bv}_I = \bzero &\quad\text{on}\;\partial P.
    \end{cases}
\end{align*}
Given a star-shaped polyhedral domain $P$ with respect to a ball $B$, the mesh assumptions lead  to $|P||B|^{-1}\leq \rho^{-3}$ in 3D (see e.g. \cite{di2020hybrid}). This and \cite[Theorem 3.2]{duran2012elementary} imply that there exists a universal positive constant 
depending only on the mesh regularity parameter $\rho$ and the dimension $d$ (but independent of $P$) such that 
\begin{align}
\norm{\beps(\widetilde{\bv}-\widetilde{\bv}_I)}_{0,P}\leq |\widetilde{\bv}-\widetilde{\bv}_I|_{1,P} 
\lesssim \norm{\vdiv(\mathbf{I}_1^{C,k_1}\bv)-\Pi_1^{0,k_1-1}\vdiv(\mathbf{I}_1^{C,k_1}\bv)}_{0,P}.\label{duran}
\end{align}
Note that the triangle inequality, \eqref{5.1.1}, and the estimates for $\Pi_1^{0,k_1-1}$  allow us to assert that 
\begin{align}
&\norm{\vdiv (\mathbf{I}_1^{C,k_1}\bv) - \Pi_1^{0,k_1-1} \vdiv \bv}_{0,P} 
\lesssim \norm{\vdiv(\bv-\mathbf{I}_1^{C,k_1}\bv)}_{0,P}+\norm{\vdiv\bv-\Pi_1^{0,k_1-1}\vdiv\bv}_{0,P}
\lesssim h^{s_1}|\bv|_{s_1+1,D(P)}.\label{est:vI}
\end{align}
Thus, \eqref{duran}, the $\mathrm{L}^2$-orthogonality of $\Pi_1^{0,k_1-1}$, and \eqref{est:vI} lead to 
\begin{align}\label{estimate-aux}
    \norm{\beps(\widetilde{\bv}-\widetilde{\bv}_I)}_{0,P} &\lesssim \norm{\vdiv(\mathbf{I}_1^{C,k_1}\bv)-\Pi_1^{0,k_1-1}\vdiv(\mathbf{I}_1^{C,k_1}\bv)}_{0,P} = \norm{\vdiv (\mathbf{I}_1^{C,k_1}\bv) - \Pi_1^{0,k_1-1} \vdiv \bv}_{0,P} \lesssim h^{s_1}|\bv|_{s_1+1,D(P)}.
\end{align}

\noindent \textit{Step 3 ($\bL^2$-error estimate)}. Since $\mathbf{I}_1^{C,k_1}\bv$ and $\mathbf{I}_1^{Q,k_1}\bv$ coincide in $\partial P$, the Poincar\'e  inequality holds. Thus,
$$\norm{\mathbf{I}_1^{C,k_1}\bv-\mathbf{I}_1^{Q,k_1}\bv}_{0,P} \lesssim h_P \norm{\beps(\mathbf{I}_1^{C,k_1}\bv-\mathbf{I}_1^{Q,k_1}\bv)}_{0,P}.$$
This, \eqref{5.1.1}, and the $H^1$ estimate of the interpolator with the triangle inequality 
$$\norm{\bv-\mathbf{I}_1^{Q,k_1}\bv}_{0,P}\leq \norm{\bv-\mathbf{I}_1^{C,k_1}\bv}_{0,P}+\norm{\mathbf{I}_1^{C,k_1}\bv-\mathbf{I}_1^{Q,k_1}\bv}_{0,P}$$ prove the required bound.
\end{proof}
\begin{remark} We recall that the scaled interpolation estimates for $\mathbf{I}_1^{Q,k_1}$ and $\mathbf{I}_2^{F,k_2}$ follow exactly as Lemma~\ref{interpolation-estimates} by adapting Proposition~\ref{prop:int} and \cite[Theorem 4.2]{Beirao22} with the correct scaling factor. 
\end{remark}
\begin{remark}\label{3d-wp} The stability and   well-posedness of  problem \eqref{eq:sto1}-\eqref{eq:sto2} in 3D is discussed in \cite[Section 5.1]{beirao2020stokes}, whereas the commutative property of $\mathbf{I}_2^{F,k_2}$ ensures the stability of the 3D reaction-diffusion equation \eqref{eq:mix1}-\eqref{eq:mix2}. The well-posedness of the discrete fully-coupled problem \eqref{eq:weak-discrete} in 3D follows using a Banach fixed-point argument with small data assumption (see \cite[Section 4]{khot2024}), while a priori error estimates are provided in \cite{rubiano2025}.
\end{remark}

\subsection{A posteriori error analysis in 3D}
Here we extend the results from \cite{munar2024} to our setting. First we invoke a 3D Helmholtz decomposition  \cite[Lemma 3.9]{gatica2016} as an extension of the 2D version used in Lemma~\ref{residual-bound-2}.

\begin{lemma} \label{lem:decomposition_Hdiv}
    Assume that $\Omega \subset \mathbb{R}^3$ is a connected domain and that $\Gamma_N$ is contained in the boundary of a convex part of $\Omega$, that is, there exists a convex domain $B$ such that $\overline{\Omega}\subseteq B$ and $\Gamma_N\subseteq \partial B$ (see Figure~\ref{fig:domain_helmhotz}). Then, for each $\bxi \in \bH_N(\vdiv,\Omega)$, there exist $\bz \in \bH^1(\Omega)$ and $\bchi \in \bH^1_N(\Omega)$ such that
    \begin{equation*}
        \bxi = \bz + \bcurl \bchi \text{ in } \Omega \text{ and } \norm{\bz}_{1,\Omega} + \norm{\bchi}_{1,\Omega} \lesssim \norm{\bxi}_{\vdiv,\Omega}.
    \end{equation*}
\end{lemma}
 
Next, we introduce the vector space $\bV_3 := \bH(\vdiv,\Omega) \cap \bH(\bcurl,\Omega)$ with corresponding local virtual space
\begin{align*}
    \bV_3^{h,k_2+1}(P) := \{\bchi\in \bV_3(P):&\, \bchi_{t}|_{\partial P} \in \mathcal{B}_3^{h,k_2+1}(\partial P), \vdiv \bchi \in \mathcal{P}_{k_2}(P), \bcurl (\bcurl\bchi)\in \boldsymbol{\mathcal{R}}_{k_2-1}(P)\},
\end{align*}
with $\mathcal{B}_3^{h,k_2}(\partial P) := \{ \bchi \in (\mathcal{B}_3)_t : \bchi_f \in \bV_{2}^{h,k_2}(f), \, \forall f\in \partial P\}$. Here, $(\mathcal{B}_3)_t$ is defined as the tangential component of the elements in $\mathcal{B}_3(\partial P)$ given for $f\in \partial P$ by
\begin{align*}
     \mathcal{B}_3(\partial P) := \{\bchi \in \bL^2(\partial P) : &\, \bchi|_f \in  \bH(\vdiv,f) \cap \bH(\vrot,f),\,  \forall f\in \mathcal{F}^h,\bchi_{f_1}\cdot \bt_e = \bchi_{f_2}\cdot \bt_e,\, \forall f_1,f_2\in \mathcal{F}^h_e,\, e\subset \partial f \}.
\end{align*}
The VE space $\bV_{2}^{h,k_2}(f)$ refers to the 2D edge space (see \cite[Section 4]{veiga-Hdiv}) which is a rotation by $\pi/2$ of the standard VE space for mixed problems given in Section~\ref{sec:vem}. Finally, the global space $\bV_3^{h,k_2+1}$ is given as follows
\begin{align*}
    \bV_3^{h,k_2+1} := \{\bchi_h\in \bV_3: \bchi_h|_P\in\bV^{h,k_2+1}_3(P), \, \forall P\in\mathcal{T}^h\}.
\end{align*}
For further details about DoFs and unisolvence we refer to \cite[Section 6]{veiga-Hdiv}. A key relation between $\bV_2^{h,k_2}$ and $\bV_3^{h,k_2+1}$ needed for the Helmholtz decomposition is given next and we refer to \cite[Theorem 8.2]{veiga-Hdiv} for a proof.
\begin{lemma} \label{lem:relation_V3_V2}
    For $k_2 \geq 1$ and $\bchi_h \in \bV_3^{h,k_2+1}$, we have that $\bcurl \bchi_h \in \bV_2^{h,k_2}$.
\end{lemma}
\begin{figure}[ht!]
    \centering
    \includegraphics[width=1.0\textwidth,trim={0.05cm .1cm 0.05cm 0cm},clip]{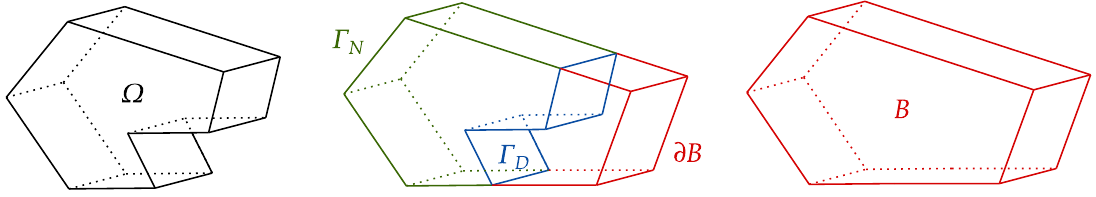}  
    \caption{Extension of $\Omega$ (left) to a convex domain $B$ (right), with mixed boundary conditions  on $\Gamma_N$ and $\Gamma_D$ (centre).\label{fig:domain_helmhotz}}
\end{figure}

Following the approach outlined in Section~\ref{sec:aposteriori}, we introduce a novel quasi-interpolation operator for the 3D VE edge space, specifically applied to $\bH^1(\Omega)$ functions. 
\begin{lemma}\label{lem:quasi-int-edge}
Let $\bchi\in \bH^{s_1+1}(\Omega)$, and $0\leq s_2 \leq k_2$. Under the mesh assumptions, there exists $\mathbf{I}_3^{Q,k_2+1}\bchi\in\bV_3^{h,k_2+1}$ such that 
   \[\norm{\bchi-\mathbf{I}_3^{Q,k_2+1}\bchi}_{0,P}+h_P|\bchi-\mathbf{I}_3^{Q,k_2+1}\bchi|_{1,P}
   \lesssim h_P^{s_2+1}|\bv|_{s_2+1,D(P)}.\]
\end{lemma}
\begin{proof}
\textit{Step 1 (Existence of interpolation in the 3D edge VE space)}. First, let $\mathbf{I}_3^{C,k_2+1}\bchi\in\bH^1_N(\Omega)$ be the 3D Cl\'ement-type interpolation in the N\'ed\'elec FE space defined on a sub-triangulation allowed by the mesh assumptions (See \cite[Theorem 5.2]{Ern2017}), satisfying the following estimate
\begin{align}
\norm{\bchi-\mathbf{I}_3^{C,k_2+1}\bchi}_{0,P}+h_P|\bchi-\mathbf{I}_3^{C,k_2+1}\bchi|_{1,P} 
\lesssim h_P^{s_2+1}|\bchi|_{s_2+1,D(P)}.\label{clement-bound-V3}
\end{align}
For each polyhedron $P\in\mathcal{T}^h$, let $\bchi_\pi=\bPi_2^{0,k_2+1}\bchi$ be the polynomial approximation of $\bchi$. Then, we consider 
\begin{align}\label{aux-prob-edge}
    \begin{cases}
        \bcurl(\bcurl(\mathbf{I}_3^{Q,k_2+1}\bchi))=\bcurl(\bcurl\bchi_\pi) &\quad\text{in}\;P,\\
        \vdiv (\mathbf{I}_3^{Q,k_2+1}\bchi) = \vdiv \bchi_\pi &\quad\text{in}\;P,\\
        (\mathbf{I}_3^{Q,k_2+1}\bchi)_t = (\mathbf{I}_3^{C,k_2+1}\bchi)_t &\quad\text{on}\;\partial P.
    \end{cases}
\end{align}
Note that the definition of the N\'ed\'elec space implies that $\jump{(\mathbf{I}_3^{C,k_2+1}\bchi)_f \cdot \bt_e} = 0$ for all $e\in \partial f$, and $f\in \partial P$. Therefore, the unique solvability of \eqref{aux-prob-edge} (see \cite[Section 6.2]{veiga-Hdiv}) for all $P$ implies that $\mathbf{I}_3^{Q,k_2+1} \bchi  \in \bV_3^{h,k_2+1}$. 

\noindent \textit{Step 2 (Error estimate)}. Subtracting $\bchi_\pi$ from \eqref{aux-prob-edge} leads to
\begin{align}\label{aux-prob-edge-2}
    \begin{cases}
        \bcurl(\bcurl(\mathbf{I}_3^{Q,k_2+1}\bchi-\bchi_\pi))=\bzero &\quad\text{in}\;P,\\
        \vdiv (\mathbf{I}_3^{Q,k_2+1}\bchi-\bchi_\pi) = 0 &\quad\text{in}\;P,\\
        (\mathbf{I}_3^{Q,k_2+1}\bchi-\bchi_\pi)_t = (\mathbf{I}_3^{C,k_2+1}\bchi-\bchi_\pi)_t &\quad\text{on}\;\partial P.
    \end{cases}
\end{align}
We invoke the analysis of \cite[Section 6.2]{veiga-Hdiv} to show that the unique solution of \eqref{aux-prob-edge-2} has the decomposition $\mathbf{I}_3^{Q,k_2+1}\bchi-\bchi_\pi = \tilde{\bpsi} + \nabla \tilde{\varphi}$, where $\tilde{\bpsi}\in \bH^1(P)$ and $\tilde{\varphi}\in \mathrm{H}^1_0(P)$ are the respective unique solutions of
\begin{align}\label{aux-prob-edge-3}
    &\begin{cases}
        \Delta\tilde{\varphi}=0 &\quad\text{in}\;P,\\
        \tilde{\varphi} = 0 &\quad\text{on}\;\partial P,
    \end{cases}\quad
    &\begin{cases}
        \bcurl\tilde{\bpsi} = \bh &\quad\text{in}\;P,\\
        \vdiv \tilde{\bpsi} = 0 &\quad\text{in}\;P,\\
        \tilde{\bpsi}_t = (\mathbf{I}_3^{C,k_2+1}\bchi-\bchi_\pi)_t &\quad\text{on}\;\partial P.
    \end{cases}
\end{align}
The above data $\bh\in \bH^1(P)$ is the solution of the following problem 
\begin{align}\label{aux-prob-edge-5}
    \begin{cases}
        \bcurl\bh = \bzero  &\quad\text{in}\;P,\\
        \vdiv \bh = 0 &\quad\text{in}\;P,\\
        \bh\cdot \bn = \vrot(\mathbf{I}_3^{C,k_2+1}\bchi-\bchi_\pi)_t &\quad\text{on}\;\partial P.
    \end{cases}
\end{align}
Notice from \eqref{aux-prob-edge-3} that $\tilde{\varphi} = 0$. Thus, the well-posedness of Maxwell-type equations from \cite[Theorem 2.3]{SHEN2014}, and the trace inequality imply that
\begin{align}
    & |\mathbf{I}_3^{Q,k_2+1}\bchi - \bchi_\pi|_{1,P} \lesssim \norm{\bcurl\tilde{\bpsi}}_{0,P} + \norm{(\mathbf{I}_3^{C,k_2+1}\bchi-\bchi_\pi)_t}_{\frac{1}{2},\partial P} \lesssim \norm{\bh}_{0,P} + |\mathbf{I}_3^{C,k_2+1}\bchi-\bchi_\pi|_{1,P}.\label{bound-cle-poly-1}
\end{align}
Next, we follow a local version of \cite{ZAGHDANI2006} to estimate $\norm{\bh}_{0,P}$. For this, we recall the Helmholtz decomposition $\bL^2(P) = \bH(\bcurl_{\bzero}, P) \oplus \bH_{\partial P}(\vdiv_0, P)$ (see \cite{Girault1986}) where $\bH(\bcurl_{\bzero}, P)$, $\bH_{\partial P}(\vdiv_0, P)$ correspond to $\bcurl$-free and $\vdiv$-free (with zero boundary condition) functions, respectively. Thus, there exist $\bh_1\in \bH(\bcurl \times \bzero, P)$ and $\bh_2 \in \bH_{\partial P}(\vdiv \cdot \bzero, P)$ such that $\bh = \bh_1 + \bh_2$. Moreover, $\bh_1=\nabla q$ for some $q\in \text{H}^1(P)$ with $\int_P q = 0$, and $\bh_2 = \bcurl(\bphi)$ for some $\bphi\in \bH_{\partial P}(\bcurl, P) \cap \bH(\vdiv_0, P)$, with 
\begin{align}\label{norm-h}
    \norm{\bh}_{0,P}^2 = \norm{\nabla q}_{0,P}^2 + \norm{\bcurl(\bphi)}_{0,P}^2.
\end{align}
An integration by parts and \eqref{aux-prob-edge-5} imply that
\begin{align*}
     \norm{\bh}_{0,P}^2 = \left|  \int_P \bh \cdot (\nabla q + \bcurl(\bphi)) \right| &= \left| -\int_P \vdiv(\bh) q + \int_{\partial P} (\bh \cdot \bn)q + \int_P \bcurl(\bh) \cdot \bphi + \int_{\partial P} \bh \cdot (\bphi\times\bn) \right| \notag \\
    & = \left|\int_{\partial P} \vrot(\mathbf{I}_3^{C,k_2+1}\bchi-\bchi_\pi)_t q\right|.
\end{align*}
This, the trace inequality together with \eqref{norm-h} and the Poincar\'e inequality lead to
\begin{align}
    \norm{\bh}_{0,P}^2 = \left|\int_{\partial P} \vrot(\mathbf{I}_3^{C,k_2+1}\bchi-\bchi_\pi)_t q\right| &\lesssim h_P^{-\frac{1}{2}} \norm{\bcurl(\mathbf{I}_3^{C,k_2+1}\bchi-\bchi_\pi)}_{0,P} (h_P^{-\frac{1}{2}}\norm{q}_{0,P} + h_P^{\frac{1}{2}}\norm{\nabla q}_{0,P} ) \notag\\
    & \lesssim h_P^{-\frac{1}{2}} \norm{\bcurl(\mathbf{I}_3^{C,k_2+1}\bchi-\bchi_\pi)}_{0,P} (h_P^{\frac{1}{2}}\norm{\nabla q}_{0,P}) \notag\\
    & \lesssim \norm{\bcurl(\mathbf{I}_3^{C,k_2+1}\bchi-\bchi_\pi)}_{0,P}\norm{\bh}_{0,P}.\label{bound-cle-poly-2}
\end{align}
The triangle inequality, the inequality $\norm{\bcurl(\cdot)}_{0,P}\lesssim |\cdot|_{1,P}$, \eqref{bound-cle-poly-2}, and \eqref{bound-cle-poly-1} lead to
\begin{align}
    & |\bchi - \mathbf{I}_3^{Q,k_2+1}\bchi|_{1,P} \leq |\bchi - \bchi_\pi|_{1,P} + |\mathbf{I}_3^{Q,k_2+1}\bchi - \bchi_\pi|_{1,P} \lesssim  |\bchi - \bchi_\pi|_{1,P} + |\bchi - \mathbf{I}_3^{C,k_2+1}\bchi |_{1,P}.\label{bound-final-Hs}
\end{align}
Finally, Lemma~\ref{approximation-estimates} and \eqref{clement-bound-V3} imply the $\bH^1$-error estimate.

\noindent \textit{Step 3 ($\bL^2$-error estimate)} Since $(\mathbf{I}_3^{Q,k_2+1}\bchi)_t-(\mathbf{I}_3^{C,k_2+1}\bchi)_t = 0$, a Poincar\'e inequality can be applied to the polynomial $\mathbf{I}_3^{Q,k_2+1}\bchi-\mathbf{I}_3^{C,k_2+1}\bchi$, which together with the triangle inequality lead to
$$\norm{\bchi - \mathbf{I}_3^{Q,k_2+1}\bchi}_{0,P} \leq \norm{\bchi-\mathbf{I}_3^{C,k_2+1}\bchi}_{0,P} + h_P|\mathbf{I}_3^{C,k_2+1}\bchi-\mathbf{I}_3^{Q,k_2+1}\bchi|_{1,P}.$$
The proof follows by applying again the triangle inequality, \eqref{bound-final-Hs}, \eqref{clement-bound-V3}, and Lemma~\ref{approximation-estimates}.
\end{proof}
Finally, we focus on the extension of the reliability and efficiency of the estimators for the 3D case. Notice that the presence of the $\bcurl$ operator motivates the redefinition of the reaction-diffusion estimator as follows
\begin{align*}
    &\Xi_{2,P}^2  := \sum_{f\in\mathcal{F}^h_{\mathrm{D}}(P)} h_f \norm{\varphi_{\mathrm{D}} - \varphi_h}_{0,f}^2  + \sum_{f\in\mathcal{F}^h_{\mathrm{D}}(P)} h_f \norm{(\nabla \varphi_{\mathrm{D}} - \bPi_2^{0,k_2}(\bbM^{-1}(\beps(\overline{\bu}_h),{p}_h) \bPi_2^{0,k_2} \bzeta_h) )\times \bn}_{0,f}^2 \\ 
    & \quad + \sum_{f\in \mathcal{F}^h_\Omega(P)} h_f \norm{ \jump{\bPi_2^{0,k_2}(\bbM^{-1}(\beps(\overline{\bu}_h),{p}_h) \bPi_2^{0,k_2} \bzeta_h)\times \bn_f} }_{0,f}^2 \\
    & \quad + h_P^2 \norm{\bPi_2^{0,k_2}(\bbM^{-1}(\beps(\overline{\bu}_h),{p}_h) \bPi_2^{0,k_2} \bzeta_h) - \nabla \varphi_h}_{0,P}^2 + h_P^2\norm{\bcurl(\bPi_2^{0,k_2}(\bbM^{-1}(\beps(\overline{\bu}_h),{p}_h) \bPi_2^{0,k_2} \bzeta_h))}_{0,P}^2.
\end{align*}
On the other hand, $\Theta_{1,P}$ and the remaining terms of $\Theta_{2,P}$ are defined analogously to \eqref{error-indicators} replacing edges $e$ and polygons $E$ with faces $f$ and polyhedrons $P$, respectively. 

Following the techniques from  Section~\ref{sec:aposteriori} and the definition of the quasi-interpolator $\mathbf{I}_1^{Q,k_1}$ in Lemma~\ref{prop:int}, we readily see that Lemma~\ref{residual-bound-1} holds in the 3D case. In turn, Lemma~\ref{efficiency-elasticity} is a consequence of bubble functions in polyhedra and faces (see \cite[Lemmas 8,9]{cangiani2017posteriori}). 

Next, we recall  integration by parts formulae involving $\bcurl(\bchi-\bchi_h)$, with $\bchi \in \bH^1_N(\Omega)$ and $\bchi_h:= \mathbf{I}_3^{Q,k_2+1}(\bchi) \in \bV_3^{h,k_2+1}$  \cite[Equation 2.17, Theorem 2.11]{Girault1986}:
\begin{subequations}\label{IBP-3D}
\begin{align}
    &\sum_{f\in \mathcal{F}^h_D}\langle \bcurl(\bchi-\bchi_h), \varphi_\mathrm{D} \rangle_{f} = -\sum_{f\in \mathcal{F}^h_D} \langle \bchi-\bchi_h ,\nabla \varphi_\mathrm{D} \times \bn_f \rangle_{f}, \\
    &\int_P \bPi_2^{0,k_2}(\bbM^{-1}(\beps(\overline{\bu}_h),{p}_h) \bPi_2^{0,k_2} \bzeta_h) \cdot \bcurl(\bchi-\bchi_h) = \int_P \bcurl\left( \bPi_2^{0,k_2}(\bbM^{-1}(\beps(\overline{\bu}_h),{p}_h) \bPi_2^{0,k_2} \bzeta_h) \right) \cdot (\bchi-\bchi_h) \notag\\
    & \quad + \sum_{f\in \mathcal{F}^h(P)} \langle \bchi-\bchi_h, \bPi_2^{0,k_2}(\bbM^{-1}(\beps(\overline{\bu}_h),{p}_h) \bPi_2^{0,k_2} \bzeta_h) \times \bn_f \rangle_f.
\end{align}
\end{subequations}
Using this and Lemmas~\ref{lem:decomposition_Hdiv}-\ref{lem:quasi-int-edge}, we can prove Lemma~\ref{residual-bound-2} in the 3D case. Finally, since $\bcurl(\nabla\varphi)=\bzero$, Lemma~\ref{efficiency-reaction-diffusion} can be extended again using bubble function  techniques and the integration by parts presented in \eqref{IBP-3D}. Therefore, Theorem~\ref{th:upper-bound} and Theorem~\ref{th:lower-bound} hold also in the 3D case.

\section{Numerical tests}\label{sec:numerical-examples}
In this section, we present  numerical results illustrating the properties of the robust estimator from Section~\ref{sec:aposteriori} and show the optimal behaviour of the associated adaptive algorithm under different polygonal convex meshes and polynomial orders. We also use L-shaped and  Australia-shaped domains to illustrate the capability of capturing singularities in non-convex domains. Finally, we show some 3D tests.
\begin{figure}[ht!]
    \centering
    \subfigure[Voronoi.\label{fig:polygonal}]{\includegraphics[width=0.19\textwidth]{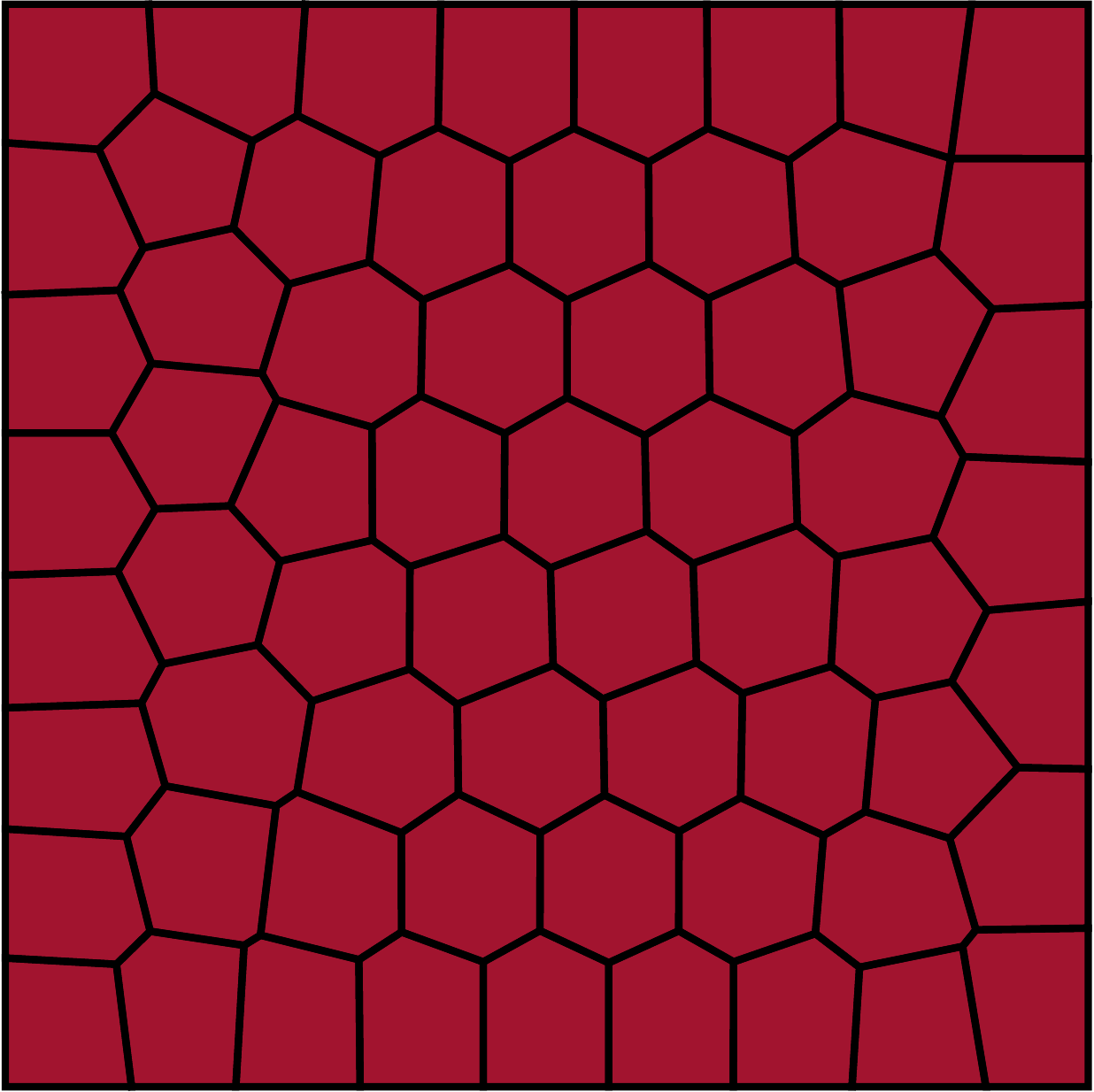}}  
    \subfigure[Square.\label{fig:square}]{\includegraphics[width=0.19\textwidth]{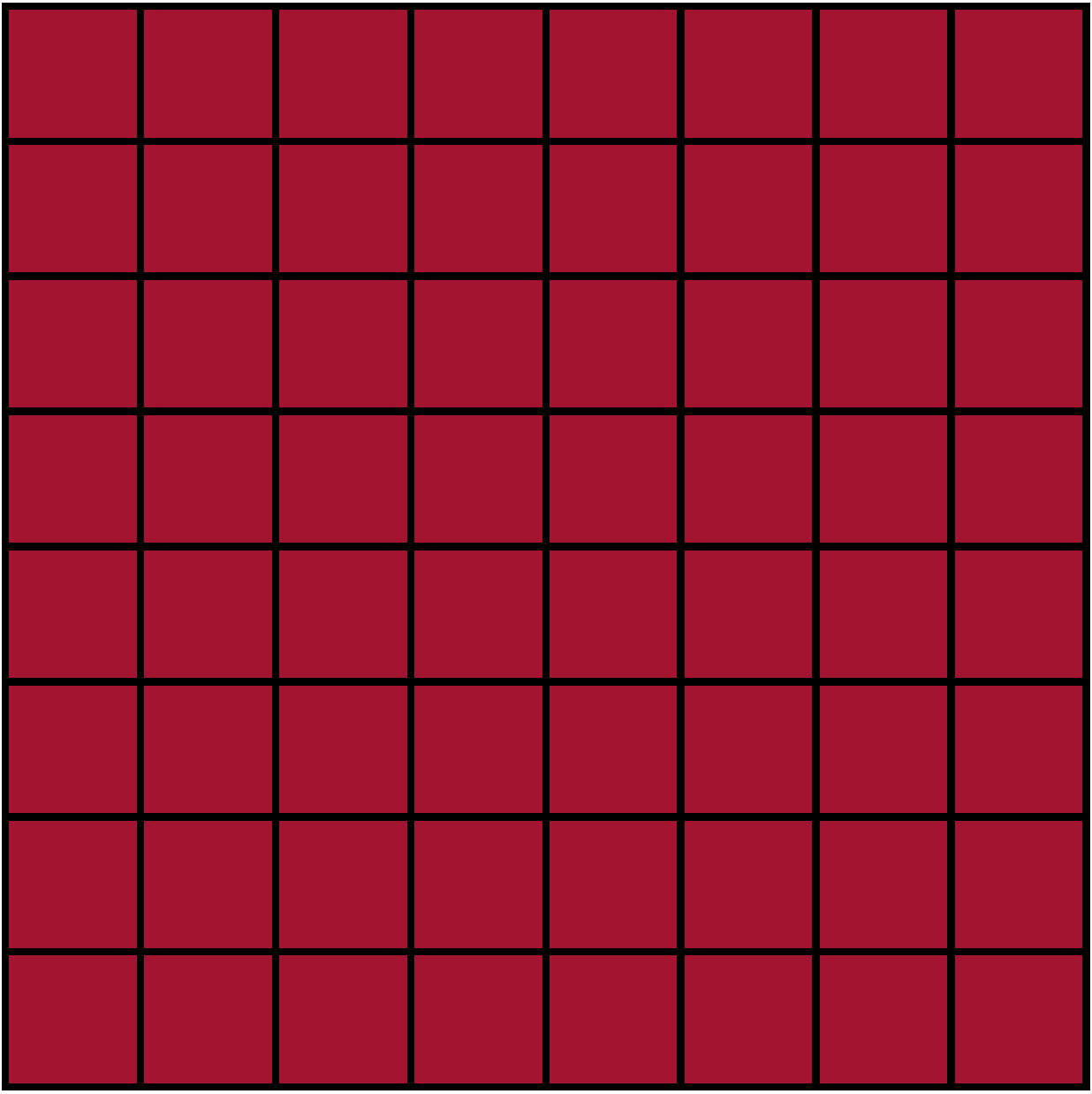}} 
    \subfigure[Crossed.\label{fig:crossed}]{\includegraphics[width=0.19\textwidth]{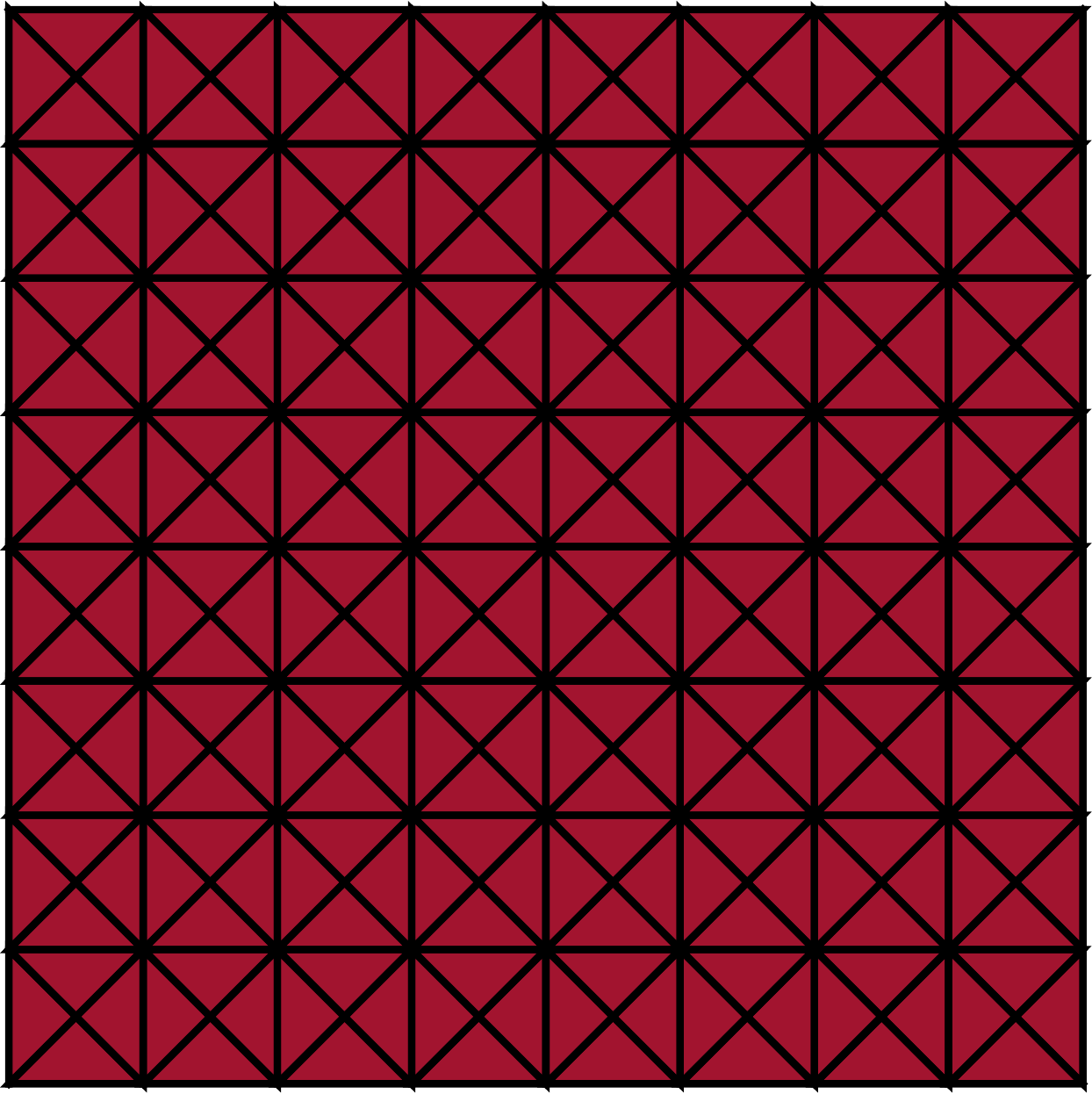}} 
    \subfigure[L.\label{fig:LShape}]{\includegraphics[width=0.19\textwidth]{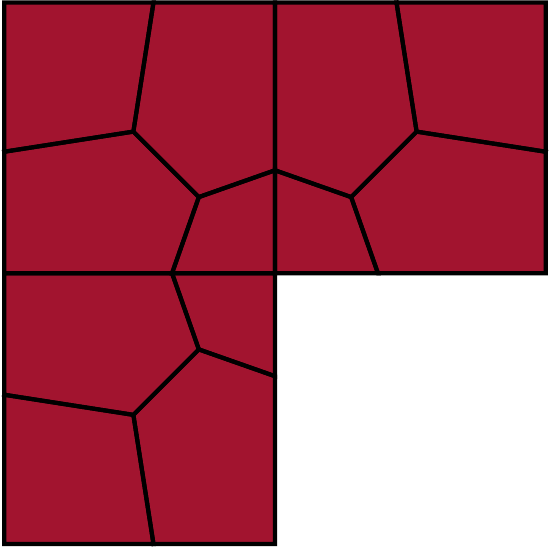}}  
    \subfigure[Australia.\label{fig:australia1}]{\includegraphics[width=0.19\textwidth]{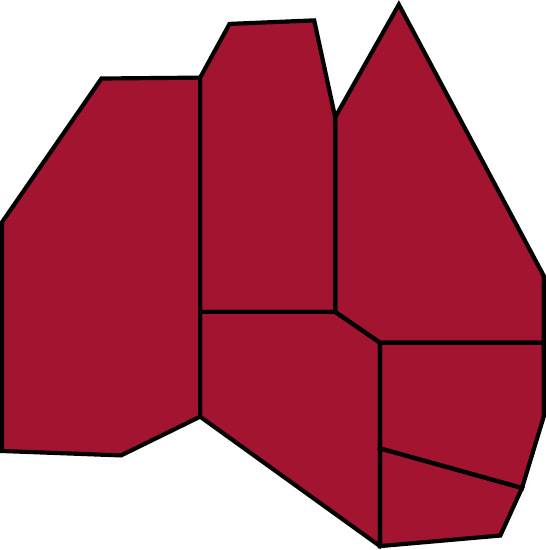}}\\  
    \subfigure[Cube.\label{fig:cube}]{\includegraphics[width=0.24\textwidth,trim={3.25cm 0cm 3.25cm .5cm},clip]{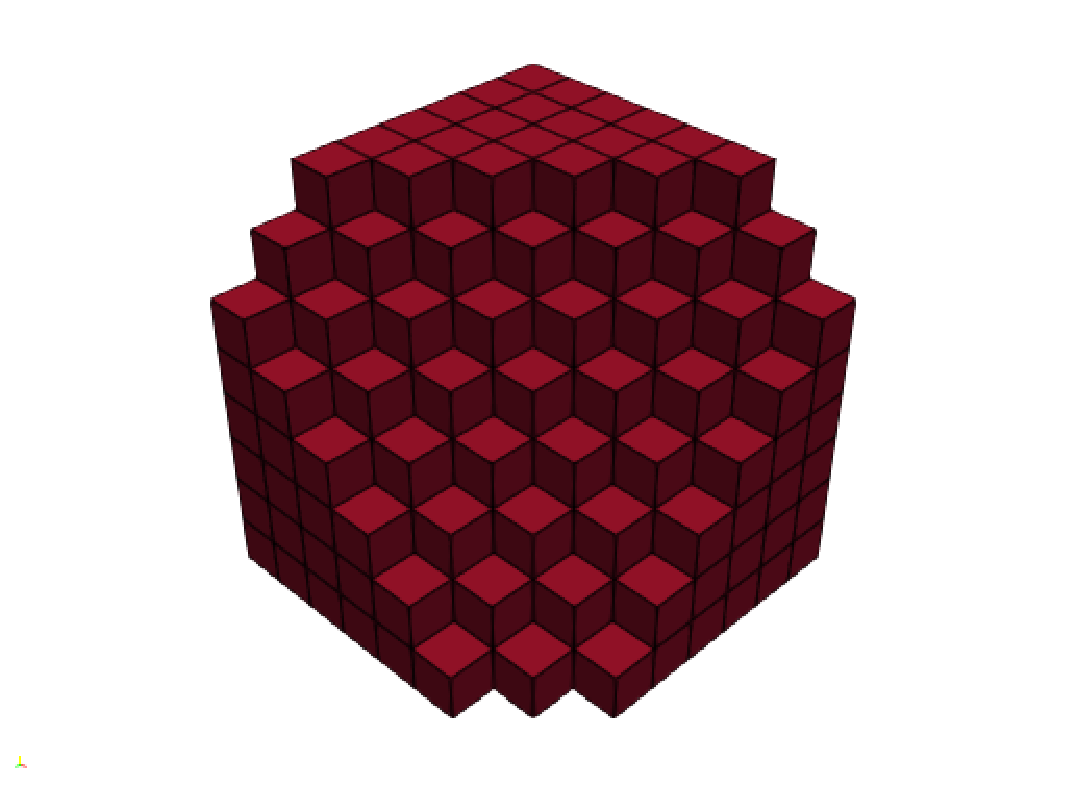}}
    \subfigure[Nove.\label{fig:novedri}]
    {\includegraphics[width=0.24\textwidth,trim={3.25cm 0cm 3.25cm .5cm},clip]{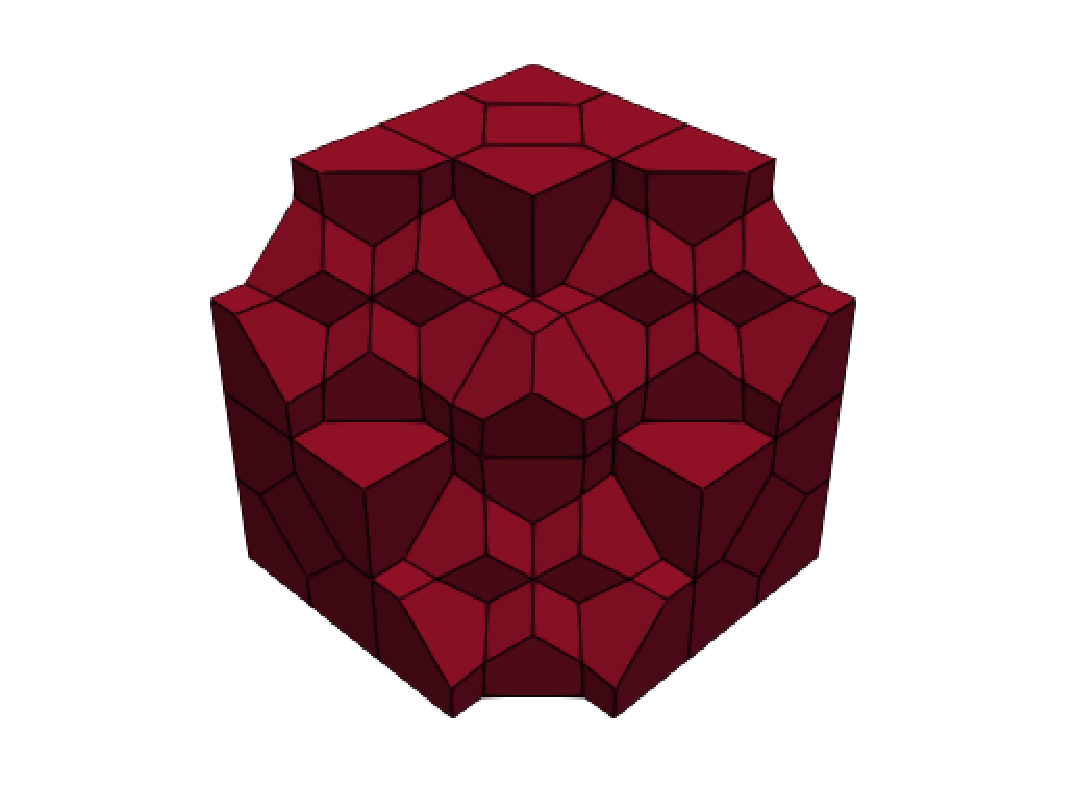}}
    \subfigure[Octa.\label{fig:octa}]
    {\includegraphics[width=0.24\textwidth,trim={3.25cm 0cm 3.25cm .5cm},clip]{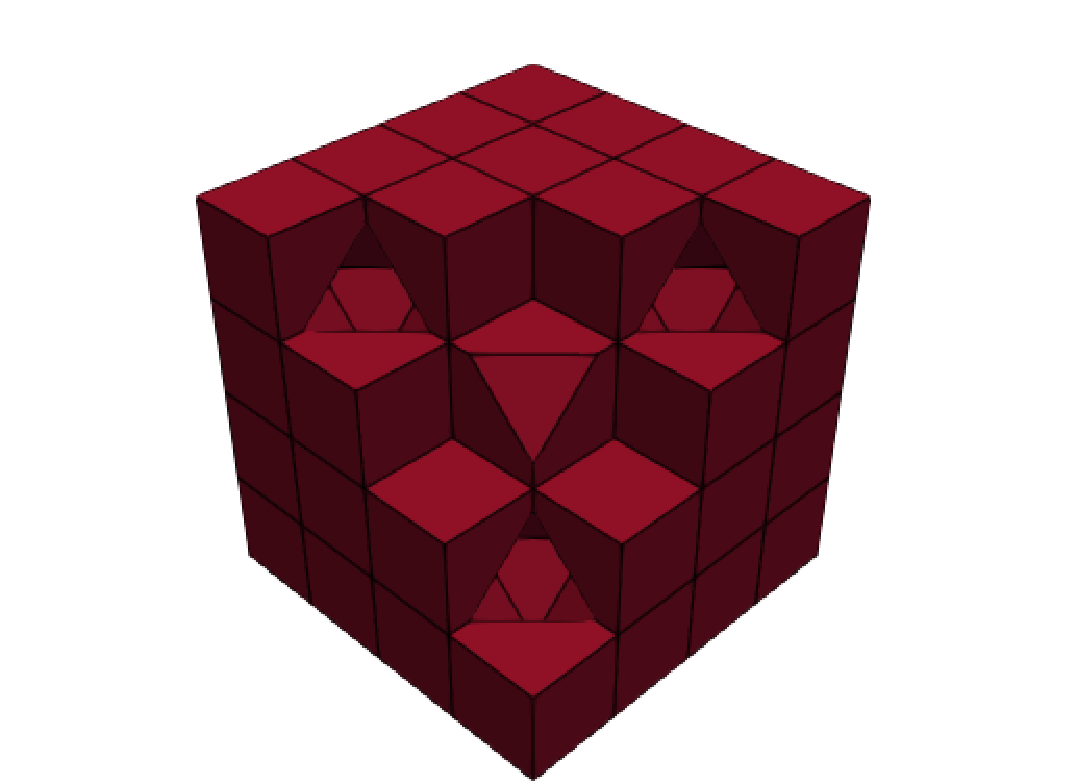}}
    \subfigure[Voro.\label{fig:voro}]
    {\includegraphics[width=0.24\textwidth,trim={3.25cm 0cm 3.25cm .5cm},clip]{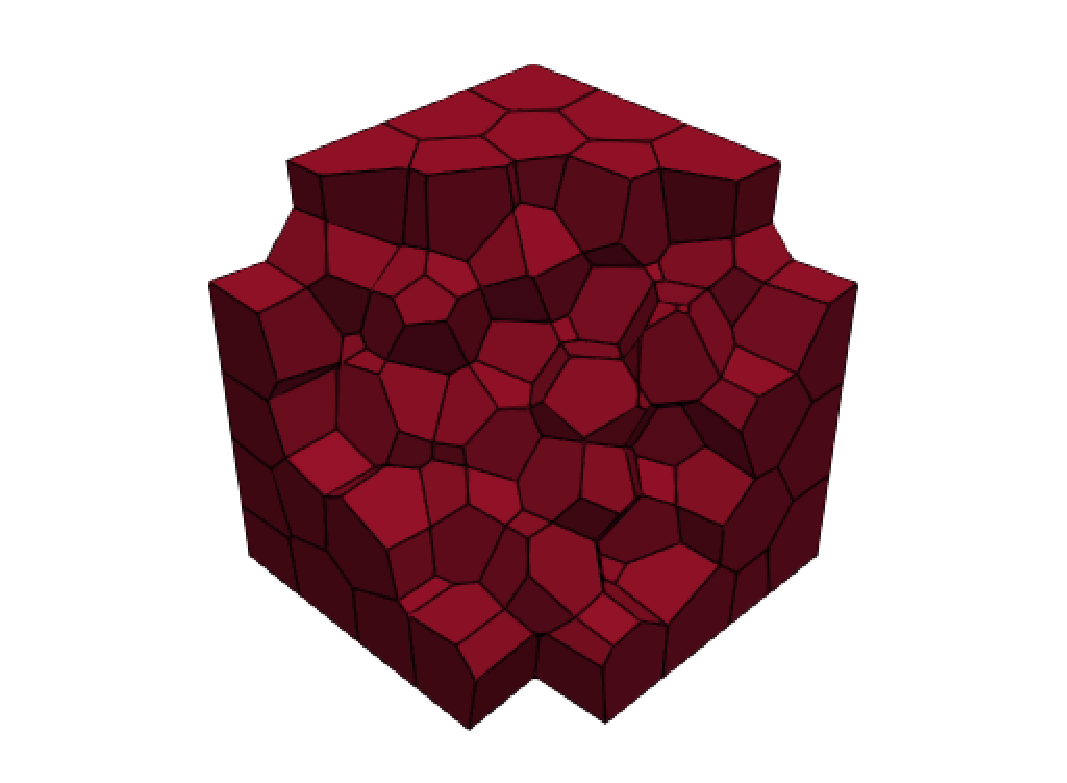}}
    
    \caption{An illustration of the distinct meshes used in the examples.\label{fig:meshes-example-1}}
\end{figure}

The error will be computed as usual in the VEM framework using the local polynomial approximation of the solution as follows
    \[\overline{\textnormal{e}}_*^2 = \norm{(\bu-\bPi_1^{\beps,k_1}\bu_h,{p}-{p}_h)}_{\bV_1^{h,k_1}\times \mathrm{Q}_1^{h,k_1}}^2 + \norm{(\bzeta-\bPi_2^{0,k_2}\bzeta_h,\varphi-\varphi_h)}_{\bV_2^{h,k_2}\times \mathrm{Q}_2^{h,k_2}}^2.\]
The numerical implementation has been done with the library \texttt{VEM++} \cite{dassi2023vem++}. We separately implemented the elasticity and reaction-diffusion equations for arbitrary orders $k_1$ and $k_2$, respectively. Then, the nonlinear coupled problem \eqref{eq:weak-discrete} is solved with an optimised Picard iteration, following the same structure as in the fixed-point analysis from  \cite{khot2024}, with a tolerance of $10^{-6}$. In this optimised version, the blocks corresponding to the coupling terms ($G_1^{\varphi_h}(\cdot)$ and $a_2^{\overline{\bu}_h,{p},h}(\cdot,\cdot)$) are the only ones that are reconstructed on each fixed point iteration. We recall that the elasticity and reaction-diffusion equations have sufficiently smooth forcing and source terms that will be manufactured according to the given exact solutions. The non-homogeneous boundary conditions require a slight modification of the right-hand side functionals as well as 
the estimator boundary term $\Xi_{1,E}$ regarding $\Gamma_{\mathrm{N}}$. The experimental order of convergence $r(\cdot)$ applied to either error $\overline{\textnormal{e}}_*$ or estimator $\Theta$ and the effectivity index (following Theorem~\ref{th:upper-bound}) of the refinement $1\leq j$ are computed from the formulae $r(\cdot)^{j+1} = -d \log\left((\cdot)^{j+1}/(\cdot)^{j}\right)/\log\left(\textnormal{DOF}_*^{j+1}/\textnormal{DOF}_*^{j}\right)$ and $\textnormal{eff}^{j} = \Theta^j/\overline{\textnormal{e}}_*^j$, with $\textnormal{DOF}_*$ indicating the total number of DoFs. In turn, the stabilisation term $S_1^E(\bu_h,\bv_h)$ follows the ``diagonal recipe" introduced in \cite{dassi2017stab} (weighted by $2\mu$) and $S_2^{\overline{\bu}_h,{p}_h,E}(\bzeta_h,\bxi_h)$ have into account the nonlinearity $\bbM^{-1}(\beps(\overline{\bu}_h),{p}_h)$ as in \cite{khot2024}. In the following experiments, we set $M \geq \max_{(x,y)\in \Omega} \quad \norm{\bbM(\beps(\bu),p)}_{F}$, where the Frobenius norm $\norm{\bbM(\beps(\bu),p)}_{F}$ is computed by the \texttt{fmincom} MATLAB optimisation subroutine.

The adaptive algorithm follows the usual strategy:
$\textnormal{SOLVE} \rightarrow \textnormal{ESTIMATE} \rightarrow \textnormal{MARK} \rightarrow \textnormal{REFINE}$.
The first three steps are performed inside \texttt{VEM++}, exploiting the efficiency capabilities of C++. For the REFINE step, we use the Matlab-based method from \cite{yu2021implementationpolygonalmeshrefinement}, which connects edge midpoints to the polygon barycentre (in 2D). In turn, the adaptive refinement used in this paper relies on the library \texttt{p4est} \cite{p4est} executed through the module GridapP4est of the Julia package \texttt{Gridap} \cite{gridap}. The \texttt{VEM++} code is executed through Matlab/Julia (2D/3D, respectively), generating a list of elements to refine, which is then processed by the refinement routine. These refinement routines restrict the mesh elements to convex polygons in 2D and cubes (with hanging nodes/faces) in 3D. However, we recall that this procedure is independent of the SOLVE, ESTIMATE, and MARK stages. Therefore, the implementation can be extended to general star-shaped polygons with more general refinement routines. 
For the MARK procedure we follow a D\"orfler/Bulk  strategy, marking the subset of mesh elements $\mathcal{K} \subseteq \mathcal{T}^h$ with the largest estimated errors such that for $\delta = \frac{1}{2}\in [0,1]$, we have $\delta \sum_{E\in \mathcal{T}^h} \Theta_E^2 \leq  \sum_{E\in \mathcal{K}} \Theta_E^2.$
\begin{figure}[htbp!]
    \centering
    \includegraphics[width=0.49\textwidth,trim={4.4cm 2.75cm 5.cm 15.7cm},clip]{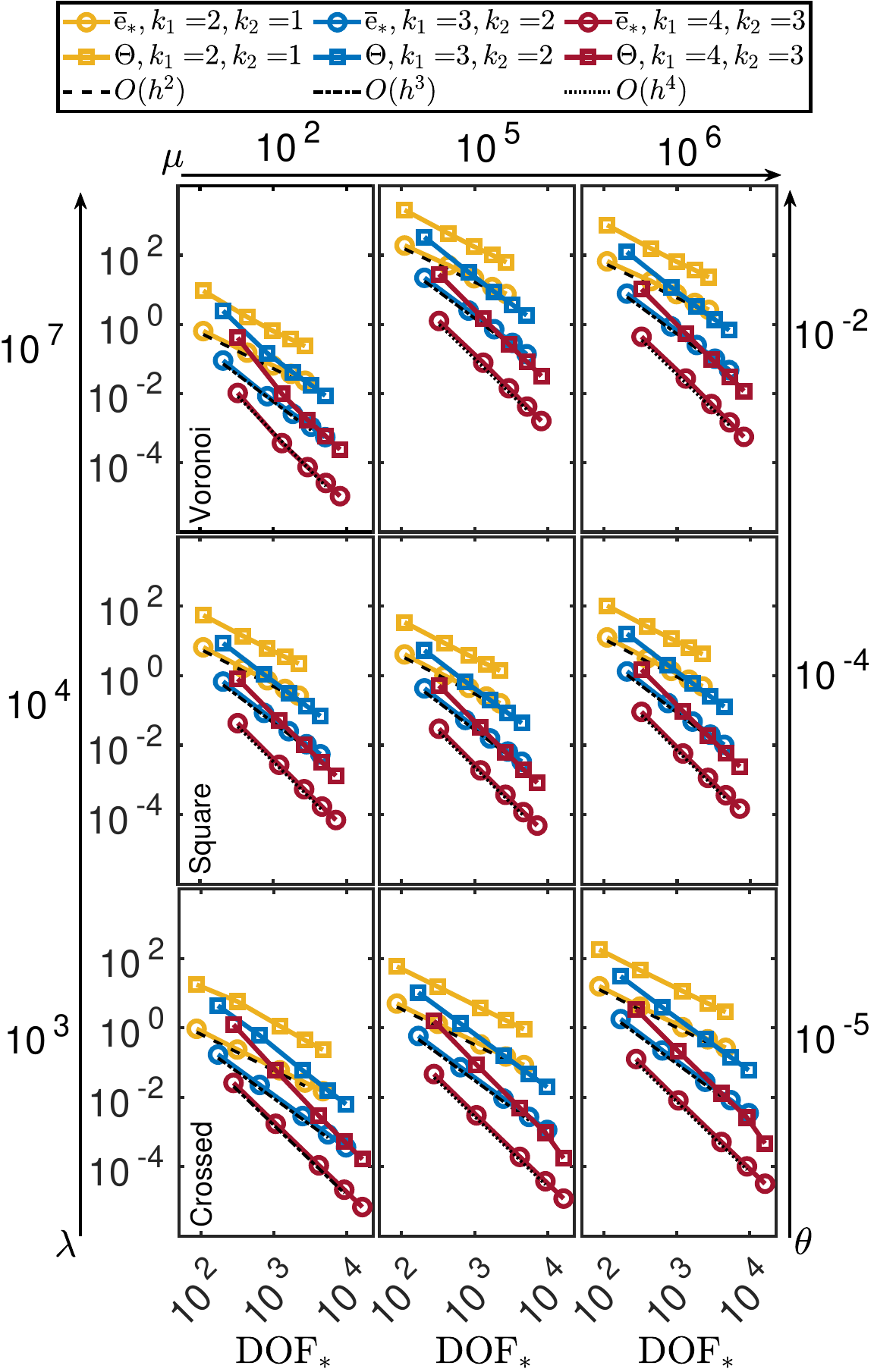}
    \includegraphics[width=0.4875\textwidth,trim={5.cm 1.75cm 4.3cm 16.7cm},clip]{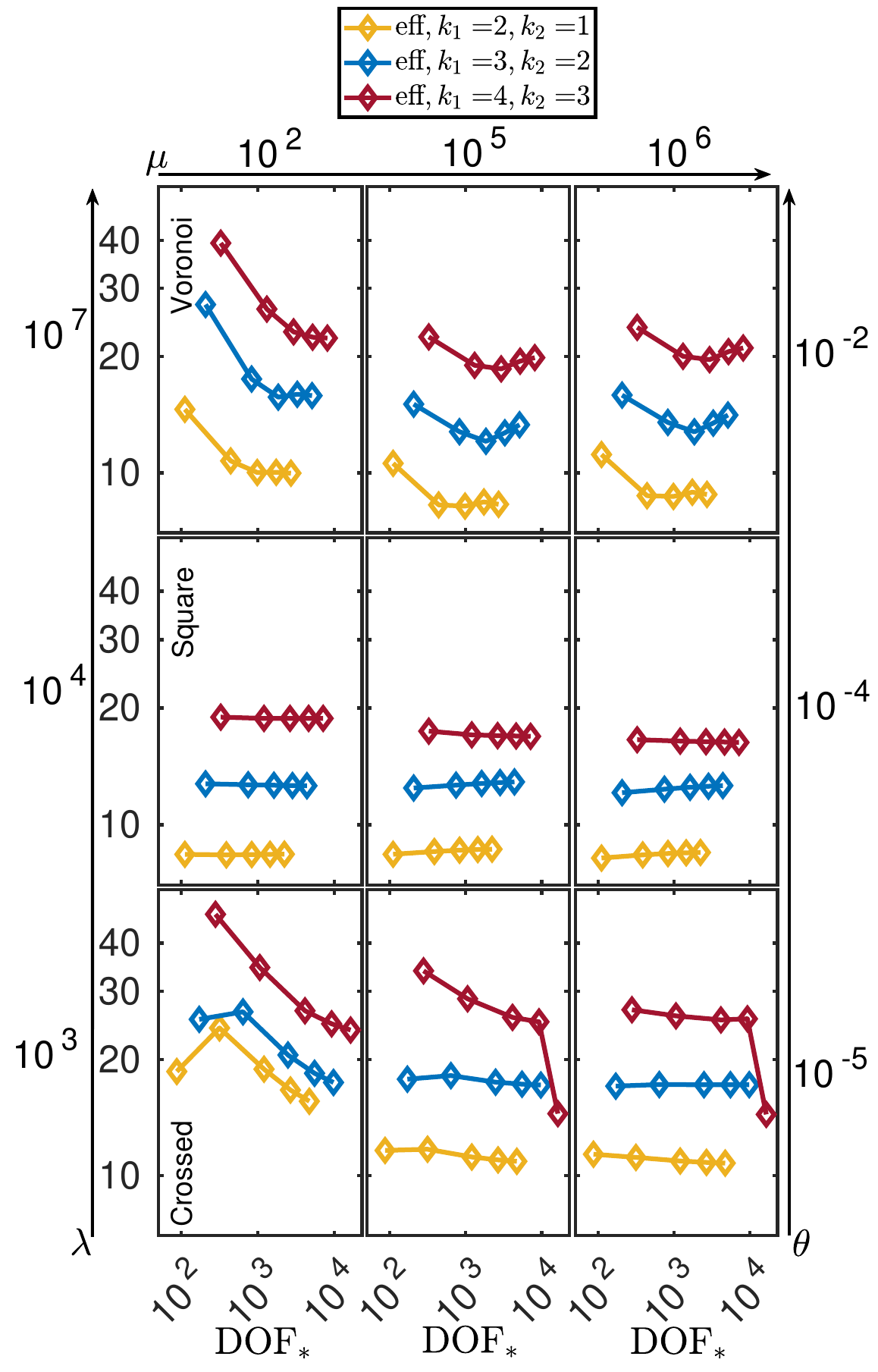}

    \caption{Example 1. Behaviour of the error $\overline{\textnormal{e}}_*$, estimator $\Theta$ (left), and effectivity index $\textnormal{eff}$ (right)  under uniform refinement across various meshes, polynomial orders, and parameter selections.\label{fig:uniform_graph}}
\end{figure}
\subsection{Example 1: Robust behaviour of the estimator under uniform refinement}\label{sec:numerical-example-1}
We consider the meshes in Fig~\ref{fig:meshes-example-1} for the unit square domain $\Omega = (0,1)^2$ with Dirichlet part $\Gamma_{\mathrm{D}} = \{(x,y) \in \partial \Omega \colon x=0\; \text{or}\; y = 0\}$, and Neumann part $\Gamma_{\mathrm{N}} = \partial \Omega\setminus \Gamma_{\mathrm{D}}$. We explore different polynomial degrees for the elasticity and reaction-diffusion equations given by $k_1=2,3,4$ and $k_2=1,2,3$. Smooth manufactured solutions and nonlinearities are set as
\begin{gather*}
    \bu(x,y) = 5^{-1}\left( x^2+x\cos(x)\sin(y), y^2+x\cos(y)\sin(x) \right)^{\tt t},\quad \ell(\varphi) = \varphi,\\ 
    \varphi(x,y) =\cos(\pi y) + \sin(\pi x) + x^2 + y^2,\quad
    \bbM(\beps(\bu),p) = 10^{-1}\exp{\left[-10^{-8} \tr(2\mu \beps(\bu) - p \bbI)\right]}\bbI.
\end{gather*}
As usual, the exact displacement $\bu$ and concentration $\varphi$ are used to compute exact Herrmann pressure ${p}$ and total flux $\bzeta$, as well as appropriate forcing term, source, and non-homogeneous traction, displacement, and flux boundary data. Finally, we use the parameter value $M=12$.

The results in Figure~\ref{fig:uniform_graph} demonstrate robust optimal convergence rates predicted by Theorem~\ref{convergence-rates} for $\overline{\textnormal{e}}_*$ and $\Theta$ under uniform refinement for a variation in the involved parameters $\mu, \lambda, \theta$. In addition, Figure~\ref{fig:uniform_graph} shows the effectivity index for the different experiments, here we can observe that it oscillates around a fixed number for each case, confirming the reliability of the estimator given by Theorem~\ref{th:upper-bound}.

\subsection{Example 2: Non-smooth solution on non-convex domains}\label{sec:numerical-example-2}
In this case, we set an initial polygonal discretisation for the  L-shaped domain $\Omega = (-1,1)^2 \setminus [0,1) \times [0,-1)$ with boundary conditions defined on  $\Gamma_{\mathrm{N}} = \{ (x,y) \in \partial \Omega : x=-1,\, y=1\}$ and $\Gamma_{\mathrm{D}} = \partial \Omega \setminus \Gamma_N$ (see Figure~\ref{fig:LShape}). In contrast, the Australia-shaped domain is set to fit inside the unit square where each element in the initial discretisation represents the in-land states (See Figure~\ref{fig:australia1}), the further east edge coincides with $x=0$ where $\Gamma_{\mathrm{N}}$ is defined, whereas $\Gamma_{\mathrm{D}} = \partial \Omega \setminus \Gamma_N$.
\begin{figure}[ht!]
    \centering 
    \includegraphics[width=0.45\textwidth,trim={12.cm 3.25cm 11.5cm 3.25cm},clip]{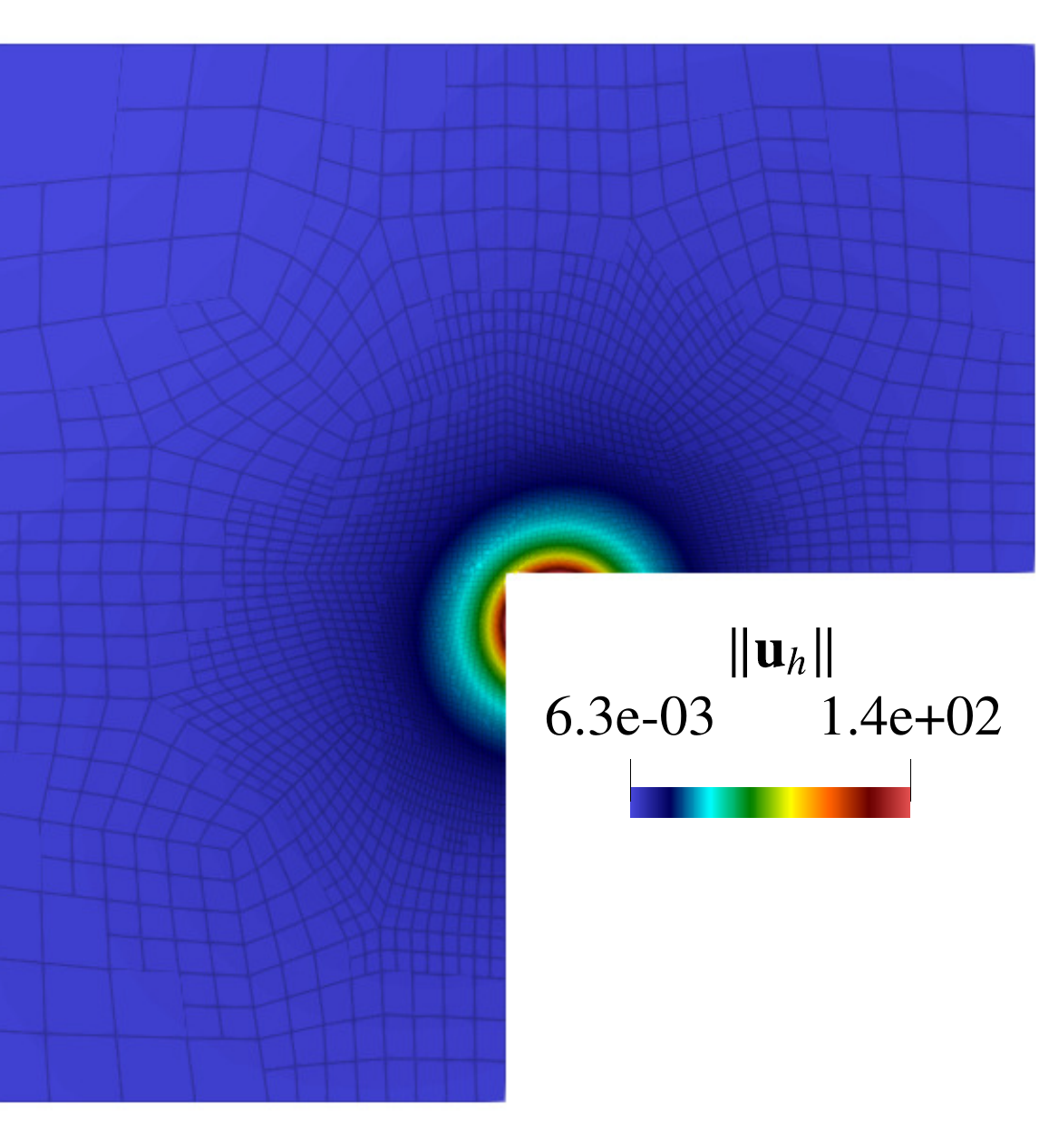}  \label{fig:L_displacement}
    \includegraphics[width=0.45\textwidth,trim={12.cm 3.25cm 11.5cm 3.25cm},clip]{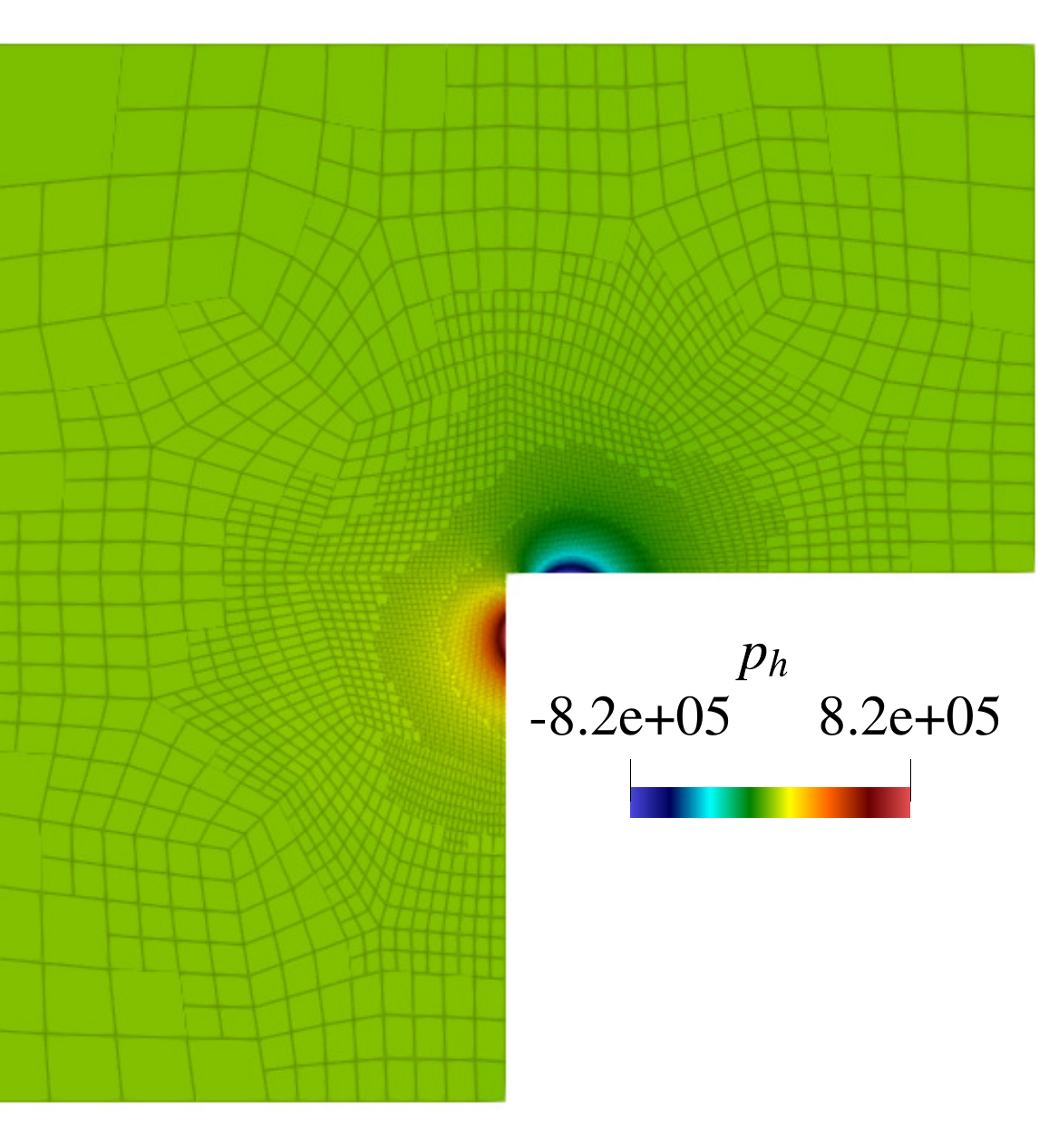} \label{fig:L_pressure}
    \includegraphics[width=0.45\textwidth,trim={12.cm 3.25cm 11.5cm 3.25cm},clip]{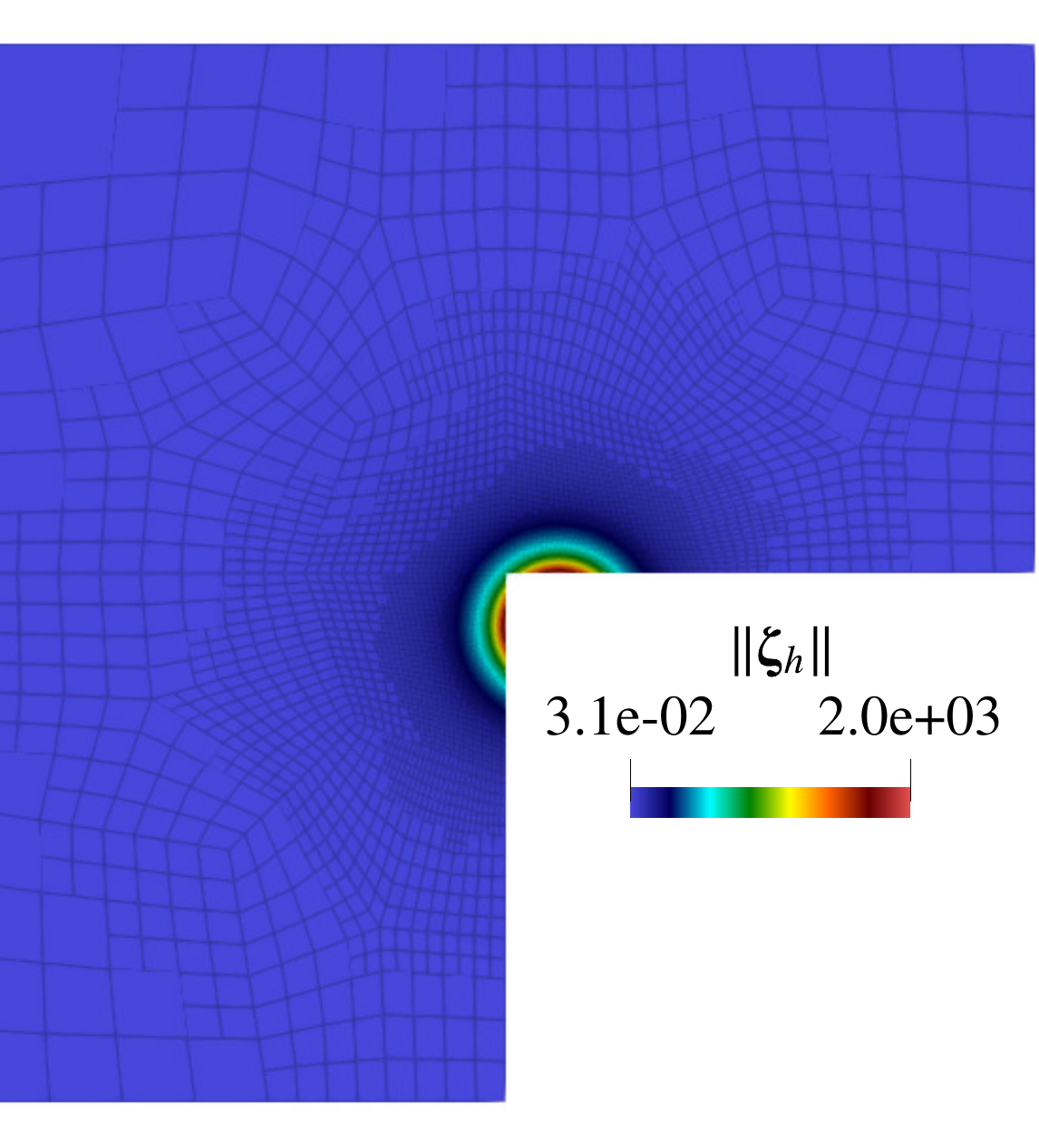} \label{fig:L_flux}
    \includegraphics[width=0.45\textwidth,trim={12.cm 3.25cm 11.5cm 3.25cm},clip]{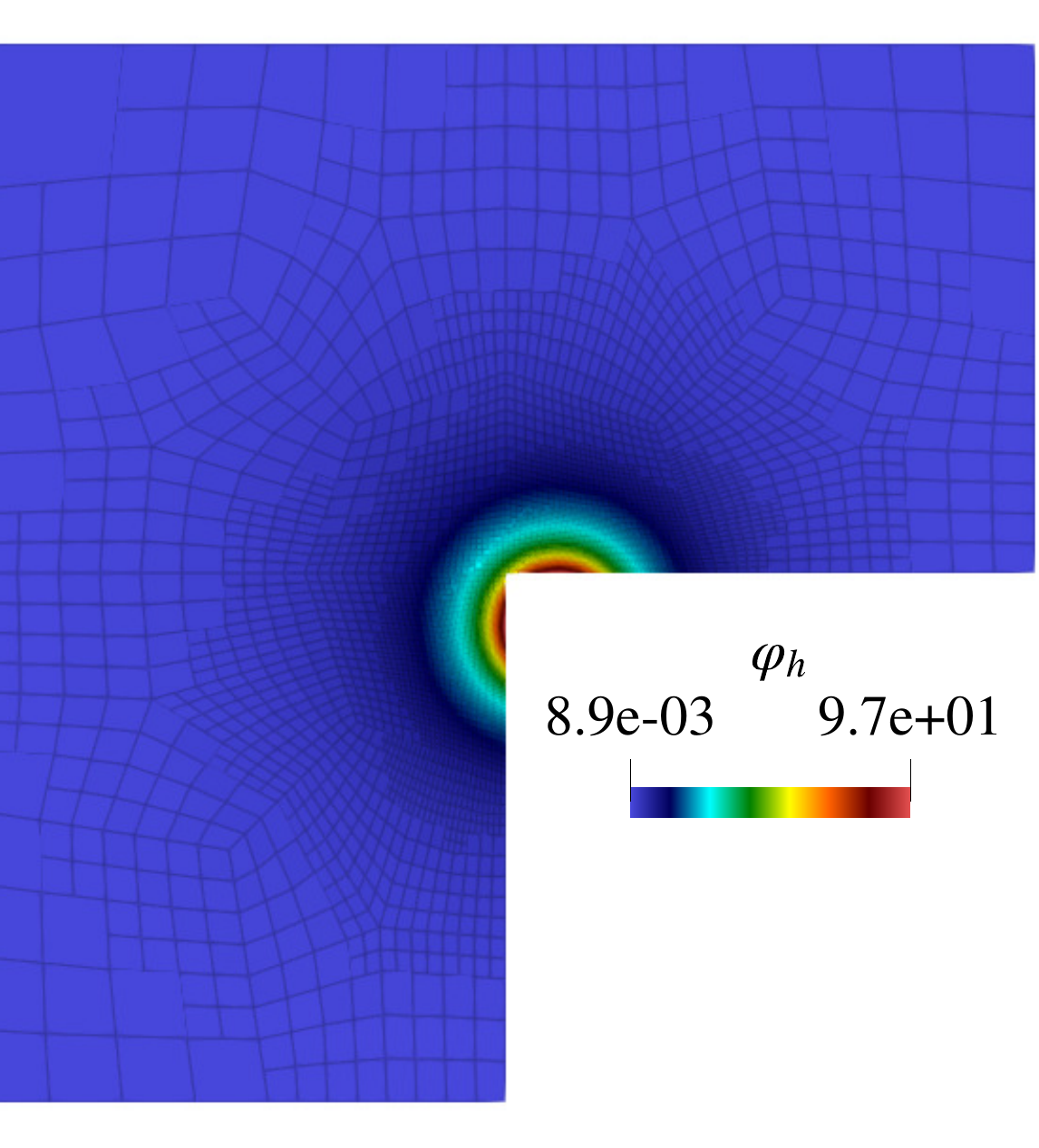} \label{fig:L_concentration}
    \caption{Example 2. Snapshots of the interest variables on the L-shape mesh are shown for polynomial degrees $k_1=2$ and $k_2=1$ after 19 refinement steps.\label{fig:solutions-example-2-L}}
\end{figure}
\begin{figure}[ht!]
    \centering  
    \includegraphics[width=0.41\textwidth,trim={1.25cm 0.25cm 1.7cm 0.cm},clip]{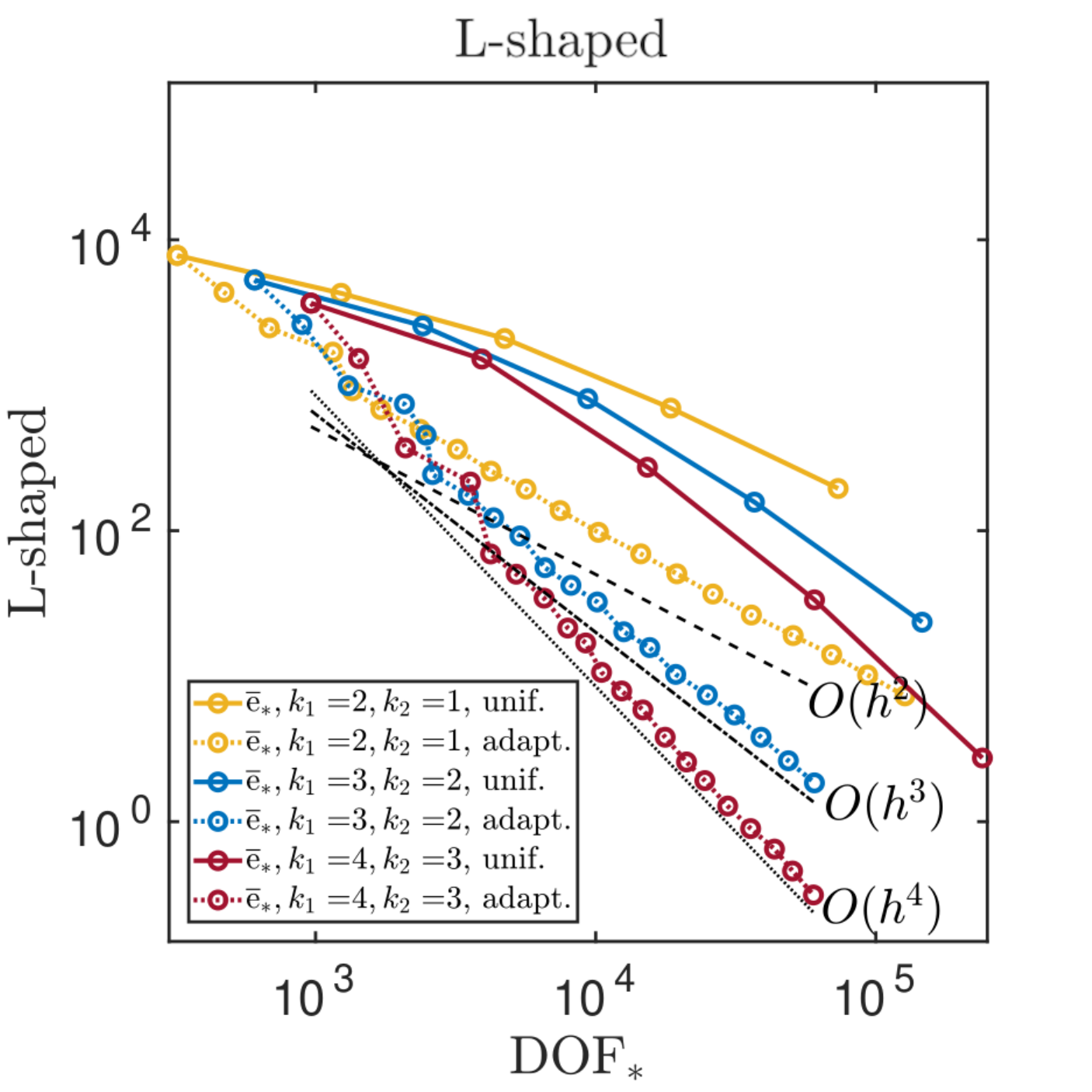}  
    \includegraphics[width=0.41\textwidth,trim={1.25cm 0.25cm 1.7cm 0.cm},clip]{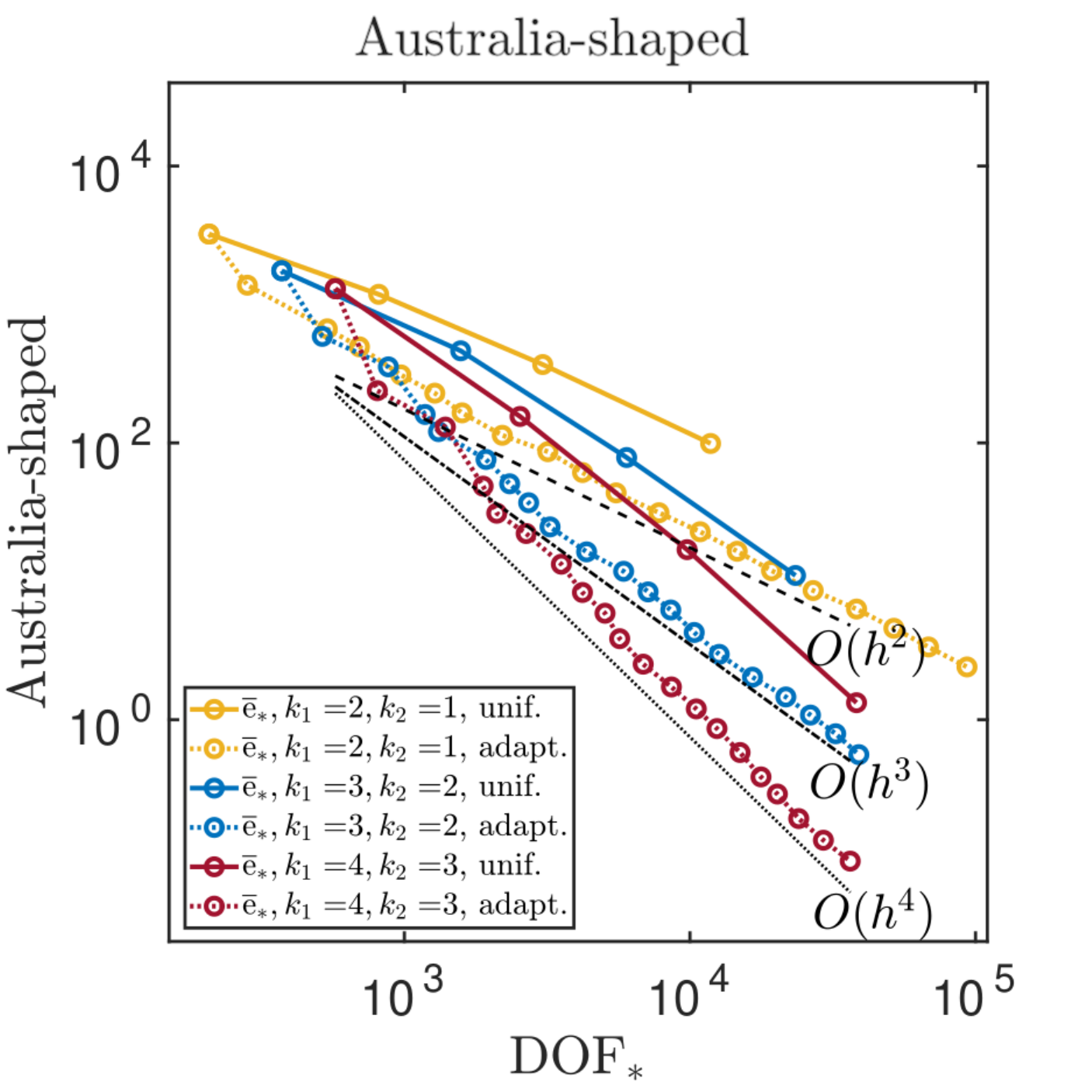} 
    \includegraphics[width=0.4\textwidth,trim={1.05cm 0.25cm 1.7cm 0.cm},clip]{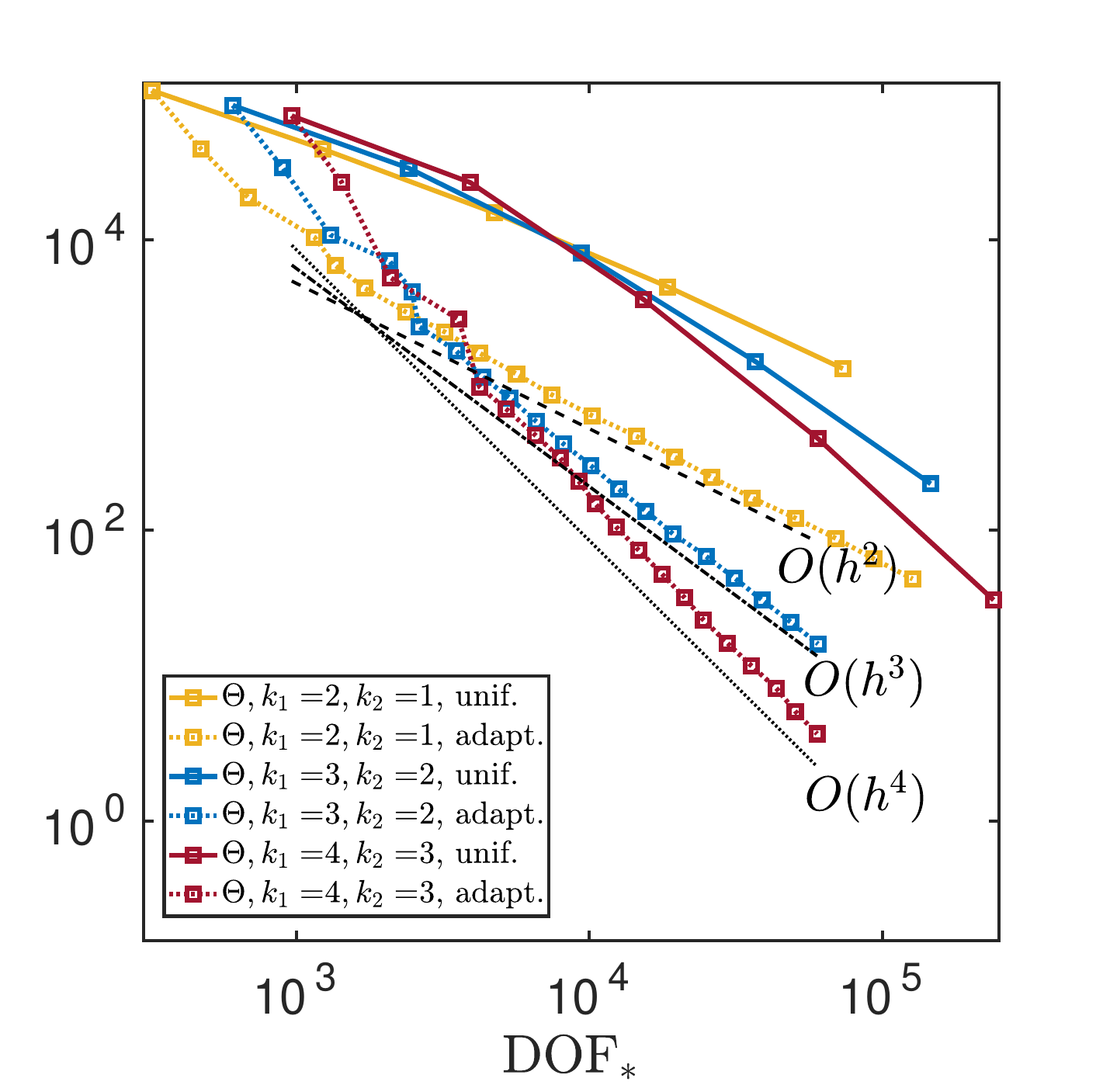}
    \includegraphics[width=0.4\textwidth,trim={1.05cm 0.25cm 1.7cm 0.cm},clip]{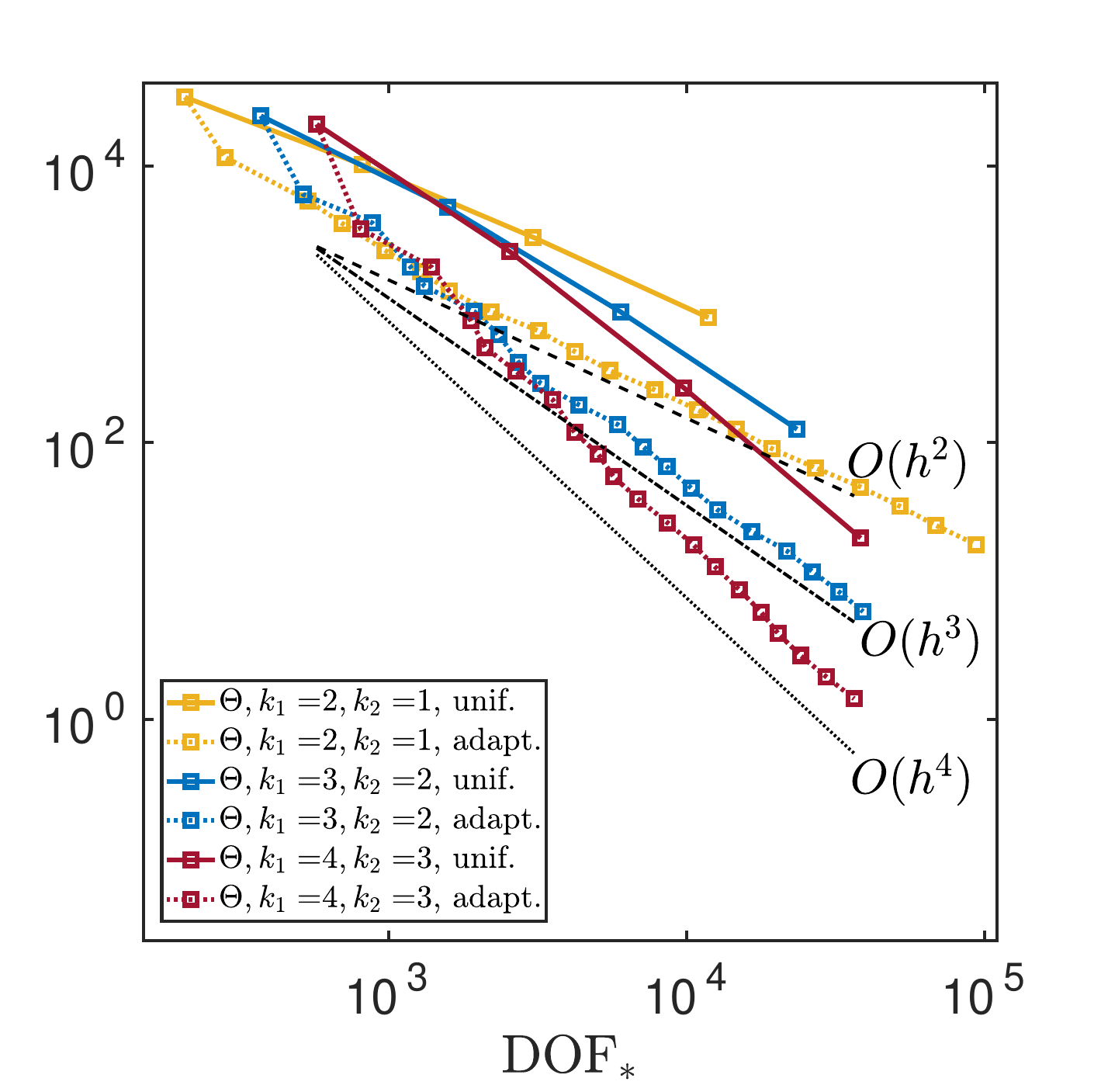}  
    \includegraphics[width=0.4\textwidth,trim={1.05cm 0.25cm 1.7cm 0.cm},clip]{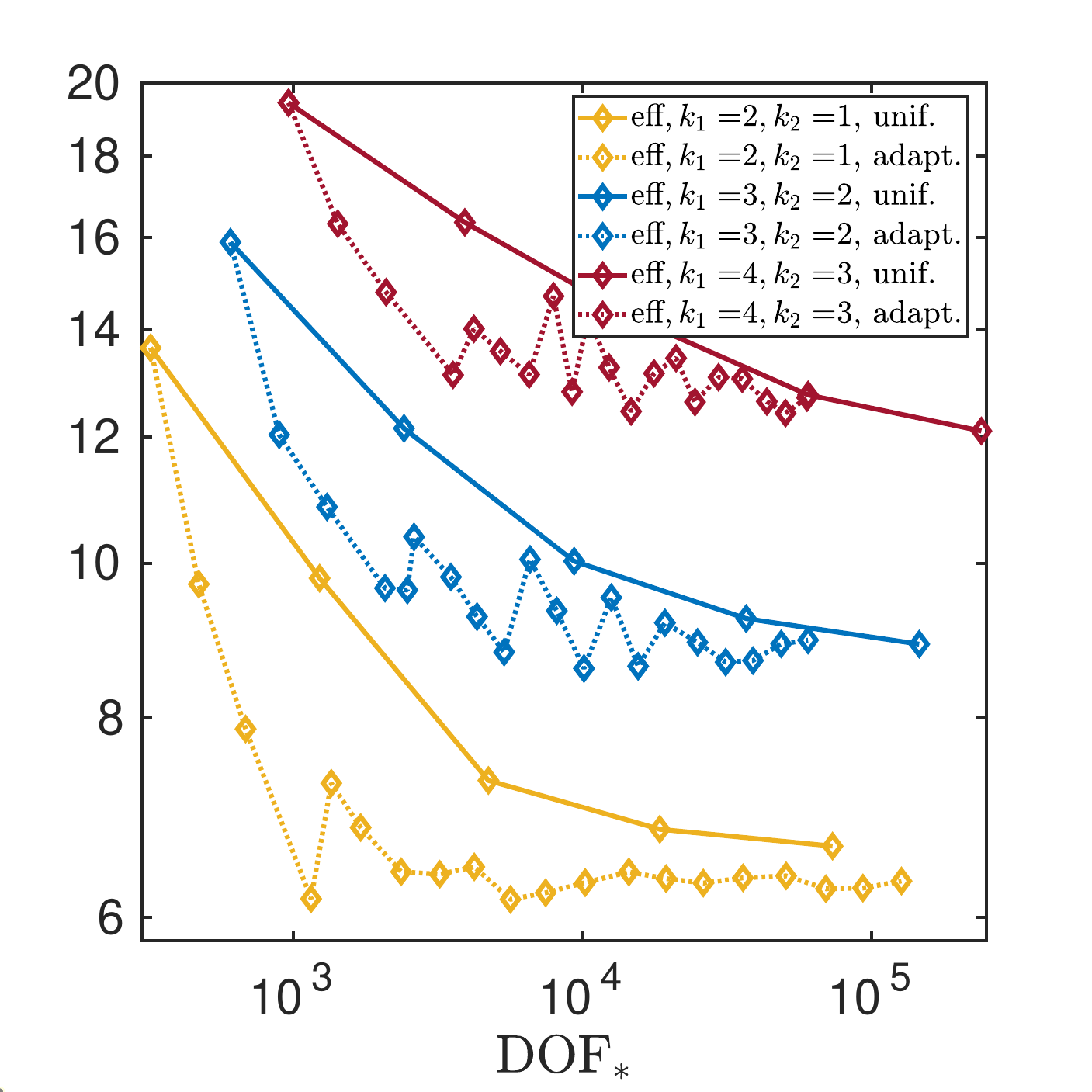}
    \includegraphics[width=0.4\textwidth,trim={1.05cm 0.25cm 1.7cm 0.cm},clip]{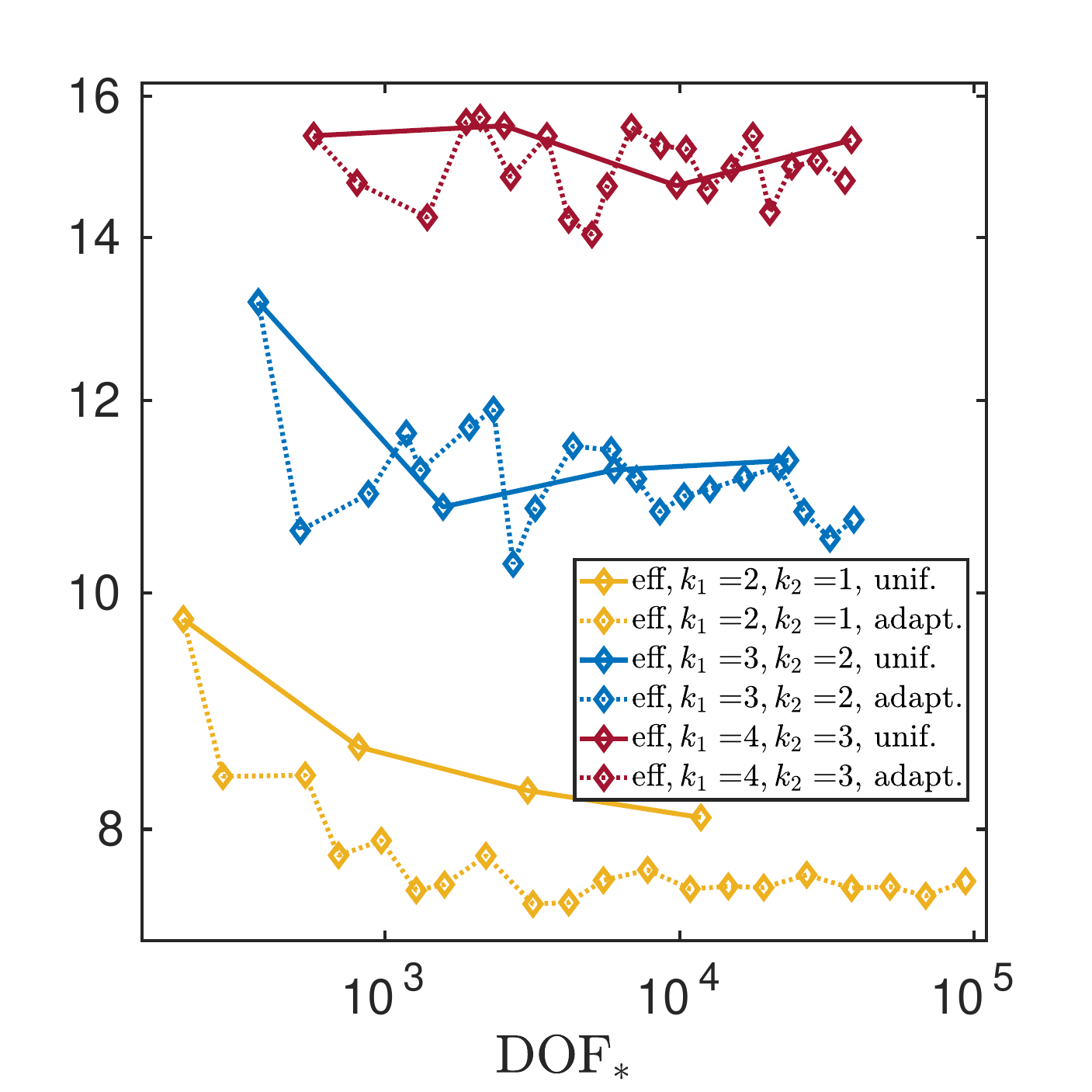} 
    
    \caption{Example 2. Behaviour of $\overline{\textnormal{e}}_*$ (top row), $\Theta$ (middle row), and $\textnormal{eff}$ (bottom row) under uniform and adaptive refinement on L-shape (left column) and Australia-shape (right column) domains for a variety polynomial orders.\label{fig:uniform_adaptive_graph}}
\end{figure}

The non-smooth manufactured solutions and the non-linear terms are set as follows
\begin{gather*}
    \bu(x,y) = \left( \frac{(x-1.0)(y-1.0)}{ \left(x - x_* \right)^2 + \left( y - y_* \right)^2 }, \frac{(x+1.0)(y+1.0)}{ \left(x - x_* \right)^2 + \left( y-y_* \right)^2 } \right)^{\tt t},\quad \ell(\varphi) = 2 + \frac{\varphi^2}{1+\varphi^2},
    \\
     \varphi(x,y) = \frac{(x-1.0)(x+1.0)(y-1.0)(y+1.0)}{ \left(x - x_* \right)^2 + \left( y - y_* \right)^2},\quad  \bbM(\beps(\bu),p) = \left( 1 + \frac{10^{-5}}{\tr(2\mu \beps(\bu) - p \bbI)} \right)\bbI,
\end{gather*}
with  parameter values $\mu = 1.4286 \times 10^3$, $\lambda = 357.1429$, $\theta = 10^{-3}$, and $M = 2$. Displacement and  concentration admit a singularity depending on the selection of the point $(x_*,y_*)^{\tt t}$. For the L-shaped domain we set $(x_*,y_*)^{\tt t} = (10^{-1},-10^{-1})^{\tt t}$, expecting high gradients close to the reentrant corner. In contrast, the Australian-shape domain expects high gradients near the location of Melbourne, i.e., $(x_*,y_*)^{\tt t}=(8.5\times 10^{-1},-10^{-1})^{\tt t}$.

Note that high values of the total error $\overline{\textnormal{e}}_*$ are expected in this simulation because the manufactured solutions attain large magnitudes near the singularity. This effect is further amplified by the parameter-dependent norms used to measure the total error. Alternatives, such as computing relative errors (with respect to the energy of the problem), can be used to mitigate this observation. 

Figure~\ref{fig:uniform_adaptive_graph} shows the behaviour of $\overline{\textnormal{e}}_h$, $\Theta$, and $\textnormal{eff}$ under uniform and adaptive refinement for different polynomial degrees. As expected,  the error decreases faster (optimally) under the adaptive procedure, and the effectivity index remains bounded, confirming the robustness of the estimator.  Figures~\ref{fig:solutions-example-2-L},~\ref{fig:solutions-example-2-A} displays approximate 
solutions after 19 adaptive mesh refinement steps according to $\Theta$. Most of the refinement occurs around the singularities, demonstrating how the method identifies regions where accuracy 
deteriorates. 

\begin{figure}[ht!]
    \centering 
    \includegraphics[width=0.45\textwidth,trim={12.cm 3.25cm 11.5cm 3.25cm},clip]{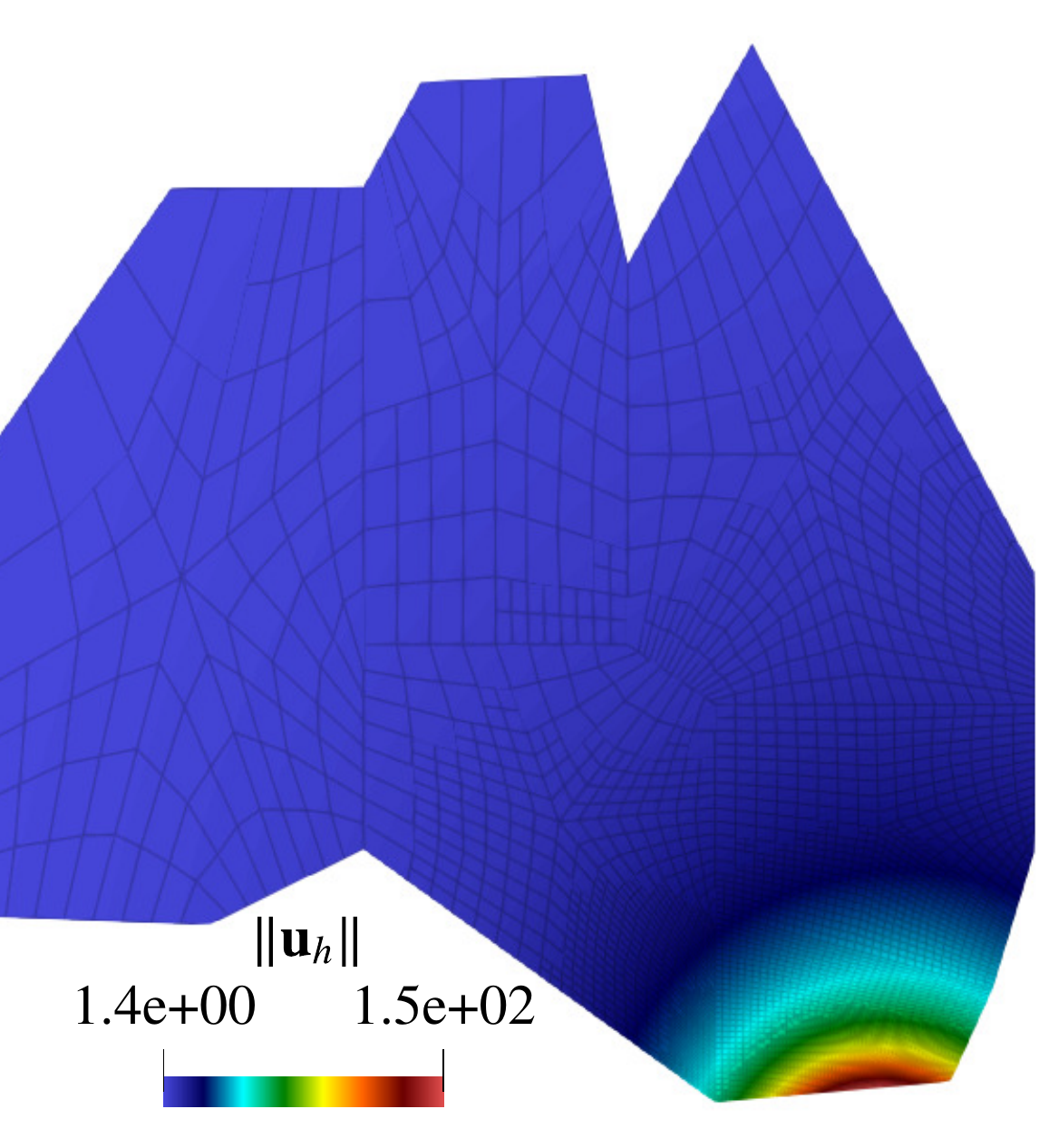}  \label{fig:australia_displacement}
    \includegraphics[width=0.45\textwidth,trim={12.cm 3.25cm 11.5cm 3.25cm},clip]{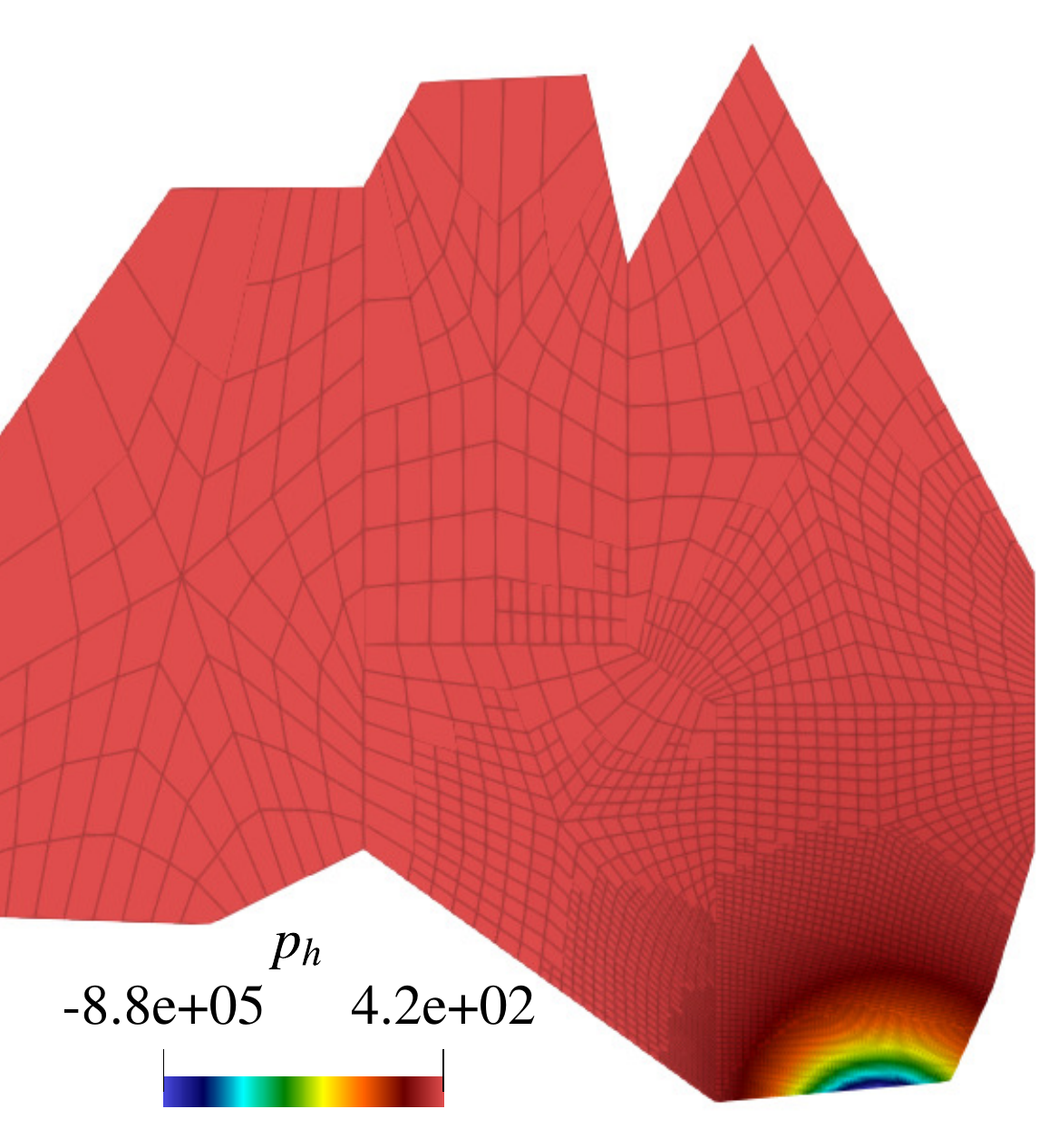} \label{fig:australia_pressure}
    \includegraphics[width=0.45\textwidth,trim={12.cm 3.25cm 11.5cm 3.25cm},clip]{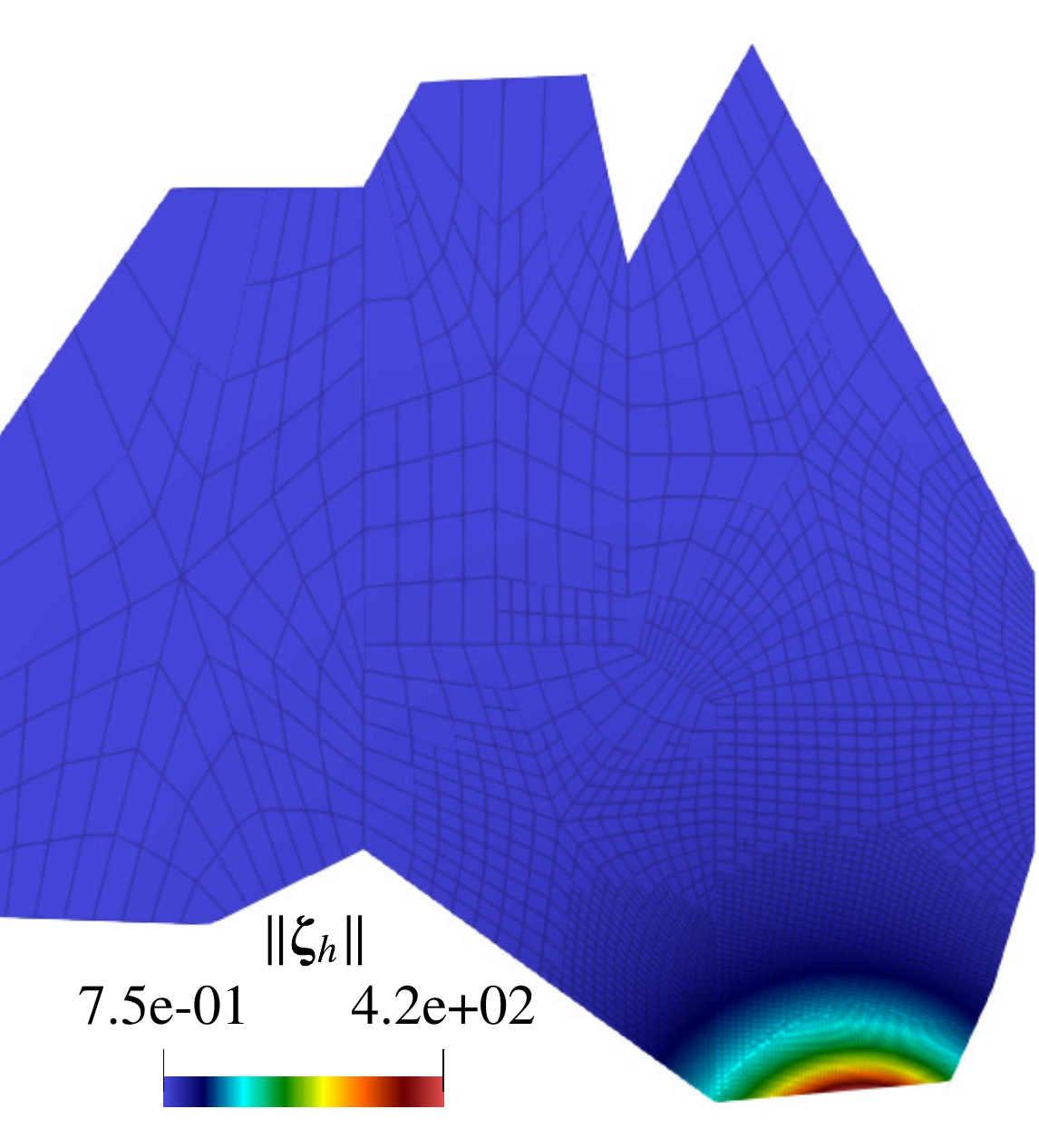} \label{fig:australia_flux}
    \includegraphics[width=0.45\textwidth,trim={12.cm 3.25cm 11.5cm 3.25cm},clip]{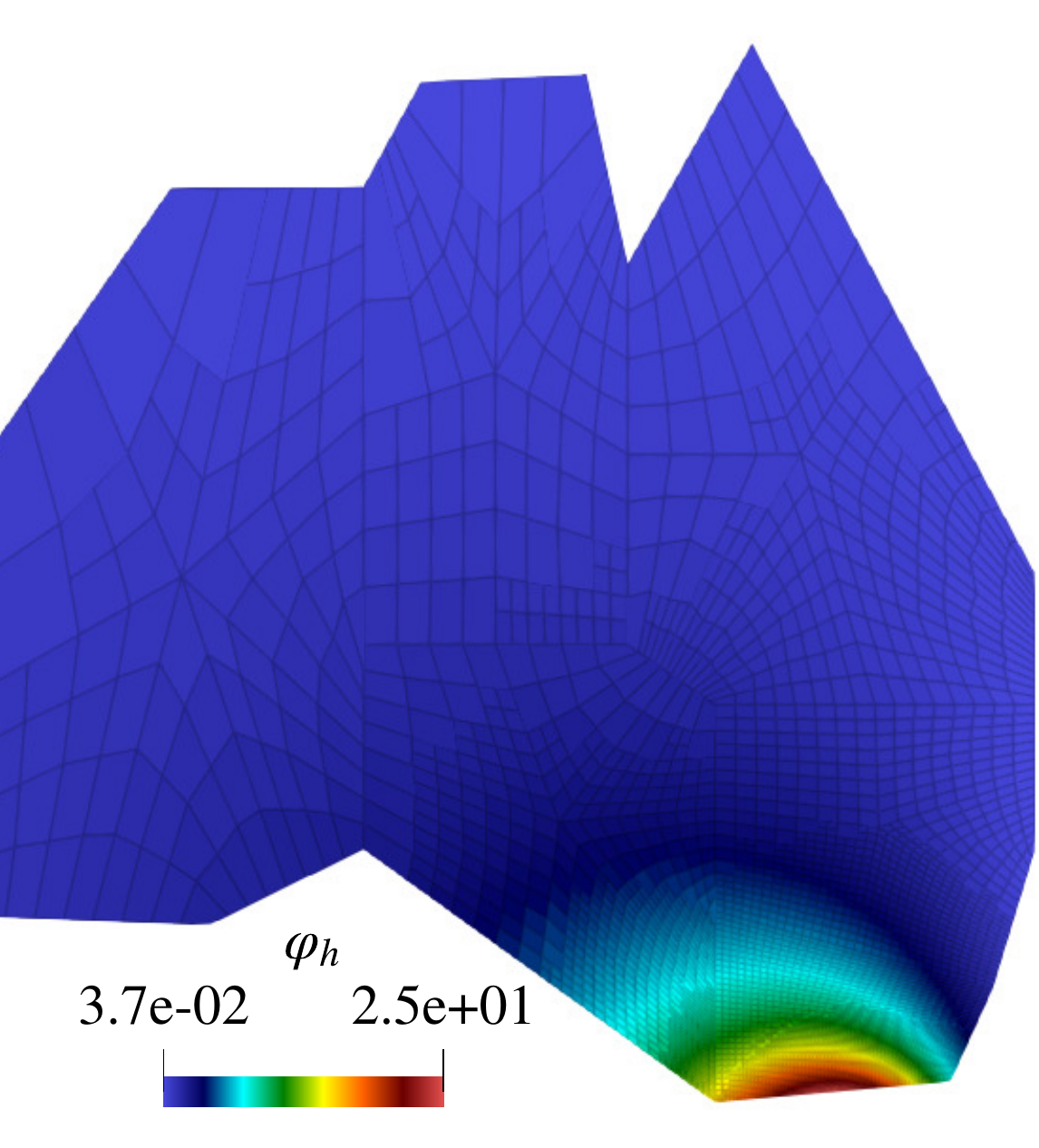} \label{fig:australia_concentration}
    \caption{Example 2. Snapshots of the interest variables on the Australia-shape mesh are shown for polynomial degrees $k_1=2$ and $k_2=1$ after 19 refinement steps.\label{fig:solutions-example-2-A}}
\end{figure}

\subsection{Example 3: Behaviour of estimator under uniform refinement in 3D}\label{sec:numerical-example-3} We consider the meshes in Figure~\ref{fig:meshes-example-1} for the unit cube domain $\Omega = (0,1)^3$ with Dirichlet part $\Gamma_{\mathrm{D}} = \{(x,y,z) \in \partial \Omega \colon x=0\; \text{or}\; y = 0\; \text{or}\; z=0\}$, and Neumann part $\Gamma_{\mathrm{N}} = \partial \Omega\setminus \Gamma_{\mathrm{D}}$. We set the polynomial degrees to $k_1 = 2$ and $k_2=2$ for the discrete spaces defined in Section~\ref{sec:discrete_spaces_3D}, noting that the implementation presented in this paper supports arbitrary polynomial degrees. For this example, we fix the adimensional parameters as $\mu = 10^2$, $\lambda = 10^3$, $\theta = 10^{-3}$, and $M=20$. On the other hand, the manufactured solutions and non-linear terms are given as follows
\begin{gather*}
    \bu(x,y,z) = 5^{-1}\left( x^2+x\cos(x)\sin(y), y^2+x\cos(y)\sin(x), z^2+x\cos(x)\cos(y) \right)^{\tt t},\\ 
    \varphi(x,y,z) =\cos(\pi y) + \sin(\pi x) + x^2 + y^2 + z^2, \\
    \ell(\varphi) = \varphi + \varphi^2, \quad \bbM(\beps(\bu),p) = 10^{-1}\exp{\left[-10^{-4} \tr(2\mu \beps(\bu) - p \bbI)\right]}\bbI.
\end{gather*}
We recall that the right-hand sides and non-homogeneous boundary conditions are introduced in the computations according to these terms. 
\begin{figure}[ht!]
    \centering
    \includegraphics[width=0.625\textwidth,trim={2.cm .cm 2.9cm .75cm},clip]{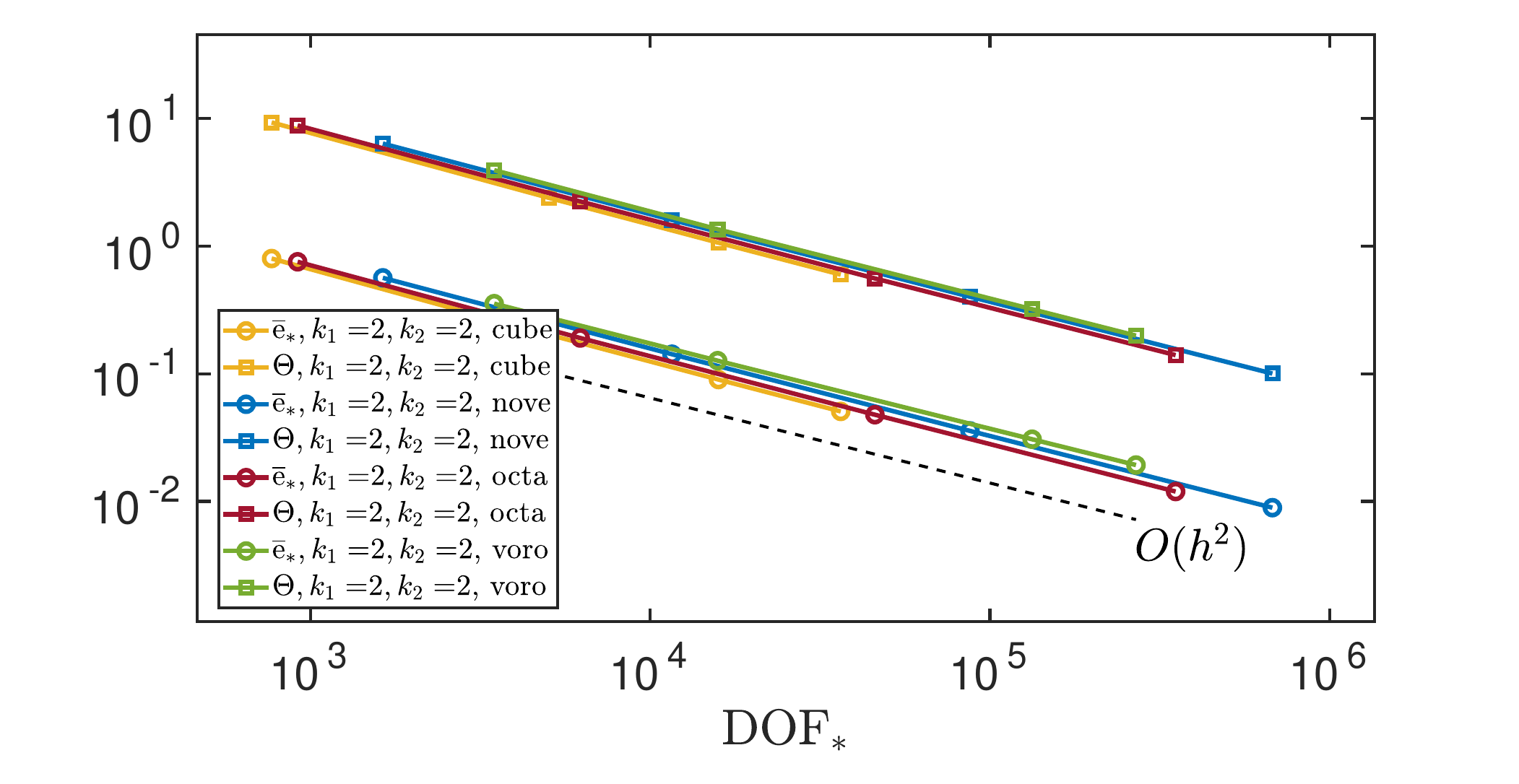}
    \includegraphics[width=0.345\textwidth,trim={0.cm -0.1cm 1.35cm 1.25cm},clip]{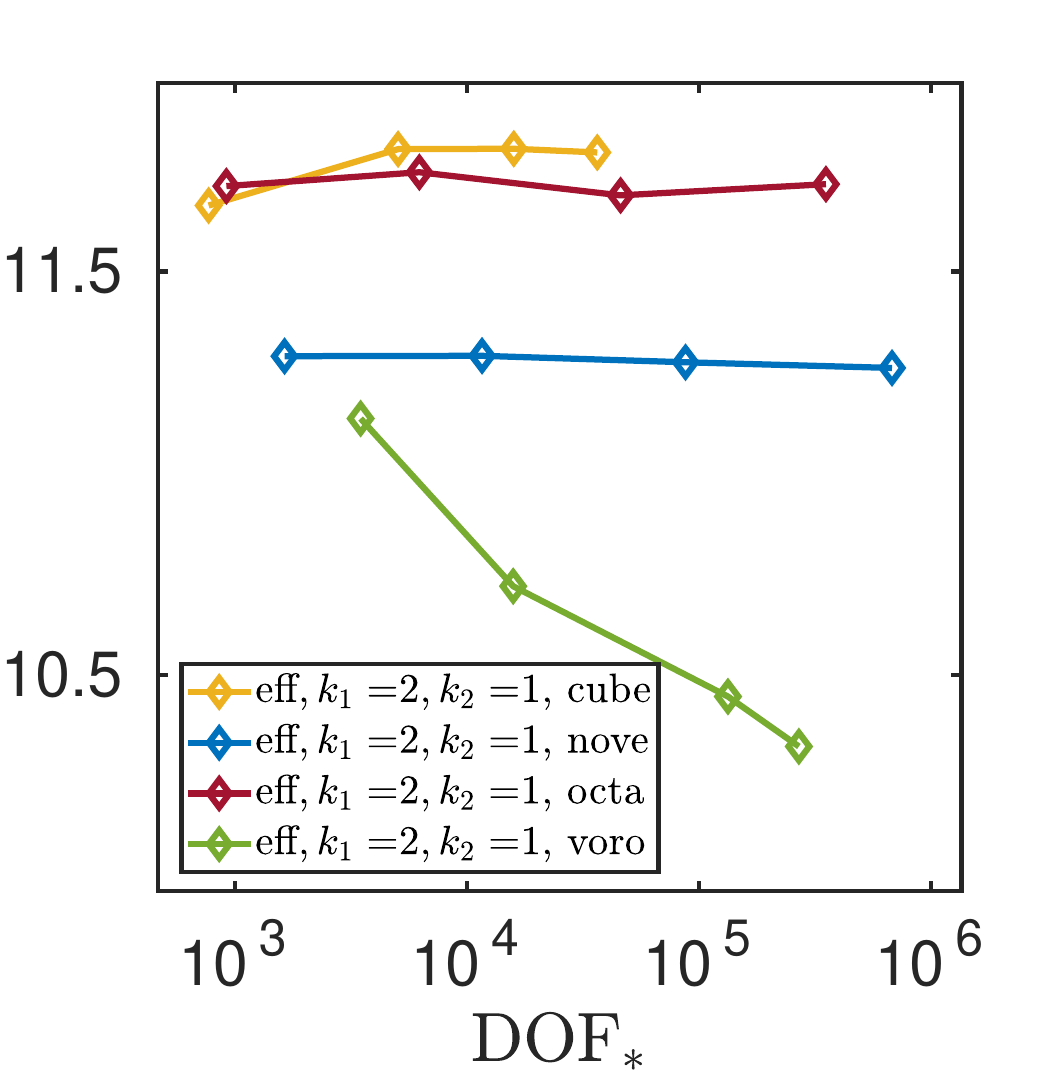}
    
    \caption{Example 3. Behaviour of the error $\overline{\textnormal{e}}_*$, estimator $\Theta$ (left), and effectivity index $\textnormal{eff}$ (right) under uniform refinement across various meshes with $k_1=2$ and $k_2=2$.\label{fig:convergence-3d-uniform}}
\end{figure}

The results reported in Figure~\ref{fig:convergence-3d-uniform} indicate optimal convergence rates for $k_1=2$ and $k_2=2$, as in \cite{rubiano2025}. In addition, the effectivity index remains bounded, which confirms the reliability and efficiency of the estimator (see Theorem~\ref{th:upper-bound} and Theorem~\ref{th:lower-bound}). 

\subsection{Example 4: Adaptivity in 3D}\label{sec:numerical-example-4} For this example, we consider the unit cube domain $\Omega = (0,1)^3$ with Dirichlet part $\Gamma_{\mathrm{D}} = \{(x,y,z) \in \partial \Omega \colon x=0\; \text{or}\; y = 0\; \text{or}\; z=0\}$, and Neumann part $\Gamma_{\mathrm{N}} = \partial \Omega\setminus \Gamma_{\mathrm{D}}$. The starting mesh is defined by $4^3$ cubes. Here we fix the adimensional parameters as $\mu = 10^2$, $\lambda = 10^3$, $\theta = 10^{-3}$, and $M=100$.
\begin{figure}[ht!]
    \centering
    \includegraphics[width=0.65\textwidth,trim={1.cm 0.cm 1.75cm 0.cm},clip]{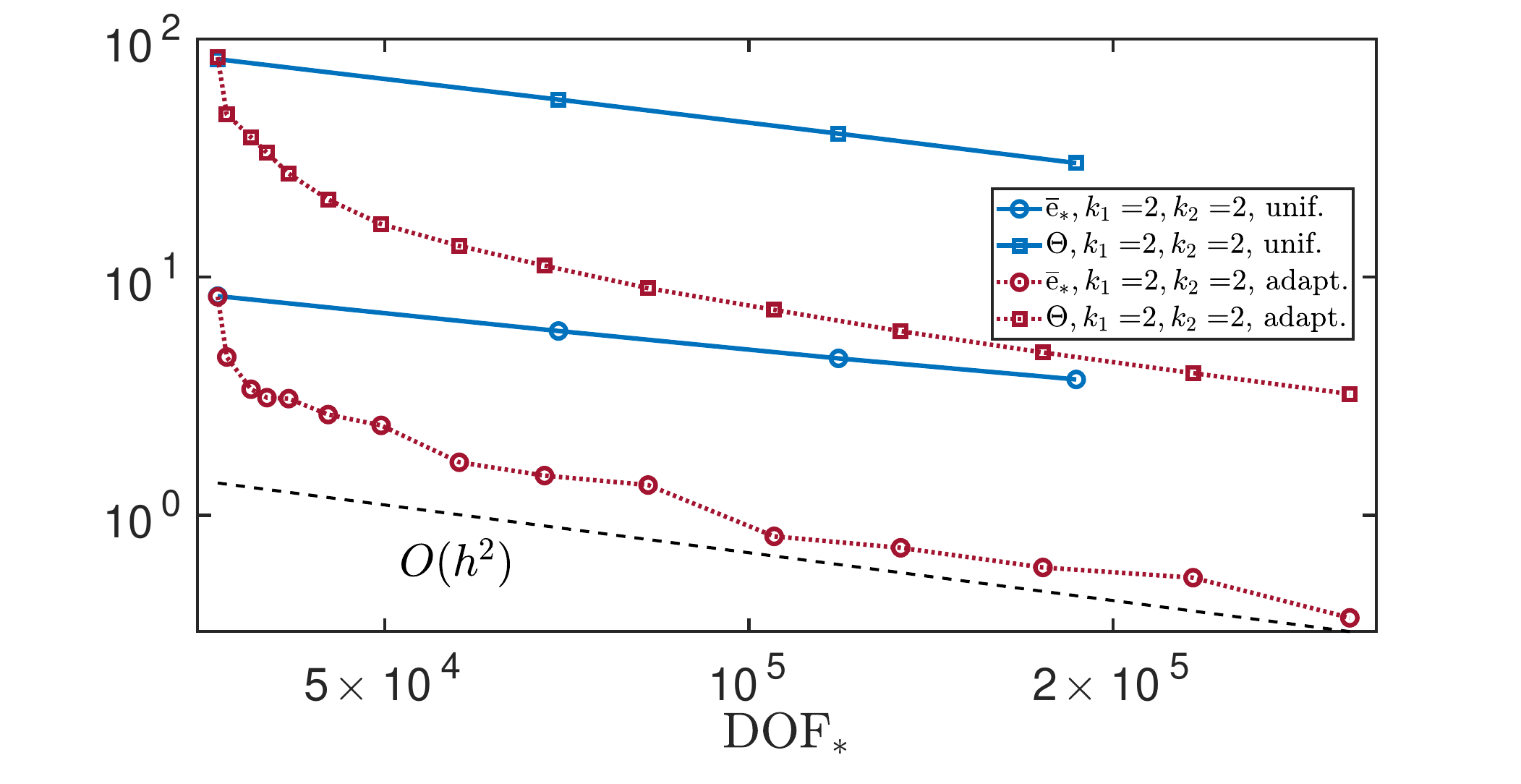}
    \includegraphics[width=0.33\textwidth,trim={.8cm 0.cm 1.25cm 0.cm},clip]{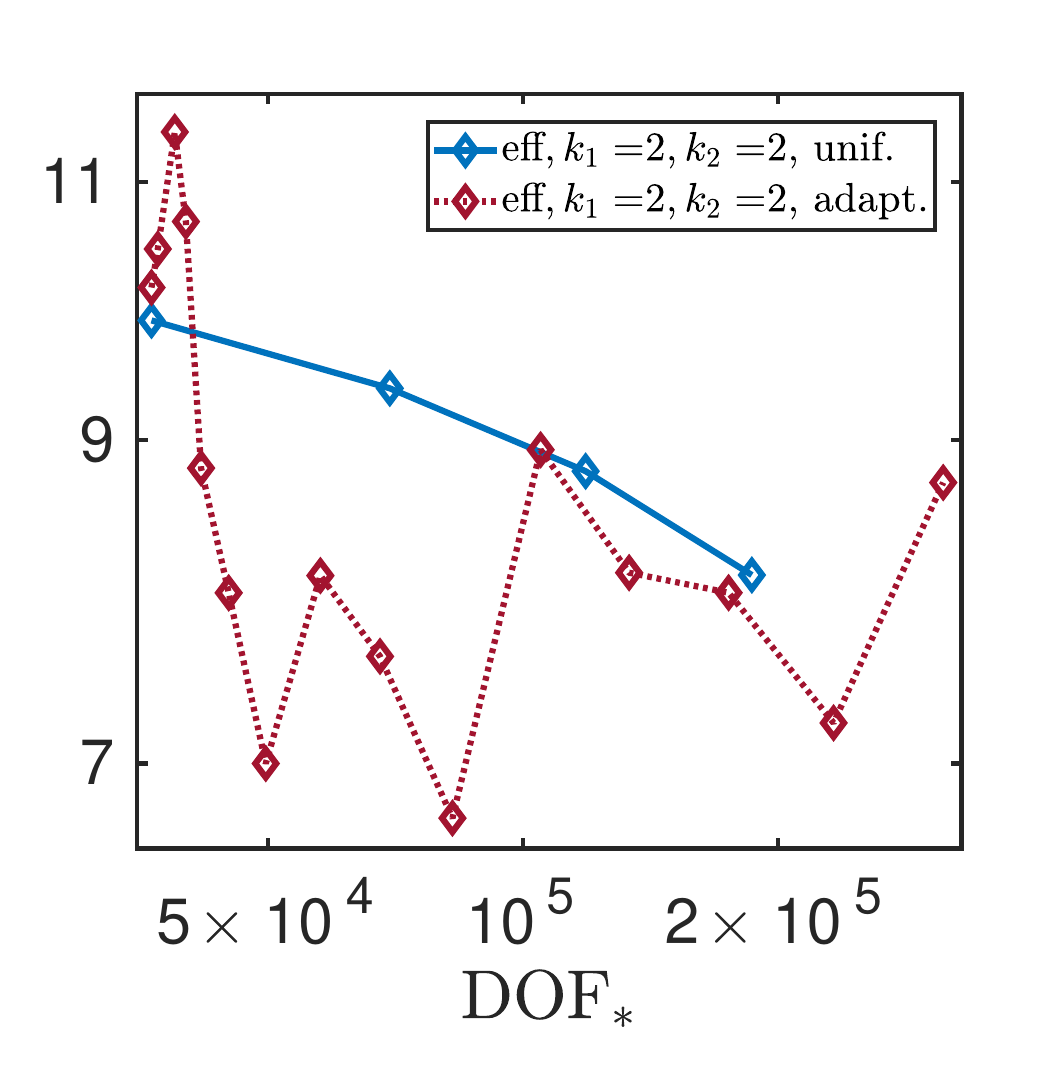}
    
    \caption{Example 4. Behaviour of the error $\overline{\textnormal{e}}_*$, estimator $\Theta$ (left), and effectivity index $\textnormal{eff}$ (right) under adaptive and uniform refinement for the lowest-order case.\label{fig:convergence-3d-adaptive}}
\end{figure}

The non-smooth manufactured solutions and non-linear terms are given as follows
\begin{gather*}
    \bu(x,y,z) = 5^{-1}\left( \left((x+0.1)^2+(y+0.1)^2+(z-1.1)^2\right)^{-1}, z\cos(x)\sin(y), x\cos(z)\sin(y) \right)^{\tt t},\\ 
    \varphi(x,y,z) = xyz\left((x-1.1)^2+(y-1.1)^2\right)^{-1},\\
    \ell(\varphi) = 10^{-2}\varphi, \quad \bbM(\beps(\bu),p) = 10^{-3}\exp{\left[-10^{-4} \tr(2\mu \beps(\bu) - p \bbI)\right]}\bbI.
\end{gather*}
The first displacement component has a singularity close to the point $(x,y,z)^{\tt t} = (0,0,1)^{\tt t}$. In addition, the concentration $\varphi$ has a high gradient  close to the line $(x,y,z)^{\tt t} = (1,1,1-t)^{\tt t}$ for $t\in \mathbb{R}$. The results reported in Figure~\ref{fig:convergence-3d-adaptive} show optimal convergence  s as predicted in \cite{rubiano2025}. Moreover, we observe that the adaptive refinement outperforms the uniform refinement. In addition, the effectivity index remains bounded, confirming the robustness of the estimator proved in Theorems~\ref{th:upper-bound} and \ref{th:lower-bound}. Finally, Figure~\ref{fig:solutions-example-4} shows snapshots of the approximate solutions on a mesh after 15 refinement steps. The method is able to capture the expected solution singularities.
\begin{figure}[ht!]
    \centering
    \includegraphics[width=0.49\textwidth,trim={7.5cm 0.3cm 5.5cm 0.75cm},clip]{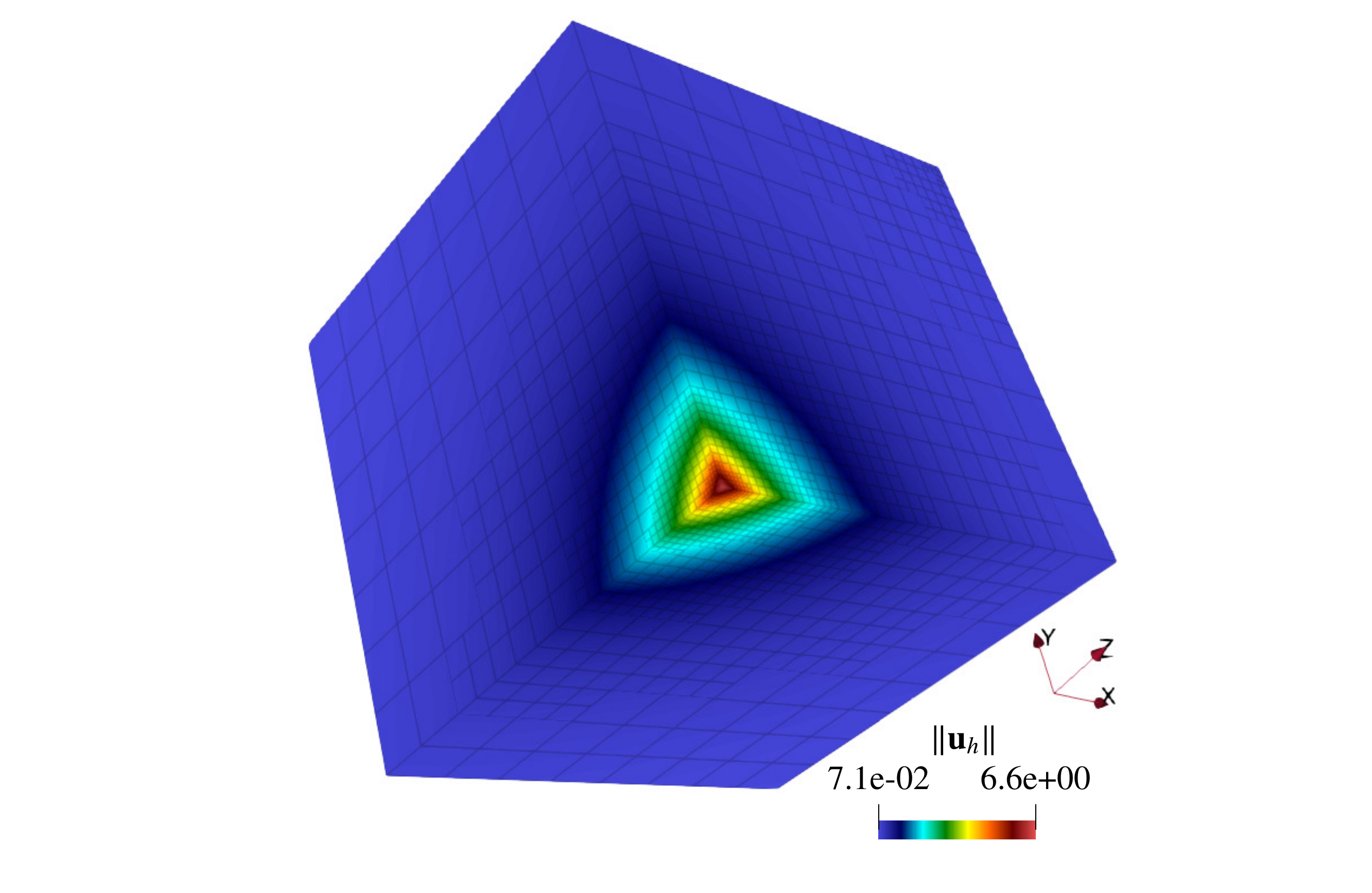}  \label{fig:displacement3d}
    \includegraphics[width=0.49\textwidth,trim={7.5cm 0.3cm 5.5cm 0.75cm},clip]{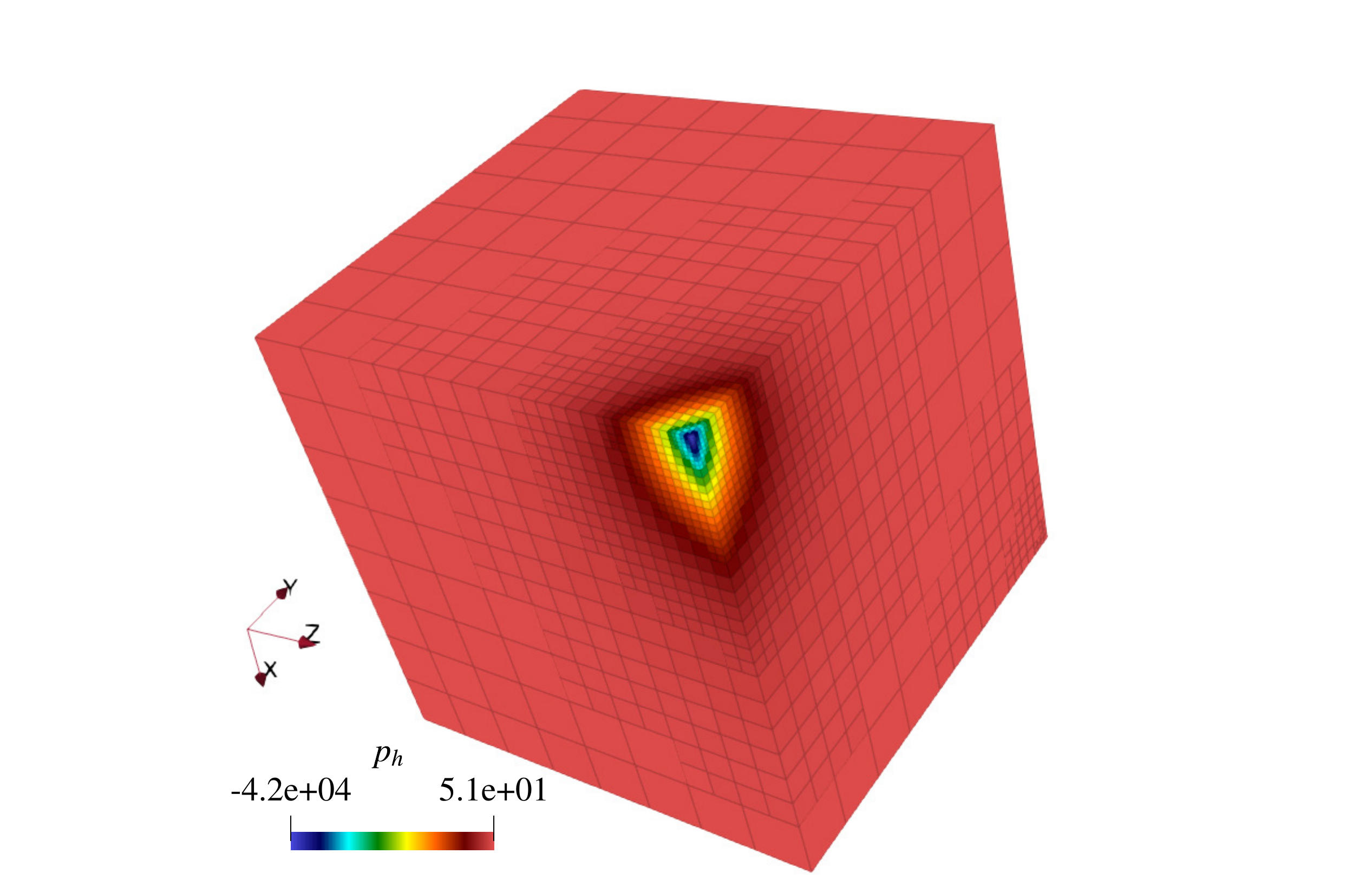} \label{fig:pressure3d}
    \includegraphics[width=0.49\textwidth,trim={7.5cm 0.3cm 5.5cm 0.75cm},clip]{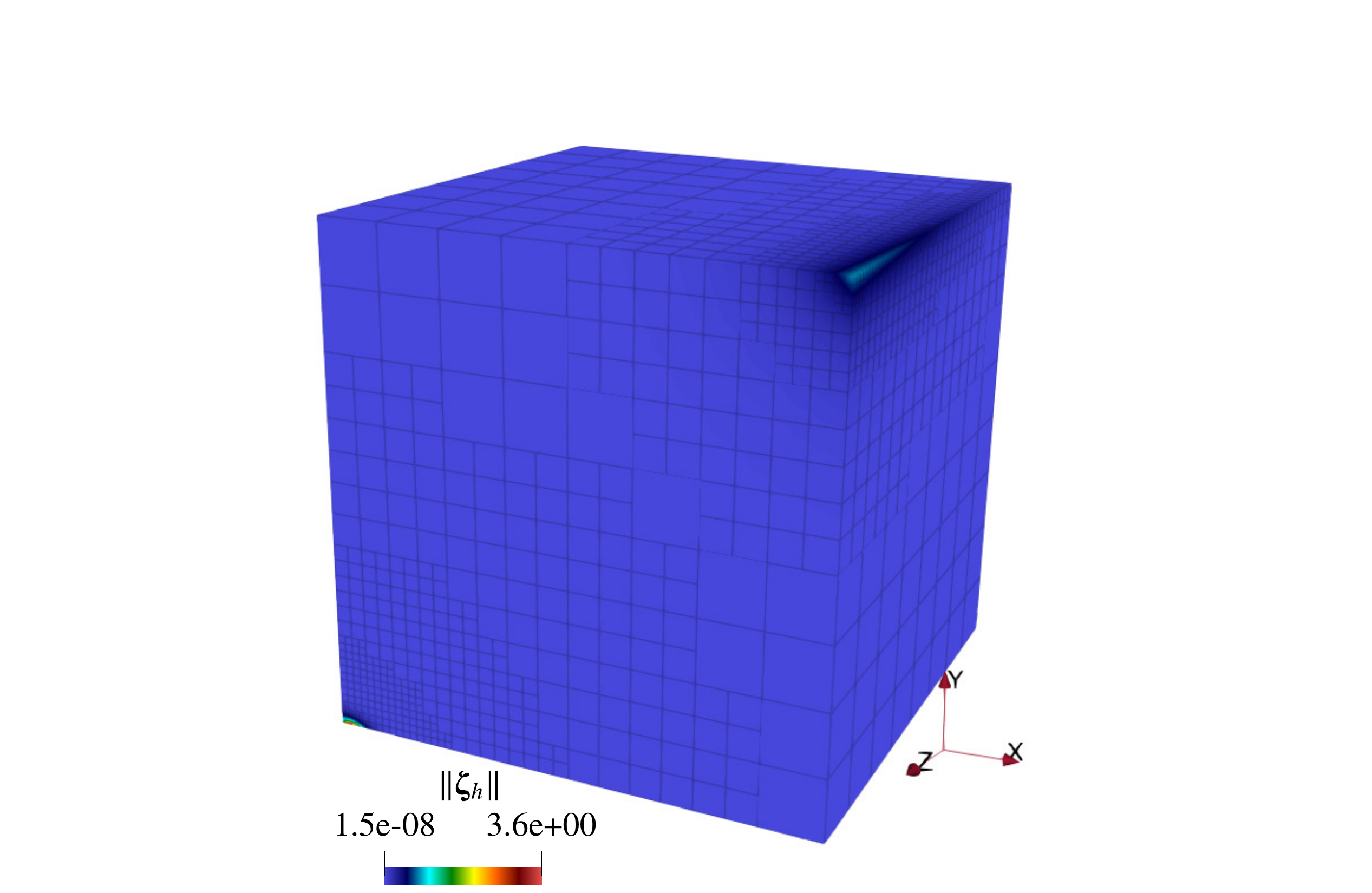} \label{fig:flux3d}
    \includegraphics[width=0.49\textwidth,trim={7.5cm 0.3cm 5.5cm 0.75cm},clip]{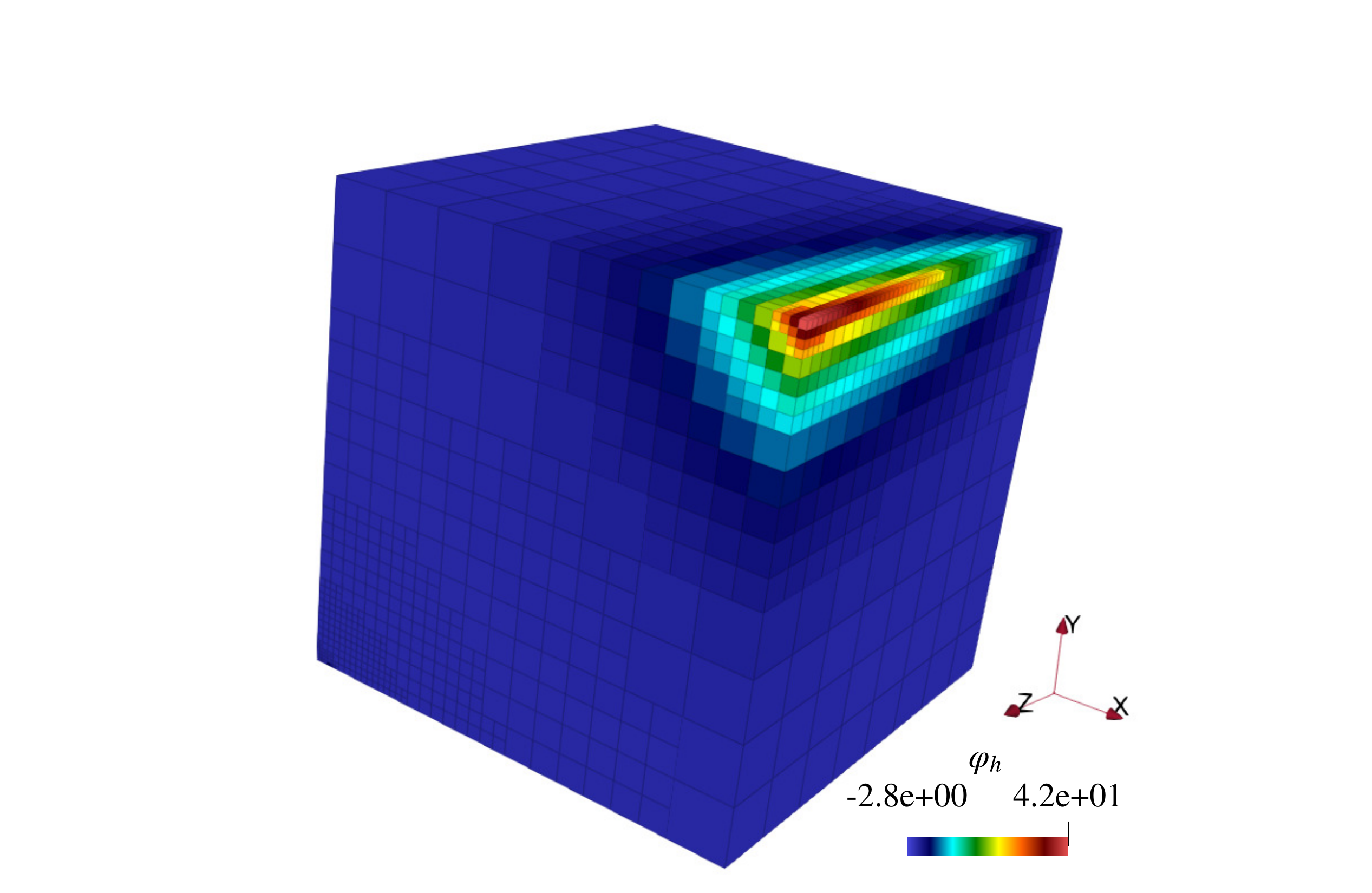} \label{fig:concentration3d}
    \caption{Example 4. Approximate solutions obtained with polynomial degrees $k_1=2$ and $k_2=2$ after 14 adaptive refinements.\label{fig:solutions-example-4}}
\end{figure}
\section*{Acknowledgement} We kindly thank A/Prof. Lorenzo Mascotto for his insight 
on the construction of quasi-interpolators  
for the edge VE spaces in 3D. We also thank Mr. Jordi Manyer 
for an 
implementation of the \texttt{Gridap} interface for \texttt{p4est}  (GridapP4est). 

\section*{Funding} This work has been partially supported by  the Australian Research Council through the \textsc{Future Fellowship} grant FT220100496. 

\bibliography{bibliography}

@article {gatica22,
    AUTHOR = {Gatica, Gabriel N. and Inzunza, Cristian and Sequeira,
              Fil\'{a}nder A.},
     TITLE = {A pseudostress-based mixed-primal finite element method for
              stress-assisted diffusion problems in {B}anach spaces},
journal = {J. Sci. Comput.},
fjournal = {Journal of Scientific Computing},
    VOLUME = {92},
      YEAR = {2022},
    NUMBER = {3},
     PAGES = {Paper No. 103, 43},
      ISSN = {0885-7474},
   MRCLASS = {65N30 (65J05 65N12 65N15 76D07)},
  MRNUMBER = {4460148}
}

@article {ahmad13,
	AUTHOR = {Ahmad, B. and Alsaedi, A. and Brezzi, F. and Marini, L. D. and
	Russo, A.},
	TITLE = {Equivalent projectors for virtual element methods},
journal = {Comput. Math. Appl.},
  fjournal = {Computers \& Mathematics with Applications},
	VOLUME = {66},
	YEAR = {2013},
	NUMBER = {3},
	PAGES = {376--391},
	ISSN = {0898-1221},
	MRCLASS = {65N30 (65N12)},
	MRNUMBER = {3073346},
	MRREVIEWER = {Francesco Calabr\`o}
}

@book{ern21,
  TITLE = {Finite Elements I: Approximation and Interpolation},
  AUTHOR = {Ern, Alexandre and Guermond, Jean-Luc},
  PUBLISHER = {{Springer}},
  ADDRESS = {Switzerland AG},
  YEAR = {2021},
  MONTH = {Feb},
}

@article{gatica2022posteriori,
  title={A posteriori error analysis of mixed finite element methods for stress-assisted diffusion problems},
  author={Gatica, Gabriel N. and G{\'o}mez-Vargas, Bryan and  Ruiz-Baier, Ricardo},
  fjournal={Journal of Computational and Applied Mathematics},
  journal = {{J}. {C}omput. Appl. Math.},
  volume={409},
  pages={114144},
  year={2022},
  publisher={Elsevier}
}

@article{unger1983theory,
  title={On the theory of stress-assisted diffusion, {II}},
  author={Unger, DJ and Aifantis, EC},
  journal={Acta Mech.},
  volume={47},
  pages={117--151},
  year={1983},
  publisher={Springer}
}

@article{mora15,
author = {Mora, David and Rivera, Gonzalo and Rodriguez, Rodolfo},
year = {2015},
month = {02},
pages = {1-25},
title = {A virtual element method for the {S}teklov eigenvalue problem},
volume = {25},
journal = {Math. {M}odels {M}eth. {A}ppl. {S}ci.},
  fjournal = {Mathematical {M}odels and {M}ethods in {A}pplied {S}ciences},
}

@ARTICLE{Loppini2018-vk,
  title     = "Competing mechanisms of stress-assisted diffusivity and
               stretch-activated currents in cardiac electromechanics",
  author    = "Loppini, Alessandro and Gizzi, Alessio and Ruiz-Baier, Ricardo
               and Cherubini, Christian and Fenton, Flavio H and Filippi,
               Simonetta",
  journal   = "Front. Physiol.",
  publisher = "Frontiers Media SA",
  volume    =  9,
  pages     = "1714",
  month     =  dec,
  year      =  2018,
  language  = "en"
}

@article{daveiga15-reaction-diffusion,
      title={Mixed Virtual Element Methods for general second order elliptic problems on polygonal meshes}, 
author = {Beir{\~a}o da Veiga, Louren\c{c}o and Brezzi, Franco and Marini, Luisa Donatella and Russo, Alessandro},
      year={2016},
 journal = {ESAIM: M2AN},
     fjournal = {ESAIM: Mathematical Modelling and Numerical Analysis - Mod\'elisation Math\'ematique et Analyse Num\'erique},
   volume = {50},
number = {3},
pages = {727--747}
}

@Inbook{Shaw2007,
author="Shaw, Derek",
title="Diffusion in Semiconductors",
bookTitle="Springer Handbook of Electronic and Photonic Materials",
year="2007",
publisher="Springer",
address="Boston, MA",
pages="121--135",
isbn="978-0-387-29185-7",
}

@article{PREVOST201183,
title = {Biomechanics of brain tissue},
journal = {Acta Biomater.},
volume = {7},
number = {1},
pages = {83-95},
year = {2011},
issn = {1742-7061},
author = {Thibault P. Prevost and Asha Balakrishnan and Subra Suresh and Simona Socrate},
}

@article{daveiga2022stability,
      title={Stability and interpolation properties for {S}tokes-like virtual element spaces}, 
      author={Meng, J and Beir{\~a}o da Veiga, L and Mascotto, L},
      year={2023},
journal = {J. Sci. Comput.},
fjournal = {Journal of Scientific Computing},
      volume={94},
      pages={e56}
}

@article{beirao2020stokes,
  title={The {S}tokes complex for virtual elements in three dimensions},
  author={Beir{\~a}o da Veiga, L. and Dassi, F. and Vacca, G.},
journal = {Math. {M}odels {M}eth. {A}ppl. {S}ci.},
  fjournal = {Mathematical {M}odels and {M}ethods in {A}pplied {S}ciences},
  volume={30},
  number={03},
  pages={477--512},
  year={2020},
  publisher={World Scientific}
}

@article{antonietti2022,
author = {Antonietti, Paola and Dassi, Franco and Manuzzi, Enrico},
year = {2022},
month = {08},
pages = {111531},
title = {Machine learning based refinement strategies for polyhedral grids with applications to virtual element and polyhedral discontinuous {G}alerkin methods},
volume = {469},
  journal = {{J}. {C}omput. {P}hys.},
  fjournal = {{J}ournal of {C}omputational {P}hysics},
}

@Article{Lederer2019,
author={Lederer, Philip Lukas
and Merdon, Christian
and Sch{\"o}berl, Joachim},
title={Refined a posteriori error estimation for classical and pressure-robust {S}tokes finite element methods},
journal = {Numer. Math.},
  fjournal={Numerische Mathematik},
year={2019},
month={Jul},
day={01},
volume={142},
number={3},
pages={713-748},
issn={0945-3245},
}

@article{khot2024,
author = {Khot, Rekha and Rubiano, Andr\'{e}s E. and Ruiz-Baier, Ricardo},
title = {Robust Virtual Element Methods for Coupled Stress-Assisted Diffusion Problems},
journal = {SIAM J. Sci. Comput.},
fjournal = {SIAM Journal on Scientific Computing},
volume = {47},
number = {1},
pages = {A497-A526},
year = {2025},
}

@article{munar2024,
title = {Residual-based a posteriori error estimation for mixed virtual element methods},
journal = {Comput. Math. Appl.},
  fjournal = {Computers \& Mathematics with Applications},
volume = {166},
pages = {182-197},
year = {2024},
issn = {0898-1221},
author = {Mauricio Munar and Andrea Cangiani and Iván Velásquez},
}

@article{gatica2016,
author = {Gatica, Gabriel and Álvarez, Mario and Ruiz-Baier, Ricardo},
year = {2016},
number = {6},
pages = {1789--1816},
title = {A posteriori error analysis for a viscous flow--transport problem},
volume = {50},
 journal = {ESAIM: M2AN},
     fjournal = {ESAIM: Mathematical Modelling and Numerical Analysis - Mod\'elisation Math\'ematique et Analyse Num\'erique}
}

@article{Cascon07,
author = {Cascón, J. and Nochetto, RICARDO and Siebert, KUNIBERT},
year = {2007},
month = {12},
pages = {1849-1881},
title = {Design and convergence of {AFEM} in {H}(div)},
volume = {20},
journal = {Math. {M}odels {M}eth. {A}ppl. {S}ci.},
  fjournal = {Mathematical {M}odels and {M}ethods in {A}pplied {S}ciences}
}

@book{murray2003mathematical,
  title={Mathematical Biology {II}: Spatial Models and Biomedical Applications},
  author={Murray, James Dickson},
  volume={3},
  address={New York},
  year={2003},
  publisher={Springer}
}

@Article{cherubini17,
   AUTHOR = {Cherubini, Christian and Filippi, Simonetta and Gizzi, Alessio and Ruiz-Baier, Ricardo},
   TITLE = {A note on stress-driven anisotropic diffusion and its role in active
 deformable media},
   YEAR = {2017},
journal ={J. Theoret. Biol.},
  fJOURNAL = {Journal of Theoretical Biology},
  pages = {221--228},
  volume = {430},
  number = {7}
}

@article{boon21,
  title={Robust preconditioners and stability analysis for perturbed saddle-point problems -- Application to conservative discretizations of {B}iot's equations utilizing total pressure},
  author={Boon, Wietse and Kuchta, Miroslav and Mardal, Kent-Andre and Ruiz-Baier, Ricardo},
journal = {SIAM J. Sci. Comput.},  
fjournal={SIAM Journal on Scientific Computing},
  volume={43},
  number={4},
  pages={B961--B983},
  year={2021},
  publisher={SIAM}}

@article{beirao13,
	AUTHOR = {Beir\~{a}o da Veiga, Louren\c{c}o and Brezzi, Franco and Cangiani, A. and
	Manzini, G. and Marini, L. D. and Russo, A.},
	TITLE = {Basic principles of virtual element methods},
journal = {Math. {M}odels {M}eth. {A}ppl. {S}ci.},
  fjournal = {Mathematical {M}odels and {M}ethods in {A}pplied {S}ciences},
	VOLUME = {23},
	YEAR = {2013},
	NUMBER = {1},
	PAGES = {199--214}
}

@article{malaeke2023mathematical,
  title={A mathematical formulation for analysis of diffusion-induced stresses in micropolar elastic solids},
  author={Malaeke, Hasan and Asghari, Mohsen},
  journal = {Arch. Appl. Mech.},
fjournal={Archive of Applied Mechanics},
  volume = {93},
  pages={3093--3111},
  year={2023},
  publisher={Springer}
}

@Article{gatica18,
    AUTHOR = {Gatica, Gabriel N. and G{\'o}mez-Vargas, Bryan and  Ruiz-Baier, Ricardo},
     TITLE = {Analysis and mixed-primal finite element discretisations 
for stress-assisted diffusion problems},
      YEAR = {2018},
 volume = {337},
PAGES = {411--438}, 
journal = {Comput. Methods Appl. Mech. Engrg.},
  fjournal = {Computer Methods in Applied Mechanics and Engineering},
}

@article{lewicka2016local,
  title={A local and global well-posedness results for the general stress-assisted diffusion systems},
  author={Lewicka, Marta and Mucha, Piotr B},
journal ={J. Elast.},
  fjournal={Journal of Elasticity},
  volume={123},
  pages={19--41},
  year={2016},
  publisher={Springer}
}

@Inbook{Brenner1994,
author="Brenner, Susanne C.
and Scott, L. Ridgway",
title="Polynomial Approximation Theory in {S}obolev Spaces",
bookTitle="The Mathematical Theory of Finite Element Methods",
year="1994",
publisher="Springer",
address="New York, NY",
pages="91--122"
}

@article{wilson1982theory,
  title={On the theory of stress-assisted diffusion, {I}},
  author={Wilson, RK and Aifantis, EC},
  journal={Acta Mech.},
  volume={45},
  number={3-4},
  pages={273--296},
  year={1982},
  publisher={Springer}
}

@book{gatica14,
author = {Gatica, Gabriel},
year = {2014},
month = {Jan},
publisher = {{Springer}},
pages = {},
title = {A Simple Introduction to the Mixed Finite Element Method. Theory and Applications},
address = {Cham},
isbn = {978-3-319-03694-6},
}

@article{daveiga15-Stokes,
author = {Beir{\~a}o da Veiga, Louren\c{c}o and Lovadina, Carlo and Vacca, Giuseppe},
year = {2017},
pages = {509--535},
title = {Divergence free Virtual Elements for the {S}tokes problem on polygonal meshes},
volume = {51},
 journal = {ESAIM: M2AN},
     fjournal = {ESAIM: Mathematical Modelling and Numerical Analysis - Mod\'elisation Math\'ematique et Analyse Num\'erique}
}

@article{braess96penalty,
     author = {Braess, Dietrich},
     title = {Stability of saddle point problems with penalty},
     journal = {M2AN - Mod\'elisation Math\'ematique et Analyse Num\'erique},
     pages = {731--742},
     publisher = {AFCET - Gauthier-Villars},
     address = {Paris},
     volume = {30},
     number = {6},
     year = {1996}
}

@phdthesis{Taralov2015,
  author      = {Maxim Taralov},
  title       = {Simulation of Degradation Processes in Lithium-Ion Batteries},
  type        = {{PhD} Thesis},
  pages       = {137 S.},
  school      = {Technische Universit{\"a}t Kaiserslautern},
  year        = {2015},
}

@article{veiga19,
author = {Beir{\~a}o da Veiga, Louren\c{c}o and Mora, David and Vacca, Giuseppe},
year = {2019},
month = {11},
pages = {},
title = {The {S}tokes Complex for Virtual Elements with Application to {Navier–{S}tokes} Flows},
volume = {81},
journal = {J. Sci. Comput.},
fjournal = {Journal of Scientific Computing}
}

@article{veiga-Hdiv,
author = {Beir{\~a}o da Veiga, Louren\c{c}o and Brezzi, Franco and Marini, Luisa Donatella and Russo, Alessandro},
year = {2016},
month = {06},
pages = {303–332},
title = {{H}(div) and {H}(curl)-conforming {VEM}},
volume = {133},
journal = {Numer. Math.},
  fjournal={Numerische Mathematik},
}

@article{cangiani2017posteriori,
  title={A posteriori error estimates for the virtual element method},
  author={Cangiani, A. and Georgoulis, E. H. and Pryer, T. and Sutton, O. J.},
  journal={Numer. Math.},
  volume={137},
  pages={857--893},
  year={2017},
  publisher={Springer}
}

@book{di2020hybrid,
title = "The Hybrid High-Order Method for Polytopal Meshes: Design, Analysis, and Applications",
author = "D. A. {Di Pietro} and Jerome Droniou",
year = "2020",
language = "English",
address = "Switzerland AG",
isbn = "9783030372026",
volume = "19",
series = "MS\&A",
publisher = "Springer",
}

@article{duran2012elementary,
  title={An elementary proof of the continuity from {$L^2_0 (\Omega)$} to {$H^1_0(\Omega)^n$}  of {B}ogovskii’s right inverse of the divergence},
  author={Duran, R.},
  journal={Rev. Uni\'on Mat. Argentina},
  year={2012},
  publisher={Uni{\'o}n Matem{\'a}tica Argentina}
}

@article{galvis07,
author = {Galvis, Juan and Sarkis, Marcus},
year = {2007},
month = {01},
pages = {350-384},
title = {Non-matching mortar discretization analysis for the coupling {S}tokes-{D}arcy equations},
volume = {26},
journal = {Elect. Trans. Numer. Anal.},
fjournal = {Electronic Transactions on Numerical Analysis}
}

@article{AINSWORTH1997,
title = {A posteriori error estimation in finite element analysis},
journal = {Comput. Methods Appl. Mech. Engrg.},
  fjournal = {Computer Methods in Applied Mechanics and Engineering},
volume = {142},
number = {1},
pages = {1-88},
year = {1997},
issn = {0045-7825},
author = {Mark Ainsworth and J.Tinsley Oden},
}

@article{MORA2017,
title = {A posteriori error estimates for a Virtual Element Method for the {S}teklov eigenvalue problem},
journal = {Comput. Math. Appl.},
  fjournal = {Computers \& Mathematics with Applications},
volume = {74},
number = {9},
pages = {2172-2190},
year = {2017},
issn = {0898-1221},
author = {Mora, David and Rivera, Gonzalo and Rodriguez, Rodolfo},
}

@article{wang20,
author = {Wang, Gang and Wang, Ying and He, Yinnian},
year = {2020},
month = {08},
pages = {},
title = {A Posteriori Error Estimates for the Virtual Element Method for the {S}tokes Problem},
volume = {84},
journal = {J. Sci. Comput.},
fjournal = {Journal of Scientific Computing}
}

@article{dassi2023vem++,
author={Dassi, Franco},
title={{VEM++}, a {C}++ library to handle and play with the virtual element method},
journal= {Numer. Algorithms},
fjournal={Numerical Algorithms},
year={2025},
month={Mar},
day={29},
issn={1572-9265}
}

@article{dassi2017stab,
title = {High-order Virtual Element Method on polyhedral meshes},
journal = {Comput. Math. Appl.},
  fjournal = {Computers \& Mathematics with Applications},
volume = {74},
number = {5},
pages = {1110-1122},
year = {2017},
issn = {0898-1221},
author = {Beir{\~a}o da Veiga, L. and Dassi, F. and Russo, A.},
}

@article{yu2021implementationpolygonalmeshrefinement,
      title={Implementation of Polygonal Mesh Refinement in {MATLAB}}, 
      author={Yue Yu},
      year={2021},
volume = {2101.03456},
journal = {arXiv preprint},
      archivePrefix={arXiv}
}

@Article{dominguez15,
author = {Domínguez , Carolina and Gatica , Gabriel N. and Meddahi , Salim},
title = {A Posteriori Error Analysis of a Fully-Mixed Finite Element Method for a Two-Dimensional Fluid-Solid Interaction Problem},
  journal = {{J}. {C}omput. Math.},
  fjournal = {{J}ournal of {C}omputational Mathematics},
year = {2015},
volume = {33},
number = {6},
pages = {606--641}
}

@article{Ern2017,
author = {Ern, Alexandre and Guermond, Jean-Luc},
title = {Finite element quasi-interpolation and best approximation},
 journal = {ESAIM: M2AN},
     fjournal = {ESAIM: Mathematical Modelling and Numerical Analysis - Mod\'elisation Math\'ematique et Analyse Num\'erique},
year = 2017,
volume = 51,
number = 4,
pages = "1367-1385",
}

@article{ZAGHDANI2006,
title = {Two new discrete inequalities of {P}oincaré–{F}riedrichs on discontinuous spaces for {M}axwell's equations},
journal = {C. R. Math.},
fjournal = {Comptes Rendus Mathematique},
volume = {342},
number = {1},
pages = {29-32},
year = {2006},
issn = {1631-073X},
author = {Abdelhamid Zaghdani and Christian Daveau},
}

@article{SHEN2014,
title = {On {Lp} estimates in homogenization of elliptic equations of {M}axwell's type},
journal = {Adv. Math.},
fjournal = {Advances in Mathematics},
volume = {252},
pages = {7-21},
year = {2014},
issn = {0001-8708},
author = {Zhongwei Shen and Liang Song},
}

@Inbook{Girault1986,
author="Girault, Vivette
and Raviart, Pierre-Arnaud",
title="Mathematical Foundation of the {S}tokes Problem",
bookTitle="Finite Element Methods for Navier-{S}tokes Equations: Theory and Algorithms",
year="1986",
publisher="Springer",
address="Berlin, Heidelberg",
pages="1--111",
isbn="978-3-642-61623-5",
}

@article{Beirao22,
author = {Beir{\~a}o da Veiga, L. and  Mascotto, L. and Meng, J.},
title = {Interpolation and stability estimates for edge and face virtual elements of general order},
journal = {Math. {M}odels {M}eth. {A}ppl. {S}ci.},
  fjournal = {Mathematical {M}odels and {M}ethods in {A}pplied {S}ciences},
volume = {32},
number = {08},
pages = {1589-1631},
year = {2022}
}

@ARTICLE{p4est,
  author = {Carsten Burstedde and Lucas C. Wilcox and Omar Ghattas},
  title = {{\texttt{p4est}}: Scalable Algorithms for Parallel Adaptive Mesh
           Refinement on Forests of Octrees},
  fjournal = {SIAM Journal on Scientific Computing},
journal = {SIAM J. Sci. Comput.},
  volume = {33},
  number = {3},
  pages = {1103-1133},
  year = {2011},
}

@article{gridap,
  year = {2020},
  publisher = {The Open Journal},
  volume = {5},
  number = {52},
  pages = {2520},
  author = {Santiago Badia and Francesc Verdugo},
  title = {Gridap: An extensible Finite Element toolbox in {J}ulia},
  journal = {J. Open Source Softw.},
  fjournal = {Journal of Open Source Software}
}

@article{rubiano2025,
      title={Robust virtual element methods for {3D} stress-assisted diffusion problems}, 
      author={Andres E. Rubiano},
      year={2025},
journal = {Proc. ANZIAM},
volume = {66},
pages = {C45--C60},
}

\end{document}